\documentclass[11pt,reqno]{amsart}

\usepackage{amsmath,amsfonts,amssymb,mathrsfs,mathtools}
\usepackage{amsthm,bbm,graphicx,booktabs,tabularx,array}
\numberwithin{equation}{section}
\usepackage{enumitem,microtype}
\usepackage{hyperref}
\usepackage{cite}
\newcommand{\bbr}{\mathbb R}
\newcommand{\bbn}{\mathbb N}
\newcommand{\bbz}{\mathbb Z}

\newcommand{\cA}{\mathcal A}
\newcommand{\cB}{\mathcal B}

\newcommand{\cE}{\mathcal E}
\newcommand{\cF}{\mathcal F}

\newcommand{\cM}{\mathcal M}

\newcommand{\dd}{\mathrm d}
\newcommand{\supp}{\operatorname{supp}}
\newcommand{\Id}{\mathrm{Id}}

\newcommand{\norm}[1]{\left\|#1\right\|}

\newtheorem{theorem}{Theorem}[section]
\newtheorem{lemma}{Lemma}[section]
\newtheorem{corollary}{Corollary}[section]
\newtheorem{proposition}{Proposition}[section]
\newtheorem{remark}{Remark}[section]
\newtheorem{definition}{Definition}[section]
\newtheorem{example}{Example}[section]

\newcommand{\cI}{\mathcal{I}}

\allowdisplaybreaks[4]
\hypersetup{hidelinks}
\begin{document}

\title[Quantitative weak flocking for the KCS equation]{Weak flocking for phase-spatially extended kinetic Cucker-Smale equation in a confining force field}

\author[Ha]{Seung-Yeal Ha}
\address[Seung-Yeal Ha]{\newline Department of Mathematical Sciences and Research Institute of Mathematics, \newline
	Seoul National University, Seoul, 08826, Republic of Korea}
\email{syha@snu.ac.kr}

\author[Wang]{Xinyu Wang$^{*}$}
\address[Xinyu Wang]{\newline School of Mathematics, \newline Harbin Institute of Technology, Harbin  150001, People's Republic of China 
	\newline Department of Mathematical Sciences, \newline
	Seoul National University, Seoul, 08826, Republic of Korea
	}
\email{wangxinyumath@hit.edu.cn}
\thanks{\textbf{Acknowledgment.} 
	The work of S.-Y. Ha is supported by National Research Foundation (NRF) grant funded by the Korea government (MIST) (RS-2025-00514472). The work of X. Wang is supported by the Natural Science Foundation of China (grants 123B2003), the China Postdoctoral Science Foundation (grants 2025M774290), Heilongjiang Province Postdoctoral Funding (grants LBH-Z24167), and the Fundamental Research Funds for the Central Universities. $^{*}$ Corresponding author.}

\begin{abstract}
We study the quantitative fast and slow weak flocking of the phase-spatially extended kinetic Cucker-Smale model in confining potential fields. The confining potential is allowed to be nonconvex. In contrast to the phase-spatially confined case, the communication weight may have no positive lower bounds, while spatial and velocity diameters may remain infinite. Moreover, an indefinite Hessian of confining potential prevents the direct use of convexity-based coercive estimates. To overcome these difficulties, we identify three structural assumptions that are sufficient for the noncompact hypocoercive method: quadratic confinement, a globally Lipschitz force, and a strict virial inequality. The admissible class of confining potentials includes genuinely nonconvex radially symmetric ones and localized oscillatory perturbations of the harmonic potential. Weak flocking analysis combines three key ingredients: a microscopic mechanical-energy estimate, a time-varying effective region, and a macroscopic Lyapunov functional. For this, we first establish global existence for initial data with finite second moments. In the exponential mechanical-energy class, we further show the uniqueness and finite-time Osgood-type stability in $1$-Wasserstein distance. For polynomially decaying initial mechanical-energy tails, we derive the optimal algebraic decay exponent of the fluctuation energy throughout the admissible communication-decay regime. The optimality is verified by a symmetric countably infinite particle solution for a fixed nonconvex potential. For exponentially decaying tails, a two-stage localization argument yields
exponential weak flocking on an optimal exponential time scale. These results show that well-posedness and weak flocking persist even for fully noncompact data under genuinely nonconvex confining forces.
\end{abstract}

\subjclass[2020]{35Q83, 35Q70, 92D50, 37N25}
\keywords{Effective region, fully noncompact support, hypocoercivity, kinetic Cucker--Smale equation, nonconvex confinement,  sharp algebraic decay, weak flocking}

\maketitle


\section{Introduction}\label{sec:1}
\setcounter{equation}{0}
Collective dynamics is ubiquitous in biological, physical, and social complex systems \cite{A-B, A-B-F, V-Z, Pe, P-R-K,Wi1}, where simple interactions among a large number of individuals can generate coherent macroscopic patterns, e.g.,  aggregation of bacteria \cite{T-B}, flocking of birds \cite{C-S-2007a}, schooling of fish \cite{T-T}, flashing of fireflies \cite{Bu,  Wi2} and herding of financial assets \cite{A-B-H-K-L,B-H-K-L-L-Y}. A central issue in the mathematical theory of collective dynamics is to figure out the interaction mechanisms and structural conditions that give rise to such emergent behaviors, and to determine the associated convergence rates. In the literature, several mathematical models were proposed to model such collective dynamics. Among them, the Cucker-Smale (in short, CS) model provides a prototype model for the emergence of flocking via nonlocal velocity alignment \cite{C-S-2007a,C-S-2007b,HaTadmor2008,Shvydkoy2021,Tadmor2021}. In this paper, we study the emergent dynamics of the kinetic Cucker--Smale (in short, KCS) model in confining potential  fields, which can be obtained from the CS model in the mean-field limit in a phase-spatially extended setting.

To set up the stage, we begin with the particle CS model. Let $x_i(t),v_i(t)\in\bbr^d$ denote the position and velocity of the $i$-th particle at time $t$. Then the dynamics of mechanical observables $(x_i, v_i)$ is governed by the Cauchy problem for the CS model \cite{C-S-2007a,C-S-2007b}:
\begin{equation}\label{A-1}
\begin{cases}
\displaystyle \dot x_i=v_i,
 \quad t>0,\quad i\in[N]:=\{1,\ldots,N\},\vspace{6pt}\\
\displaystyle \dot v_i=\frac{\kappa}{N}\sum_{j=1}^N
\phi(|x_j-x_i|)(v_j-v_i), \vspace{6pt}\\
\displaystyle (x_i, v_i) \Big|_{t = 0} = (x_i^0, v_i^0),
\end{cases}
\end{equation}
where $\kappa>0$ is the coupling strength, and $\phi$ is a positive, radially symmetric, non-increasing and Lipschitz continuous communication weight function depending on the relative distance between particles. A prototype example for $\phi$ is an algebraically decaying one:
\begin{equation}\label{A-2}
\phi_\beta(r)=(1+r^2)^{-\beta/2},\qquad r\geq0,\quad \beta\geq0.
\end{equation}

Note that the Cauchy problem for the CS model \eqref{A-1}-\eqref{A-2} is globally well-posed by the standard Cauchy-Lipschitz theory, and the spatial decay rate $\beta$ of $\phi_\beta$ does play a key role in unconditional flocking v.s. conditional flocking. More precisely, for $\beta \in [0, 1]$, $\phi$ is not integrable and it is called a `` {\it long-ranged } '' communication weight. In this case, all initial configurations lead to a mono-cluster flocking state asymptotically \cite{HaLiu2009}. In contrast, for $\beta > 1$, $\phi$ is integrable and it is called `` {\it short-ranged} '' communication weight. In this case, multi-cluster flocking may emerge depending on the initial configurations \cite{bi-cluster}. In the last two decades, the emergent dynamics of the Cauchy problem \eqref{A-1} has been extensively studied from various points of view, e.g., mean-field limit, stability of multi-cluster flocking configurations, collision avoidance, stochastic perturbations, infinite-particle dynamics, etc \cite{C-C-R-2011,C-F-R-T-2010,ChenYin2019,HaLiu2009,HaTadmor2008,Shvydkoy2021,w1,w2,w3,w4,w5,W-MT-2026,FK2019,FKM2018}.

In real applications, CS particles move collectively in the presence of pairwise attractive  interaction force  fields \cite{ChenYin2023,LiChen2025,ShuTadmor2020,C-C-K-T-2025,w4,CFFS2011}. In this setting, the governing dynamics of CS particles reads as follows.
\begin{equation}\label{A-3}
\begin{cases}
\displaystyle \dot x_i=v_i, \quad t > 0, \quad i  \in [N],\vspace{6pt} \\
\displaystyle \dot v_i
=\frac{\kappa}{N}\sum_{j=1}^N\phi(|x_j-x_i|)(v_j-v_i)
-\frac1N\sum_{j=1}^N\nabla W(x_i-x_j), \vspace{6pt}\\
\displaystyle (x_i, v_i) \Big|_{t = 0} = (x_i^0, v_i^0),
\end{cases}
\end{equation}
where $W$ is an even pairwise interaction potential. In the mean-field regime with $N \gg 1$, the mechanical observables associated with \eqref{A-3} can be effectively approximated by the corresponding observables to the Vlasov-McKean type equation \cite{C-C-R-2011,C-F-R-T-2010,HaLiu2009,k12}: 
\begin{equation}\label{A-4}
	\begin{cases}
\displaystyle \partial_t f+v\cdot\nabla_x f
+\nabla_v\cdot\bigl[(\cA[f]-\nabla W*\rho)f\bigr]=0,
\quad (t,x,v)\in\bbr_+\times\bbr^d\times\bbr^d,\vspace{8pt}\\
\displaystyle  \cA[f]
:=\kappa\int_{\bbr^{2d}}\phi(|x-x_*|)(v_*-v)f(t,x_*,v_*)\dd x_* \dd v_*, \quad  \rho := \int_{\bbr^{d}} f(t, x,v) \dd v, \vspace{8pt}\\
\displaystyle f \Big|_{t = 0} = f_0.
\end{cases}
\end{equation}

Note that the asymptotic state resulting from \eqref{A-4} is stronger than velocity alignment: the alignment force removes velocity fluctuations asymptotically, while the attractive  confining force drives the spatial variance to zero. Before we move on to the asymptotic dynamics, we need to make sure the global well-posedness of \eqref{A-4}. First, we recall the definition of weak solution to \eqref{A-4} below.
\begin{definition}[Weak solution]\label{D1.1}
Let $T \in (0, \infty]$ be an extended real number. A one-particle distribution function $
	f\in L^\infty\bigl([0,T); L^1(\mathbb R^{2d})\bigr)
	$
	is called a weak solution to the Cauchy problem \eqref{A-4} on \([0,T)\) with initial
	datum \(f_0\) satisfying 
	\[
f_0 \in(L^1 \cap L_+^\infty)(\mathbb{R}^{2d}) \quad \mbox{and} \quad  ( |x|^2 + |v|^2)f_0 \in L^1(\mathbb{R}^{2d})
         \]
 if the following relations hold:
	
	\begin{enumerate}
		\item The map $ t\longmapsto f:=f(t,x,v)$  is weakly continuous from \([0,T) \) into $L^1(\mathbb R^{2d})$, and
		\begin{equation*}
			\sup_{0\leq t < T}
			\int_{\mathbb R^{2d}}
			\bigl(|x|^2+|v|^2\bigr)f(t,x,v)\dd x\dd v
			<\infty.
		\end{equation*}
		
		\item For every test function
		$
		\psi\in {\mathcal C}_c^1\bigl([0,T) \times\mathbb R^{2d}\bigr),
		$
		one has
	\begin{align*}
		&\int_{\mathbb R^{2d}}\psi(0,x,v)f_0(x,v)\dd x\dd v\\
		&\qquad\quad
		+\int_0^T\int_{\mathbb R^{2d}}
		\Big[
		\partial_t\psi
		+v\cdot\nabla_x\psi
		+\bigl(\mathcal A[f]-\nabla W*\rho\bigr)
		\cdot\nabla_v\psi
		\Big] f(t,x,v)\dd x\dd v\dd t
		=0.
	\end{align*}
	\end{enumerate}
\end{definition}

For compactly supported initial data (or phase-spatially confined setting), the method of characteristic flow and finite-time Wasserstein stability for \eqref{A-4} were already developed in \cite{P-R-T-2015,C-F-R-T-2010}. In contrast, in a fully noncompact  regime (or phase-spatially extended setting), the standard Dobrushin-type stability is no longer closed, because the spatial Lipschitz constant of the velocity alignment force contains an unbounded velocity factor. Adapting the velocity-truncation argument and compactness argument of \cite{H-W-2026,K-M-T-2015}, we can show that finite second moments yield a global weak solution, while an exponential mechanical-energy tail yields uniqueness and an Osgood-type stability in 1-Wasserstein metric (see Theorem \ref{T3.2}). \newline 

Next, we discuss the large-time behaviors of the KCS model \eqref{A-4}. Most available flocking results with attractive potentials use compactly supported data, convexity, or both. If
\begin{equation*}
{\mathcal D}_x(t):=\sup\bigl\{|x-x_*|:(x,v),(x_*, v_*)\in\supp f\bigr\}<\infty,
\end{equation*}
then the monotonicity of $\phi$ yields 
\[ \phi(|x-x_*|)\geq\phi({\mathcal D}_x(t)) \quad \mbox{on the spatial support of $f$}. \]
This converts the exact energy identity into a coercive estimate. Suppose that the confining potential is uniformly convex with $\lambda > 0$: 
\begin{equation}\label{A-5}
\frac{1}{\lambda} \Id \preceq D^2W(x) \preceq\lambda \Id,
\end{equation}
where $\Id$ is the $d\times d$ identity matrix and for two matrices $A$ and $B$, the relation $A \preceq B$ represents that $B-A$ is a positive semidefinite matrix. Then, the relation \eqref{A-5} provides quadratic confinement, a globally Lipschitz force, and a coercive virial inequality at once. These conditions naturally lead to the strong flocking in terms of spatial and velocity diameters \cite{ChenYin2023,ShuTadmor2020,ShuTadmor2021,LiChen2025}. However, for fully noncompact  data, both diameters may be infinite for every finite time. In this paper, we employ a nonlinear functional ${\mathcal F}$ measuring a kind of flocking of \eqref{A-4}:
\begin{equation}\label{A-6}
\cF(t):=\int_{\bbr^{4d}}
\bigl(|x-x_*|^2+|v-v_*|^2\bigr)
 f(t,x,v)f(t,x_*, v_*)\dd x\dd v\dd x_*\dd v_*.
\end{equation}
The functional ${\mathcal F}$ is a two-point correlation functional for one-particle distribution function $f$. Note that a support-diameter-based argument gives no information on the decay of \eqref{A-6}. 

Next, we recall the concept of weak flocking for \eqref{A-4}  in a fully noncompact  support setting in the following definition: 
\begin{definition}[Weak flocking]\label{D1.2}
Let $f\in L^\infty\bigl( [0, \infty); L^1(\mathbb R^{2d})\bigr)$ be a global weak solution to \eqref{A-4}. Then, the KCS model \eqref{A-4} exhibits a weak flocking if the functional ${\mathcal F}$ decays to zero asymptotically
\[  \lim_{t \to \infty} \cF(t) = 0. \]
\end{definition}
\begin{remark}\label{R1.1}
	Depending on the decay mode of the functional $\cF$, we can classify the weak flocking as slow and fast weak flocking: there exist positive constants $\lambda_1$ and $\tilde{\lambda}_1$ such that 
\[ 
\cF(t) \leq {\mathcal O}(1) \begin{cases}
(1 + t)^{-\lambda_1}, \quad & \mbox{slow weak flocking}, \\
e^{-\tilde{\lambda}_1 t}, \quad & \mbox{fast weak flocking}.
\end{cases}
\]
\end{remark}
\noindent Unlike the well-studied compact support regime \cite{ShuTadmor2020,LiChen2025}, the KCS model \eqref{A-4} in the fully noncompact setting with nonconvex confinement force has not been studied. Realistic particle ensembles typically have decaying tails rather than sharp cutoffs, and noncompact distributions arise in the mean-field limit under suitable moment assumptions and in kinetic equations with diffusion, noise, or thermal fluctuations \cite{S2004,L2016,D-L1989,C-Z-2016}. Even compactly supported data may instantly develop full phase-space support in diffusive models \cite{ks2,CM-2008,D-F-T-2010,HuangZhang2022}. Recent works on the KCS model without confining interaction forces have shown that a time-varying effective region can recover weak flocking from polynomial or exponential tails \cite{HaWangXue2025,H-W-CS-2026,HWGKJK2026,H-W-2025}. The attractive force introduces two new obstacles: kinetic and potential energy are continuously exchanged, so high-order velocity moments are no longer monotone, and a nonconvex potential does not provide pointwise Hessian coercivity. Tail propagation and the geometry of the potential must therefore be handled simultaneously. In this paper, we address the following questions: 
\vspace{0.2cm}

\begin{itemize}
\item (Q1):~Under what conditions in terms of system parameters, confining potential and initial data, does the KCS model \eqref{A-4} exhibit weak flocking in the sense of Definition \ref{D1.2}? 
\vspace{0.1cm}
\item(Q2):~If so, what is the convergence rate of the weak flocking functional?
\end{itemize}
\vspace{0.2cm}
The purpose of this paper is to answer these questions in an affirmative manner. For the first question  (Q1), we consider two classes of fully noncompact initial data depending on the decay of their microscopic mechanical-energy tails. For this, we introduce a mechanical energy density $h_0$ for the initial datum $f_0$:
\begin{equation}\label{A-7}
	h_0(x,v):=\frac{1}{2}|v|^2+(W*\rho_0)(x), \quad  \rho_0(x) := \int_{\bbr^{d}} f_0(x,v) \dd v.
\end{equation}
The first and second terms on the right-hand side of $\eqref{A-7}_1$ correspond to the microscopic kinetic and pairwise confining potential energies, respectively. 
\begin{definition} \label{D1.3}
Let $f_0 \in(L^1 \cap L_+^\infty)(\mathbb{R}^{2d})$ be the initial datum for the Cauchy problem \eqref{A-4}. 
\begin{enumerate}
\item
$f_0$ belongs to a polynomial mechanical-energy tail class of order $q>1$ if the following relation holds:
\begin{equation*}
	\mathcal M_q(f_0)
	:=
	\int_{\bbr^{2d}}
	\bigl(1+h_0(x,v)\bigr)^q f_0(x,v) \dd x\dd v
	<\infty.
\end{equation*}
\item
$f_0$ belongs to an exponential mechanical-energy tail class if there exists $a>0$ such that
\begin{equation*}
	\mathcal M_{e,a}(f_0)
	:=
	\int_{\bbr^{2d}}
	e^{a h_0(x,v)}f_0(x,v)\dd x\dd v
	<\infty.
\end{equation*}
\end{enumerate}
\end{definition}
\noindent The pairwise interaction potential $W$ in \eqref{A-7} is not assumed to be convex as in \cite{ShuTadmor2020,C-C-K-T-2025} (see \eqref{A-5}). Instead, we employ the following structural conditions:~there exist positive constants $\underbar{c}_{W},~\bar{c}_{W}$ and $\nu_W$ such that 
\begin{equation}\label{A-8}
	\underbar{c}_W |z|^2\leq W(z)\leq \bar{c}_W |z|^2,
	\quad
	\|D^2W \|_{L^\infty}<\infty,
	\quad
	z\cdot\nabla W(z)\geq\nu_WW(z).
\end{equation}
These conditions in \eqref{A-8} provide quadratic confinement, global Lipschitz continuity of the confining force, and strict virial coercivity, respectively, while we still allow the Hessian of $W$ to be indefinite.

The main weak flocking results of this paper are twofold. First, we derive estimates on the fast and slow weak flocking depending on initial data and the decay exponent $\beta$ of the communication weight function. For this, we use a generalized energy functional and the method of time-varying effective region to exploit the decay mode of the distribution function at the far field. More precisely, we define a generalized energy functional consisting of three component functionals:~for $t \geq 0$, 
\begin{align}
\begin{aligned} \label{A-9}
& {\mathcal I}_x(t) := \int_{\bbr^{2d}} |x|^2 f(t,x,v)\dd x \dd v, \quad {\mathcal E}_K(t) :=\int_{\bbr^{2d}} |v|^2 f(t,x,v)\dd x \dd v, \\
& {\mathcal E}_{I,e}(t):=\int_{\bbr^{2d}} W(x-x_*)\rho(t,x)\rho(t,x_*)\dd x\dd x_*, \quad \Phi(r):=\kappa\int_0^r s\phi(s)\dd s, \\
& {\mathcal E}_{I,s}(t) :=2\int_{\bbr^{2d}}  x\cdot v f(t,x,v)\dd x\dd v
+\int_{\bbr^{2d}} \Phi(|x-x_*|)\rho(t,x)\rho(t,x_*)\dd x\dd x_*, \\
& {\mathcal E}(t) := {\mathcal E}_K(t) + \omega(t) {\mathcal E}_{I,s}(t) + {\mathcal E}_{I,e}(t).
\end{aligned}
\end{align}
Here, ${\mathcal E}_K,~{\mathcal E}_{I,e},~{\mathcal E}_{I,s},$  and ${\mathcal E}$ correspond to functionals proportional to kinetic, external-interaction, self-interaction, and total energies, respectively. The time-dependent weight $\omega = \omega(t)$ is the gauge function determined by the growth scale of time-varying effective region. The $\Phi$-term cancels the alignment contribution in the derivative of the longitudinal momentum, while the virial condition in $\eqref{A-8}_3$ controls the potential energy. 

With these energy functionals, we provide a quantitative answer to (Q2): we establish the optimal polynomial decay rate for polynomial mechanical-energy tails  and exponential weak flocking with an optimal exponential time scale for exponential mechanical-energy tails. More precisely, let $f $ be a global weak solution to \eqref{A-4} satisfying the following polynomial decay conditions:
\[ \beta \in (0, 2], \quad \mathcal M_q(f_0) < \infty, \quad \mbox{for some $q > 1$}. \]
Then, we can derive the equivalence relation between the flocking functional $\cF$ and the total energy functional $\cE$, and a Gr\"onwall-type differential inequality for $\cE$:
\begin{equation} \label{A-10}
\cF(t) \asymp \cE(t), \quad  \cE'(t)
\leq
- |{\mathcal O}(1)|(1+t)^{-1}\cE(t)
+  |{\mathcal O}(1)| (1 + t)^{-\big( 1 +  \frac{2}{\beta}(q-1)\big)}, \quad t \gg 1.
\end{equation}
The latter differential inequality in $\eqref{A-10}_2$ implies 
\begin{equation}\label{A-11}
	\cE(t)
	\lesssim
	(1+t)^{-\frac{2}{\beta}(q-1)}, \quad t \gg 1.
\end{equation}
We combine the first equivalence relation in $\eqref{A-10}_1$ and \eqref{A-11} to obtain the same decay for $\cF$:
\begin{equation}\label{A-12}
	\cF(t)
	\lesssim
	(1+t)^{-\frac{2}{\beta}(q-1)}, \quad t \gg 1.
\end{equation}
The endpoint case $\beta=0$ can be treated by a simpler argument, leading directly to exponential weak flocking. For details, we refer to Theorem \ref{T3.3}, Section \ref{sec:4.1}, and Section \ref{sec:4.2}. \newline 

For exponentially decaying mechanical-energy tails, the argument requires a different mechanism.  A preliminary localization first yields an integrable decay of the macroscopic energy, which in turn gives a uniformly bounded characteristic energy shift (see Proposition \ref{NP4.3} and Lemma \ref{NL4.13}).  A high-energy descent argument then generates a fixed effective region with a uniform positive communication strength, leading to exponential weak flocking (see Lemma \ref{NL4.15} and Proposition \ref{NP4.5}). Together with the universal exponential lower bound, this shows that the exponential time scale is optimal. More precisely, suppose 
\[0<\beta\leq2 \quad  \text{and}\quad  \cM_{e,a}(f_0)<\infty \quad \text{for some} \quad a>0,\] 
then there exist some positive constants $C,c>0$ such that
\begin{equation*}
	\frac{\min \Big \{2, \frac{1}{\bar{c}_W}  \Big \}\cF(0)}{\max \Big \{2, \frac{1}{\underbar{c}_W}  \Big \}} e^{-2\kappa\bar c_\phi t} \leq \cF(t)\leq Ce^{-ct},
	\qquad t\geq0.
\end{equation*}
Here, $\bar c_\phi $ denotes the upper bound of the communication weight defined in \eqref{C-2}. For details, we refer to Theorem \ref{T3.3} and Section \ref{sec:4.3}. \newline 

Second, we show that the polynomial decay rate in \eqref{A-12} is optimal when $d\ge2$. For this, we explicitly construct a radially symmetric confining potential $W_{\rm rs}$ satisfying a suitable set of conditions in \eqref{A-8}. Based on this potential, for each fixed $q>1$, we then construct a countably infinite particle solution whose initial configuration has a finite $q$-th mechanical-energy moment, while a sparse sequence of high-energy particles still contributes non-negligibly to the fluctuation functional along a suitably chosen sequence of large times. 
Furthermore, we show that for every $\delta>0$,
\[
\limsup_{t\to\infty}(1+t)^{\frac{2}{\beta}(q-1) +\delta}\cF(t)=\infty,
\]
i.e.,~the polynomial decay exponent $\frac{2}{\beta}(q-1)$ in \eqref{A-12} is optimal. We refer to Theorem \ref{T3.4} and  Section \ref{sec:5} for more details. \newline

The main novelty of this paper lies in the development of a hypocoercive framework for quantitative weak flocking in a fully noncompact phase-spatial setting under nonconvex confinement. Unlike compact-support arguments \cite{ShuTadmor2020,LiChen2025,C-C-K-T-2025}, our method does not rely on the uniform positive lower bound of the communication weight and therefore applies to physically relevant particle distributions with genuinely unbounded spatial and velocity supports, including Gaussian and Cauchy distributions. Furthermore, the confining potential plays a crucial role in the quantitative analysis. Its hypocoercive effect allows us to establish optimal convergence rates, which are beyond the reach of the existing non-confined theory \cite{HaWangXue2025,H-W-CS-2026,HWGKJK2026,H-W-2025}. In the polynomially decaying tail regime, the optimal algebraic exponent is determined solely by the decay order of the initial mechanical-energy tail and the far-field decay exponent of the communication weight. In the exponentially decaying tail regime, the same structural framework yields exponential weak flocking with an optimal exponential time scale. See Section \ref{sec:3.4} for more details.\newline 

The rest of this paper is organized as follows. In Section~\ref{sec:2}, we study conservation laws, energy dissipation, and  elementary estimates. In Section \ref{sec:3}, we review previous results on the KCS model with compact support, and then we describe a sufficient framework for weak flocking and discuss two main results on the quantitative weak flocking. In Section~\ref{sec:4}, we derive slow and fast weak flocking estimates.  In Section~\ref{sec:5}, we show that the polynomial decay exponent  is optimal by constructing an explicit potential and corresponding weak solution. Finally, Section~\ref{sec:6} is devoted to a brief summary of main results and some remaining issues for future study. In Appendix \ref{app-A}, we present a standard approximation argument for fully noncompact data. In Appendix \ref{app-B}, we provide proofs of Gr\"onwall-type lemmas.  In Appendix \ref{app-C}, we provide the proof of Osgood-type stability estimate.  In Appendix \ref{app-D}, we provide a proof of Proposition \ref{P5.1}.  In Appendix \ref{app-E}, we provide a proof of Proposition \ref{P5.2}. 

\vspace{.5cm} 

\noindent\textbf{Gallery of Notations:}  Let $\mathcal C(\Omega)$, $\mathcal C_b(\Omega)$, and $\mathcal C_c^1(\Omega)$ denote the spaces of continuous functions, bounded continuous functions, and continuously differentiable functions with compact support on $\Omega$, respectively. The notation $L^\infty_+(\Omega)$ stands for the cone of bounded nonnegative measurable functions on $\Omega$. We denote by $|\cdot|$ the Euclidean norm.
For two nonnegative quantities $F$ and $G$, we write
\[
F\lesssim G
\]
if there exists a positive constant $C$, independent of the variables under consideration, such that $F\leq CG$. We write $F\asymp G$ if both $F\lesssim G$ and $G\lesssim F$ hold. The constant $C$ is a generic one which may change from line to line. When its dependence is relevant, it is indicated by subscripts, such as $C_T$ or $C_{q,a}$.  \newline

Throughout the paper, we use $t$-dependence in $f = f(t,z)$ and $\rho = \rho(t,x)$ as a subscript for notational simplicity as long as there is no confusion:
\[ f_t(z) \equiv f(t,z), \quad \rho_t(x) \equiv \rho(t,x), \quad \dd z = \dd x \dd v. \]

\vspace{0.5cm}

\section{Preliminaries}\label{sec:2}
\setcounter{equation}{0}
In this section, we establish several fundamental estimates, including conservation laws and energy dissipation estimates, which will play a crucial role in the subsequent analysis.

\subsection{Conservation laws}\label{sec:2.1} 
In this subsection, we study conservation laws such as conservation of mass, the first spatial moment and momentum.
\begin{lemma}[Conservation laws]\label{L2.1}
Suppose that initial datum satisfies normalizations:
\begin{equation} \label{B-1}
\int_{\bbr^{2d}} f_0(z)\dd z=1,
\quad
\int_{\mathbb R^{2d}} vf_0(z)\dd z=0,
\quad
\int_{\bbr^{2d}} xf_0(z)\dd z=0,
\end{equation}
and let $f = f_t(z)$ be a global smooth solution to \eqref{A-4} with sufficiently fast decay at $|z| = \infty$. Then, the normalizations in \eqref{B-1} propagate along \eqref{A-4}:
\begin{equation} \label{B-2}
\int_{\bbr^{2d}} f_t(z)\dd z=1,
\quad
\int_{\bbr^{2d}} vf_t(z)\dd z=0,
\quad
\int_{\bbr^{2d}} xf_t(z)\dd z=0, \quad t > 0.
\end{equation}
\end{lemma}
\begin{proof}
\noindent Below, we deal with a global smooth solution with sufficiently fast decay at infinity, but the approximation in Appendix~\ref{app-A} can justify the same identities in a less regular solution class. \newline

\noindent $\bullet$~(Conservation of mass):~We integrate \eqref{A-4} over $\bbr^{2d}$ to get 
\[ \frac{\dd}{\dd t} \int_{\bbr^{2d}} f_t(z) \dd z =  \int_{\bbr^{2d}}  \partial_t  f_t(z) \dd z = 0. \]
\noindent $\bullet$~(Conservation of momentum):~It follows from \eqref{A-4} that 
\[
\partial_t (v f_t) + \nabla_x \cdot ( v \otimes v f_t)
+\nabla_v\cdot\Big [  \Big ( v \otimes \cA[f_t]- v \otimes (\nabla W*\rho_t) \Big )f_t \Big ]= (\cA[f_t]-\nabla W*\rho_t)f_t.
\]
We integrate the above relation over $\bbr^{2d}$ using a suitable decay at $|z| = \infty$ to get
\begin{equation} \label{B-3}
\frac{\dd}{\dd t} \int_{\bbr^{2d}} v f_t(z) \dd z =  \int_{\bbr^{2d}}  \bigl[(\cA[f_t]-\nabla W*\rho_t)f_t \bigr] \dd z=: {\mathcal I}_{11} + {\mathcal I}_{12}.
\end{equation}
In the sequel, we estimate the terms ${\mathcal I}_{1i},~i=1,2$ one by one. \newline

\noindent $\diamond$ Case 1~(Estimate of ${\mathcal I}_{11}$):~We use the antisymmetry of the velocity alignment interaction via $(z, z_*)~~\leftrightarrow~~(z_*, z)$ to find 
\begin{equation} \label{B-4}
{\mathcal I}_{11} = \int_{\bbr^{2d}} \cA[f_t](z)f_t(z)\dd z
=\kappa\int_{\bbr^{4d}} \phi(|x-x_*|)(v_*-v)f_t(z)f_{t}(z_*)\dd z \dd z_*=0.
\end{equation}
\noindent $\diamond$ Case 2~(Estimate of ${\mathcal I}_{12}$):~Since $W$ is even, $\nabla W$ is odd and
\begin{align}
\begin{aligned} \label{B-5}
{\mathcal I}_{12} &= - \int_{\bbr^{2d}}(\nabla W*\rho_t)f_t(z) \dd z = -\int_{\bbr^{d}} (\nabla W*\rho_t)(x) \rho_t(x) \dd x \\
&= -\int_{\bbr^{2d}} \nabla W(x-x_*)\rho_t(x)\rho_t(x_*)\dd x\dd x_*=0.
\end{aligned}
\end{align}
In \eqref{B-3}, we combine \eqref{B-4} and \eqref{B-5} to get the desired estimate:
\begin{equation} \label{B-6}
\frac{\dd}{\dd t} \int_{\bbr^{2d}} v f_t(z) \dd z=0, \quad \mbox{i.e.,} \quad  \int_{\bbr^{2d}} v f_t(z) \dd z =  \int_{\bbr^{2d}} v f_0(z) \dd z = 0.
\end{equation}
\noindent $\bullet$~(Conservation of center of mass): It follows from \eqref{A-4} that 
\[
 \partial_t (xf_t) +\nabla_x \cdot(x \otimes v f_t) 
+\nabla_v\cdot\bigl[ x \otimes (\cA[f_t]-\nabla W*\rho_t)f_t\bigr]=  (v f_t).
\]
We integrate the above relation over $\bbr^{2d}$ using a suitable decay of $f_t$ at $|z| = \infty$, and use \eqref{B-6} to get 
\[
\frac{\dd}{\dd t}\int_{\bbr^{2d}}  x f_t(z)\dd z =  \int_{\bbr^{2d}} v f_t(z)\dd z=0.
\]
This yields the desired estimate. \newline
\end{proof}

\subsection{Energy estimate} \label{sec:2.2}
In this subsection, we study time-evolution of energy functional. Before we move on to the time-evolution of energy functionals in \eqref{A-9}, we discuss relations between functionals \eqref{A-9}. 
\begin{lemma}[Equivalence of functionals]\label{L2.2}
Let $f = f_t(z)$ be a global smooth solution to \eqref{A-4} satisfying the normalizations \eqref{B-2} and suppose that $W$ satisfies \eqref{A-8}. Then, the following assertions hold:~
\begin{enumerate}
\item
The moment of inertia can be rewritten in terms of two-point spatial correlation:
\begin{equation} \label{B-7}
{\mathcal I}_x(t) = \frac{1}{2}\int_{\bbr^{2d}} |x-x_*|^2\rho_t(x)\rho_t(x_*)\dd x\dd x_*.  
\end{equation}
\item
External interaction energy is equivalent to the moment of inertia uniformly in time:
\begin{equation} \label{B-8}
2 \underbar{c} _W {\mathcal I}_x(t)  \leq {\mathcal E}_{I,e}(t) \leq 2 \bar{c}_W  {\mathcal I}_x(t), \quad t \geq 0.
\end{equation}
\item
The kinetic energy can be rewritten as two-point velocity correlation:
\begin{equation} \label{B-9}
{\mathcal E}_K(t) = \frac{1}{2}\int_{\bbr^{4d}} |v-v_*|^2 f_t(z) f_{t}(z_*)\dd z \dd z_*.
\end{equation}
\item
The functional $\cE_K + \cE_{I,e}$ is equivalent to weak flocking functional:
\begin{equation} \label{B-10}
 \min \Big \{2, \frac{1}{\bar{c}_W}  \Big \} (\cE_K(t) + \cE_{I,e}(t) ) \leq {\mathcal F}(t) \leq \max \Big \{2, \frac{1}{\underbar{c}_W}  \Big \} (\cE_K(t) + \cE_{I,e}(t)).
 \end{equation}
\end{enumerate}
\end{lemma}
\begin{proof}
(1)~We use the normalizations: 
\[ \int_{\bbr^d} \rho_t(x) \dd x =1 \quad \mbox{and} \quad \int_{\bbr^d} x\rho_t(x) \dd x=0 \]
to find the desired estimate \eqref{B-7}:
\begin{align*}
\begin{aligned}
& \int_{\bbr^{2d}} |x-x_*|^2 \rho_t(x) \rho_t(x_*) \dd x\dd x_*  \\
& \hspace{1cm} =  2\int_{\bbr^{d}} |x|^2 \rho_t(x) \dd x -2 \Big| \int_{\bbr^{d}} x \rho_t(x) \dd x \Big|^2 
=2\int_{\bbr^d} |x|^2\rho_t(x)\dd x = 2 {\mathcal I}_x(t).
\end{aligned}
\end{align*}
(2)~By $\eqref{A-8}_1$:
\[ \underbar{c}_W |x-x_*|^2 \leq W(x-x_*) \leq \bar{c}_W |x-x_*|^2, \]
we have
\[
\underbar{c}_W \int_{\bbr^{2d}} |x-x_*|^2\rho_t(x)\rho_t(x_*)\dd x\dd x_*
\leq  {\mathcal E}_{I,e}(t)\leq
\bar{c}_W \int_{\mathbb R^{2d}} |x-x_*|^2\rho_t(x)\rho_t(x_*)\dd x\dd x_*.
\]
Combining $\eqref{A-9}_3$ and \eqref{B-7} yields the desired estimates. \newline

\noindent (3)~Again, we use \eqref{B-2} to get 
\begin{align*}
\begin{aligned}
&\int_{\bbr^{4d}} |v-v_*|^2 f_t(z) f_{t}(z_*)\dd z \dd z_* \\
& \hspace{0.8cm} = 2  \Big( \int_{\bbr^{2d}} |v|^2 f_t(z) \dd z \Big) \Big( \int_{\bbr^{2d}} f_t(z) \dd z \Big) - 2 \Big| \int_{\bbr^{2d}} v f_t(z) \dd z \Big |^2 =  2 {\mathcal E}_K(t).
\end{aligned}
\end{align*}
\noindent (4)~We use the normalization condition
\[ \int_{\bbr^{2d}} v f_t(z) \dd z =0, \]
\eqref{B-8}, and \eqref{B-9} to find 
\begin{align}
\begin{aligned} \label{B-11}
{\mathcal F}(t) &= \int_{\bbr^{4d}}
\bigl(|x-x_*|^2+|v-v_*|^2\bigr)
 f_t(z) f_t(z_*)\dd z \dd z_* \\
 &= 2 {\mathcal I}_x(t) + 2 {\mathcal E}_K(t) \leq \frac{1}{\underbar{c}_W} {\mathcal E}_{I,e}(t) + 2 {\mathcal E}_K(t)  \leq \max \Big \{2, \frac{1}{\underbar{c}_W}  \Big \} (\cE_K(t) + \cE_{I,e}(t) ).
\end{aligned}
\end{align}
On the other hand, we again use \eqref{B-8} and \eqref{B-9} to get 
\begin{equation} \label{B-12}
{\mathcal F}(t) =  2 {\mathcal I}_x(t) + 2 {\mathcal E}_K(t) \geq  \frac{1}{\bar{c}_W} {\mathcal E}_{I,e}(t) +  2 \cE_K(t)
\geq \min \Big \{2, \frac{1}{\bar{c}_W}  \Big \} (\cE_K(t) + \cE_{I,e}(t) ).
\end{equation}
Finally, we combine \eqref{B-11} and \eqref{B-12} to get the desired estimates. 
\end{proof}
Recall the energy functionals:
\begin{align}
\begin{aligned} \label{B-13}
& {\mathcal E}_K :=\int_{\bbr^{2d}} |v|^2 f_t(z)\dd z, \quad {\mathcal E}_{I,e} :=\int_{\bbr^{2d}} W(x-x_*)\rho_t(x)\rho_t(x_*)\dd x\dd x_*, \\
& \Phi(r) :=\kappa\int_0^r s\phi(s)\dd s, \quad  \cE_{I,s}(t) :=2\int_{\bbr^{2d}} (x\cdot v) f_t(z)\dd z
+\int_{\bbr^{2d}} \Phi(|x-x_*|)\rho_t(x)\rho_t(x_*)\dd x\dd x_*.
 \end{aligned}
 \end{align}
In the next lemma, we study time-evolution of the above functionals. 
\begin{lemma}[Conservation laws and energy dissipation]\label{L2.3}
Suppose that initial datum $f_0$ satisfies normalizations:
\[
\int_{\bbr^{2d}} f_0(z)\dd z=1,
\quad
\int_{\mathbb R^{2d}} vf_0(z)\dd x \dd v=0,
\quad
\int_{\bbr^{2d}} xf_0(z)\dd x \dd v=0,
\]
and let $f = f_t(z)$ be a global smooth solution to \eqref{A-4} with sufficiently fast decay at $|z| = \infty$. The energy functionals ${\mathcal E}_K, {\mathcal E}_{I,e}$ and ${\mathcal E}_{I,s}$ in \eqref{B-13} satisfy 
\begin{align*} 
\begin{aligned} 
& (i)~\frac{\dd{\mathcal E}_K}{\dd t} =-\kappa  \int_{\bbr^{4d}} \phi(|x-x_*|)|v-v_*|^2 f_t(z)f_{t}(z_*)\dd z \dd z_* -2 \int_{\bbr^{d}} j_t(x) \cdot(\nabla W*\rho_t)(x) \dd x. \\
& (ii)~\frac{\dd{\mathcal E}_{I, e}}{\dd t} = 2 \int_{\bbr^d} j_t(x)\cdot(\nabla W*\rho_t)(x)\dd x. \\
& (iii)~ \frac{\dd  \cE_{I,s}}{\dd t}  = 2 \cE_K(t) - \int_{\bbr^{2d}} (x-x_*)\cdot\nabla W(x-x_*)\rho_t(x)\rho_t(x_*)\dd x\dd x_*.
\end{aligned}
\end{align*}
Here, the momentum density $j_t(x)$ is defined as follows: \[j_t(x):=\int_{\bbr^d}vf_t(x,v)\dd v.\] 
\end{lemma}
\begin{proof}
\noindent (i)~We first study time-evolution of ${\mathcal E}_K$. For this, we multiply \eqref{A-4} by $|v|^2$ to obtain
\begin{align}
\begin{aligned} \label{B-14}
& \partial_t \Big( |v|^2 f_t \Big) + \nabla_x \cdot \Big ( |v|^2v f_t  \Big )
+\nabla_v\cdot \Big [ |v|^2(\cA[f_t]-\nabla W*\rho_t) f_t \Big] \\
& \hspace{4cm} = 2v \cdot  (\cA[f_t]-\nabla W*\rho_t) f_t.
\end{aligned}
\end{align}
We integrate \eqref{B-14} over $\bbr^{2d}$ using a suitable decay of $f_t$ at $|z| = \infty$ to find 
\begin{align}
\begin{aligned} \label{B-15}
\frac{\dd{\mathcal E}_K}{\dd t}  
&=2 \int_{\bbr^{2d}} v\cdot\cA[f_t]f_t(z)\dd z
 -2 \int_{\bbr^{d}} j_t(x)\cdot(\nabla W*\rho_t)(x)\dd x\\
&=-\kappa  \int_{\bbr^{4d}} \phi(|x-x_*|)|v-v_*|^2 f_t(z)f_{t}(z_*)\dd z \dd z_* \\
&\quad-2 \int_{\bbr^{d}} j_t(x) \cdot(\nabla W*\rho_t)(x) \dd x,
\end{aligned}
\end{align}
where the first term in the right-hand side of \eqref{B-15} follows by exchanging transformation $(z, z_*)~\leftrightarrow~(z_*, z)$ and \eqref{A-4}.\newline 

\noindent (ii)~Now, we use the continuity equation:
\begin{equation*}
\partial_t \rho_t+\nabla_x \cdot j_t=0
\end{equation*}
to find the desired estimate:
\begin{align*}
\begin{aligned}
\frac{\dd{\mathcal E}_{I,e}}{\dd t}  &= 2 \int_{\bbr^{2d}} W(x-x_*)\partial_t \rho_t(x)\rho_t(x_*)\dd x\dd x_* \\
& = -2 \int_{\bbr^{2d}} W(x-x_*)  (\nabla_x \cdot j_t ) \rho_t(x_*)\dd x\dd x_* \\
& = 2 \int_{\bbr^{d}} j_t(x)\cdot(\nabla W*\rho_t)(x)\dd x.
\end{aligned}
\end{align*}
\noindent (iii)~We use \eqref{A-4} and integration by parts to get 
\begin{align}
\begin{aligned} \label{B-18}
& \frac{\dd}{\dd t}\left(2\int_{\bbr^{2d}} x\cdot v f_t(z) \dd z \right) \\
& \hspace{1cm} =2 \cE_K(t)+2\int_{\bbr^{2d}} x\cdot\cA[f_t]f_t(z) \dd z -2\int_{\bbr^d} x\cdot(\nabla W*\rho_t)(x)\rho_t(x)\dd x.
\end{aligned}
\end{align}
Again, we use the continuity equation twice and the identity $\Phi'(r)=\kappa r\phi(r)$ to find 
	\begin{align}
	\begin{aligned} \label{B-19}
	& \frac{\dd}{\dd t}\int_{\bbr^{2d}} \Phi(|x-x_*|)\rho_t(x)\rho_t(x_*)\dd x\dd x_*  \\
	& \hspace{1.5cm} =2\int_{\bbr^{2d}}  \nabla_x\Phi(|x-x_*|)\cdot j_t(x)\rho_t(x_*)\dd x\dd x_*\\
	& \hspace{1.5cm}  =\kappa\int_{\bbr^{4d}}  \phi(|x-x_*|)(x-x_*)\cdot(v-v_*)f_t(z)f_{t}(z_*)\dd z \dd z_*\\
	& \hspace{1.5cm}  =-2\int_{\bbr^{2d}}  x\cdot\cA[f_t]f_t(z) \dd z.
	\end{aligned}
	\end{align}
We add \eqref{B-18} and \eqref{B-19} to get 
\begin{equation} \label{B-20}
\frac{\dd}{\dd t}  \cE_{I,s}(t) = 2\cE_K(t) -2\int_{\bbr^d} x\cdot(\nabla W*\rho_t)(x)\rho_t(x)\dd x.
\end{equation}	
On the other hand, since $\nabla W$ is odd, we also have
\begin{equation} \label{B-21}
	2\int_{\bbr^d} x\cdot(\nabla W*\rho_t)(x)\rho_t(x)\dd x
	=\int_{\bbr^{2d}} (x-x_*)\cdot\nabla W(x-x_*)\rho_t(x)\rho_t(x_*)\dd x\dd x_*.
\end{equation}
We combine \eqref{B-20} and \eqref{B-21} to get the desired estimate:
\[
 \frac{\dd}{\dd t}  \cE_{I,s}(t) = 2 \cE_K(t) - \int_{\bbr^{2d}} (x-x_*)\cdot\nabla W(x-x_*)\rho_t(x)\rho_t(x_*)\dd x\dd x_*.
\]
\end{proof}
\begin{remark} \label{R2.1} 
\begin{enumerate}
\item
By adding time-evolution estimates of $\cE_K$ and $\cE_{I,e}$, one has 
\begin{equation} \label{B-21-1}
\frac{\dd}{\dd t} \Big({\mathcal E}_K + {\mathcal E}_{I,e} \Big)  =- \underbrace{ \kappa\int_{\bbr^{4d}}\phi(|x-x_*|)|v-v_*|^2f_t(z)f_{t}(z_*)\dd z \dd z_*}_{=: \Lambda (t)} \leq 0.
\end{equation}
Therefore, the functional ${\mathcal E}_K + {\mathcal E}_{I, e}$ is uniformly bounded:
\begin{equation} \label{B-22}
0 \leq \sup_{0 \leq t < \infty} (\cE_K + \cE_{I, e})(t)\leq  (\cE_K + \cE_{I, e})(0).
\end{equation}
\item
Recall the energy dissipation functional in \eqref{B-21-1}:
\[  \Lambda(t) =\kappa{\int_{\bbr^{4d}}\phi(|x-x_*|)|v-v_*|^2f_t(z)f_{t}(z_*)\dd z \dd z_*. }\]
In the spatially extended setting, the communication weight $\phi$ has a rough global bound:
\[ 0 \leq \phi(r) \leq \bar{c}_{\phi}, \quad \forall~r \geq 0. \]
Thus, the functional $\Lambda$ satisfies 
\begin{equation*}
0 \leq \Lambda(t) \leq 2\kappa\bar{c}_{\phi} {\mathcal E}_K(t),
\end{equation*}
and the energy estimates in Lemma \ref{L2.3} yield
\[ -2 \bar{c}_{\phi} \kappa{\mathcal E}_K(t) \leq \frac{\dd}{\dd t} (\cE_K + \cE_{I, e})(t)  \leq 0, \]
which is not sufficient to derive any meaningful decay estimate of $\cE_K + \cE_{I, e}$. To deal with this issue, we will introduce  the method of effective region and hypocoercivity estimate later. Even for a detailed careful estimate, the dissipation functional $\Lambda$ needs to  be controlled only by $\cE_K$ (see Proposition \ref{P4.2}): For a suitable set of conditions, one can see  that for all $\gamma \geq 1$,
\[  \Lambda(t) \geq  |{\mathcal O}(1)| (1+t)^{-\frac{\gamma\beta}{2}} \cE_K(t) -  |{\mathcal O}(1)| (1 + t)^{-\gamma (  \frac{\beta}{2} + q-1)}, \quad t  \gg 1. \]
That is, $\cE_K + \cE_{I,e}$ does satisfy a differential inequality:
\[
\hspace{1cm} \frac{\dd}{\dd t} (\cE_K(t) + \cE_{I,e}(t)) \leq -|{\mathcal O}(1)| (1+t)^{-\frac{\gamma\beta}{2}} \cE_K(t) + |{\mathcal O}(1)| (1 + t)^{-\gamma (  \frac{\beta}{2} + q-1)}, \quad t  \gg 1,
\]
which is not closed as well so that we cannot derive decay estimate for $\cE_K + \cE_{I,e}$ via Gr\"onwall's lemma as it is. This is why we introduce the functional $\cE_{I,s}$ to derive a closed differential inequality for the linear combination of functionals $\cE_K, \cE_{I,e}$ and $\cE_{I,s}$. 
\vspace{0.2cm}
\item
In Section \ref{sec:4}, we will see the following equivalence relation depending on the decay rate $\beta$ of the communication weight $\phi$:
\[ \cE=  {\mathcal E}_K + \omega(t) {\mathcal E}_{I,s} + {\mathcal E}_{I,e} \quad \asymp \quad  \cF. \]
Then we will derive Gr\"onwall's differential inequality for $\cE$. The above outlined procedures will be treated rigorously  in Section \ref{sec:4}.
\end{enumerate}
\end{remark}

\vspace{0.5cm}

For $q>1$ and $a>0$, we define energy density and weighted moments of $f_0$:
\begin{align}
\begin{aligned} \label{B-24}
& h_0(z) :=\frac12|v|^2+ (W*\rho_0)(x), \quad  h_t(z) :=\frac12|v|^2+ (W*\rho_t)(x), \quad t > 0, \\
& \cM_q(f_0) :=\int_{\bbr^{2d}}(1+ h_0(z))^qf_0(z)\dd z, \quad  \cM_{e,a}(f_0) :=\int_{\bbr^{2d}} e^{ah_0(z)}f_0(z)\dd z.
\end{aligned}
\end{align}
\begin{lemma} \label{L2.4}
Suppose the initial datum satisfies normalization \eqref{B-1} and $W$ satisfies \eqref{A-8}. Then, the following relations hold.
\begin{align*}
\begin{aligned}
& (i)~\int_{\bbr^{d}} |x-x_*|^2\rho_0(x_*)\dd x_* =|x|^2 + {\mathcal I}_x(0). \\
& (ii)~\frac12|v|^2+ \underbar{c}_W \Big (|x|^2 + {\mathcal I}_x(0) \Big)
\leq h_0(z)
\leq \frac12|v|^2+ \bar{c}_W \Big (|x|^2+ {\mathcal I}_x(0) \Big).
\end{aligned}
\end{align*}
\end{lemma}
\begin{proof}
\noindent (i)~By \eqref{B-1}, one has 
\begin{align}
\begin{aligned} \label{B-25}
&\int_{\bbr^d} |x-x_*|^2\rho_0(x_*)\dd x_*  \\
& \hspace{1cm} = \int_{\bbr^d} ( |x|^2 + |x_*|^2 - 2 x \cdot x_* )\rho_0(x_*)\dd x_* \\
& \hspace{1cm} = |x|^2+ {\mathcal I}_x(0) - 2 x \cdot \Big( \int_{\bbr^d}  x_* \rho_0(x_*)\dd x_* \Big) = |x|^2+ {\mathcal I}_x(0).
\end{aligned}
\end{align}
\noindent (ii)~It follows from $\eqref{A-8}_1$:
\begin{equation} \label{B-26}  
\underbar{c}_W  \int_{\bbr^{d}} |x - x_*|^2 \rho_0(x_*) \dd x_* \leq  (W*\rho_0)(x)  \leq  \bar{c}_W \int_{\bbr^{d}} |x - x_*|^2 \rho_0(x_*) \dd x_*.  
\end{equation}
We combine \eqref{B-25} and \eqref{B-26} to get 
\[
\frac12|v|^2+ \underbar{c}_W (|x|^2+ {\mathcal I}_x(0))
\leq h_0(z)
\leq \frac12|v|^2+ \bar{c}_W (|x|^2+ {\mathcal I}_x(0)).
\]
\end{proof}

\subsection{Gr\"onwall-type lemmas} \label{sec:2.3}
In this subsection, we provide three Gr\"onwall-type lemmas whose proofs can be found in Appendix~\ref{app-B}.
\begin{lemma}  \label{L2.5}
Let \(y :[t_0,\infty)\to \bbr_+\) be a locally absolutely continuous function such that 
\[
y'(t) + c(1+t)^{-\alpha} y(t) \leq C(1+t)^{-\alpha-\lambda}, \quad  \mbox{a.e.}~t > t_0 \gg 1,
\]
where parameters are constants satisfying a set of relations:
\[  c > 0, \quad  C > 0,  \quad 0 \leq \alpha < 1, \quad \lambda > 0. \]
Then, $y$ decays at least algebraically:~there exists a positive constant ${\tilde C} = {\tilde C}(c, C,t_0,y(t_0), \lambda, \alpha)$ depending on parameters such that 
\[ y(t)\leq {\tilde C} (1+t)^{-\lambda}, \quad t \gg t_0. \]
\end{lemma}
\begin{proof}
We leave the proof in Appendix \ref{app-B-1}.
\end{proof}
\begin{lemma}  \label{L2.6}
Let \(y :[t_0,\infty)\to \bbr_+\) be a locally absolutely continuous function such that
\[
y'(t) + c_0(1+t)^{-\alpha} y(t) \leq  C_0(1+t)^{-\alpha}e^{-\sigma(1+t)^\lambda}, \quad t > t_0,
\]
where  parameters are positive constants satisfying a set of relations:
\[  c_0 > 0, \quad  C_0 > 0,  \quad \sigma > 0,  \quad 0 \leq \alpha < 1, \quad \lambda > 0. \]
Then, we have
\begin{equation*}
y(t)\leq {\tilde C}_0\exp\left[-{\tilde c}_0(1+t)^{\min\{1-\alpha,\lambda \}}\right], \quad t \gg t_0,
\end{equation*}
where ${\tilde c}_0$ and ${\tilde C}_0$ are positive constants depending on parameters. 
\end{lemma}
\begin{proof}
We leave the proof in Appendix \ref{app-B-2}.
\end{proof}
\begin{lemma}  \label{L2.7}
Let \(y :[t_0,\infty)\to \bbr_+\) be a locally absolutely continuous function such that 
\begin{equation*}
y'(t) + \frac{c_1}{{1+t}} y(t) \leq \frac{C_1}{1+t} g(t), \quad t > t_0,
\end{equation*}
where $c_1, C_1>0$. Then, the following assertions hold.
\begin{enumerate}
\item
If $g(t)\leq C_2(1+t)^{-\lambda}$, then
\[
y(t)\leq
\begin{cases}
C(1+t)^{-\min\{c_1, \lambda \}}, \quad &c_1 \neq  \lambda,\\
C(1+t)^{-\lambda}\log(2+t), \quad  &c_1= \lambda.
\end{cases}
\]
\item
If $g(t)\leq C_2e^{-ct^\lambda}$, then 
\[ y(t)\leq C_\eta(1+t)^{-\eta}, \quad \mbox{for every $0<\eta< c_1$}. \]
\end{enumerate}
\end{lemma}
\begin{proof}
We leave the proof in Appendix \ref{app-B-3}.
\end{proof}

\vspace{0.5cm}

\section{Description of main results}\label{sec:3}
\setcounter{equation}{0}
In this section, we recall previous results on strong and weak flocking estimates in a phase-spatially confined setting, and then we provide the well-posedness theory of system \eqref{A-4} in a phase-spatially extended setting. Finally, we present the main results on the quantitative weak flocking estimates.

\subsection{A framework for weak flocking}\label{sec:3.1}
In this subsection, we delineate a sufficient framework for weak flocking $({\mathcal F}_A)$ in terms of confining potential, communication weight function, coupling strength and initial data:
\vspace{0.1cm}
\begin{itemize}
\item
 $({\mathcal F}_A1)$ (Confining potential): ~The potential $W$ is twice continuously differentiable, and it satisfies  a set of conditions: there exist positive constants
 $\underbar{c}_W, \bar{c}_W, \nu_W$ such that for $ x \in \bbr^d$, 
 \begin{align}
 \begin{aligned} \label{C-1}
&  W(-x) = W(x), \quad \underbar{c}_W |x|^2\leq W(x)\leq \bar{c}_W |x|^2, \\
& x \cdot\nabla W(x)\geq \nu_WW(x), \quad \|D^2W \|_{L^{\infty}} <\infty.
\end{aligned}
\end{align}
\item
 $({\mathcal F}_A2)$ (Communication weight):~The communication weight function $\phi\in \mathcal C_b^1(\mathbb{R}_{+})$ is nonnegative, and there exist constants $0 < {\underbar c}_\phi \leq \bar{c}_\phi <\infty$ such that
\begin{equation}\label{C-2}
 \frac{\underbar{c}_\phi}{(1+r^2)^{\beta/2}} \leq \phi(r)\leq \frac{\bar{c}_\phi}{(1+r^2)^{\beta/2}},
 \quad r\geq 0.
\end{equation}
\item
 $({\mathcal F}_A3)$ (Coupling strength):~The coupling strength $\kappa$ is sufficiently large such that 
 \[ \kappa > \frac{1}{\underbar{c}_{\phi}}.  \] 
\item
 $({\mathcal F}_A4)$ (Initial datum):~The initial datum $f_0  \in (L^1 \cap L_+^{\infty})(\bbr^{2d})$ satisfies
\begin{equation}  \label{C-3}
\int_{\bbr^{2d}} f_0(z) \dd z = 1,~~\int_{\bbr^{2d}} xf_0(z)\dd z=0,~~\int_{\bbr^{2d}}  vf_0(z)\dd z=0,~~\int_{\bbr^{2d}} |z|^2 f_0(z)\dd z <\infty.\end{equation}
\end{itemize}
\begin{remark} \label{R3.1}
In what follows, we provide several comments on $({\mathcal F}_A)$.
\begin{enumerate}
		\item
		Since $\nabla W(-x) = - \nabla W(x)$, we have $\nabla W(0) = 0$, and  note that we do not impose a sign condition for the Hessian $D^2 W$, except a bounded condition. 
		\vspace{0.1cm}
\item
		In the fully noncompact setting, the far-field decay of the communication weight $\phi$ plays an essential role. Since the spatial diameter can be infinite, no uniform positive lower bound for $\phi(|x-y|)$ is available on the whole spatial support, and the decay of $\phi$ at large distances directly determines the strength of the localized alignment dissipation. This is fundamentally different from the compact-support setting. Indeed, if the spatial support remains uniformly bounded, then we have
		\[
		\phi(|x-y|)
		\geq
		\phi({\mathcal D}^{\infty}_x)
		>0, \quad \mbox{on the support},
		\]
		 where ${\mathcal D}^{\infty}_x$ is a uniform bound for the spatial diameter of the projected spatial support of $f$. In that case, only the behavior of $\phi$ on a bounded interval is relevant, whereas its precise far-field decay is immaterial.
		 \vspace{0.1cm}
		\item
		The zero-momentum and zero-center of mass assumptions are only the choice of inertial frame. The set of conditions \eqref{C-3}
	provides a basic integrability framework for constructing global weak solutions to \eqref{A-4}. In particular, finite mass and finite second moments are natural requirements for controlling the kinetic and interaction energies, while the nonnegativity of the initial datum is preserved along \eqref{A-4}. These assumptions are consistent with the weak solution framework developed for kinetic flocking equations in \cite{K-M-T-2015}. 
	\end{enumerate}
\end{remark}	
\begin{example}\label{Ex3.1}
(One-parameter family of confining potentials satisfying $({\mathcal F}_A1)$): 
The conditions in \eqref{C-1} are satisfied by uniformly convex potentials with quadratic growth and bounded Hessian, as in related confinement models \cite{ShuTadmor2020,C-C-K-T-2025}. Note that uniform convexity is not required in the present work. For $a \in (1,2),$ we define a one-parameter family of confining potentials:
		\begin{equation} \label{C-4}
		W_a(z) = {\widetilde W}_a(r) := \frac12 r^2+a\bigl(1-\cos r \bigr),
		\quad r = |z|, \quad  z\in\bbr^d.
		\end{equation}
Note that $W_a$ is genuinely nonconvex, and  it is easy to see that 
\begin{align}
\begin{aligned}   \label{C-5}
& \widetilde{W}_a^{\prime}(r) = r + a \sin r, \quad \widetilde{W}_a^{\prime \prime}(r) =  1 + a \cos r, \quad  \partial_{x_i} W_a(x) = \widetilde{W}_a^{\prime}(r) \frac{x_i}{r}, \\
& \partial_{x_j} \partial_{x_i} W_a(x) =  \widetilde{W}_a^{\prime \prime}(r) \frac{x_i x_j}{r^2}
 +  \widetilde{W}_a^{\prime}(r) \Big(  \frac{\delta_{ij}}{r} - \frac{x_i x_j}{r^3}  \Big).
\end{aligned}
\end{align}	
Now, we check that $W_a$ satisfies four conditions in \eqref{C-1}:
\begin{align} 
\begin{aligned} \label{C-6}
& (i)~  W_a(-x) = W_a(x), \quad   \frac{1}{2} |x|^2
		\leq W_a(x)
		\leq \frac{1+a}{2} |x|^2, \\
& (ii)~   x \cdot\nabla W_a(x)\geq\frac{2-a}{1+a}W_a(x), \quad \| \partial_{x_j} \partial_{x_i} W_a \|_{L^{\infty}} < \infty.
\end{aligned}
\end{align}
\begin{enumerate}
\item
 (Verification of (i)):~We use the defining relation of $W_a$ in \eqref{C-4} and $ 0 \leq 1-\cos r\leq \frac{r^2}{2},~r > 0$ to get the desired estimate.
 \vspace{0.1cm}
\item
 (Verification of (ii)):~We use \eqref{C-5}, \eqref{C-6}  and $ \sin r\geq -\frac{r}{2},~~ r \geq 0$ to find 
\[
 x \cdot \nabla W_a(x) = \sum_{i=1}^{d} x_i \partial_{x_i} W_a(x) =   r\widetilde W_a'(r)
		= r^2+ar\sin r \geq 
		\left(1-\frac{a}{2} \right)r^2 \geq
		\frac{2-a}{1+a}W_a(x).
\]
On the other hand, it follows from \eqref{C-5} that 
\begin{equation} \label{C-7}
| \partial_{x_j} \partial_{x_i} W_a(x) | \leq \Big |\widetilde{W}_a^{\prime \prime}(r) \frac{x_i x_j}{r^2} \Big |
 +  \Big|  \widetilde{W}_a^{\prime}(r) \Big(  \frac{\delta_{ij}}{r} - \frac{x_i x_j}{r^3}  \Big) \Big|.
\end{equation}
Note that the first term in \eqref{C-7} can be estimated as follows:
\begin{equation} \label{C-8}
\Big |\widetilde{W}_a^{\prime \prime}(r) \frac{x_i x_j}{r^2} \Big | \leq | \widetilde{W}_a^{\prime \prime}(r) | \leq 1 + a,
\end{equation}
and the second term is also bounded:
\begin{equation} \label{C-9}
 \Big|  \widetilde{W}_a^{\prime}(r) \Big(  \frac{\delta_{ij}}{r} - \frac{x_i x_j}{r^3}  \Big) \Big| \leq 2  \Big | 1 + a \frac{\sin r}{r} \Big |  \leq  2  \Big ( 1 + a \Big ).
\end{equation}
Finally, we combine \eqref{C-7}, \eqref{C-8} and \eqref{C-9} to get 
\[
\| \partial_{x_j} \partial_{x_i} W_a  \|_{L^\infty} \leq  1 + a +  2  \Big ( 1 + a \Big )  = 3 + 3a < 9.
\]
\end{enumerate}
This example shows that our structural assumptions are strictly weaker than the uniform convexity.
\end{example}

\subsection{Previous results and well-posedness theory} \label{sec:3.2}
In this subsection, we first review the previous results for the KCS model in a phase-spatially confined setting. Second, we provide the well-posedness theory of system \eqref{A-4} in a phase-spatially extended setting.

\subsubsection{Phase-spatially confined setting} \label{sec:3.2.1}
Consider the KCS model with an attractive power-law potential in a phase-spatially confined setting. For this, we introduce projected spatial and velocity diameters:
\begin{equation*}
\hspace{-0.7cm} {\mathcal D}_x(t)
		:=
		\sup \Big \{
		|x-x_*|: z, z_*\in\supp f(t)
		\Big \},\quad
		{\mathcal D}_v(t)
		:=
		\sup \Big \{
		|v-v_*|: z, z_* \in\supp f(t)
		\Big \}.
	\end{equation*}
Next, we introduce the second framework $({\mathcal F}_B)$ for weak and strong flocking in terms of system parameters and initial data: 
\vspace{0.2cm}
\begin{itemize}
\item
 $({\mathcal F}_B1):$~The confining potential is radially symmetric and grows algebraically:~there exists a $U \in {\mathcal C}^2([0,\infty)),~~\alpha>2,~~ 0<k_1\leq k_2$ such that 
 \[  W(x)=U(|x|), \quad  U(0)=0, \quad k_1r^{\alpha-1}
		\leq U'(r)
		\leq k_2r^{\alpha-1},
		\quad r \geq 0.  \]
\item
 $({\mathcal F}_B2)$:~Communication weight function satisfies regularity, positivity and bound conditions:
 \[ \phi\in {\mathcal C}_b^1(\bbr_+), \quad  (1+r^2)^{-\beta/2}
		\leq\phi(r)\leq1,
		\quad
		r\geq 0, \quad \beta \in [0, 1].
\]

\item
$({\mathcal F}_B3)$:~Initial datum satisfies positivity, regularity and unit mass conditions:
\[  f_0 \geq 0, \quad f_0\in {\mathcal C}_c^1(\mathbb R^{2d}), \quad \int_{\bbr^{2d}} f_0(z) \dd z = 1. \]
\end{itemize}

\begin{theorem} 
\emph{\cite{LiChen2025,ChenYin2023,ShuTadmor2020}}
	\label{T3.1}
Suppose the framework $({\mathcal F}_B1) - ({\mathcal F}_B3)$ holds, and let $f = f(t,z)$ be a global classical solution to \eqref{A-4}. Then, the following collective dynamics emerge.
\begin{enumerate}
\item
(Weak flocking):~There exists a positive constant $C = C(f_0)$ such that
	\begin{equation*}
		\begin{aligned}
			{\mathcal F}(t)  \leq 
			C(1+t)^{-\frac{\alpha}{\alpha-2}},
			\quad t\geq0.
		\end{aligned}
	\end{equation*}
\item	
(Strong flocking):~Spatial and velocity diameters decay at least algebraically: 
\[
{\mathcal D}_x(t) \leq C(1+t)^{ -\frac{\alpha}{2(\alpha-1)(\alpha-2)}} \quad \mbox{and} \quad {\mathcal D}_v(t) \leq C(1+t)^{
			-\frac{\alpha^2}
			{4(\alpha-1)(\alpha-2)}}, \quad t \geq 0.
\]
\end{enumerate}	
In particular, if potential function $W$ satisfies \eqref{A-5} and $\int_0^{\infty} s\phi(s) \dd s =\infty$, strong  flocking emerges exponentially.
\end{theorem}
\subsubsection{Phase-spatially extended setting} \label{sec:3.2.2}
Let $\mathcal P_2(\bbr^{2d})$ be a probability measure space consisting of probability measures with finite second moment. For probability measures $\mu,\nu$ with finite first moment, let $\Pi(\mu,\nu)$ be the set of their couplings and recall 1-Wasserstein distance between $\mu$ and $\nu$:
\[
\mathcal{W}_1(\mu,\nu):=\inf_{\pi\in\Pi(\mu,\nu)}
\int_{\bbr^{2d}\times\bbr^{2d}}|z-\bar z|\dd\pi(z,\bar z).
\]
For $\mu\in\mathcal P_2(\bbr^{2d})$, similar to \eqref{B-24}, we set the spatial marginal $\rho_\mu$ and the energy density $h_\mu$ as 
\[ 
\rho_\mu :=(\pi_x)_\#\mu,
\quad
\pi_x(x,v):=x
, \quad  h_\mu(z):=\frac12|v|^2+(W*\rho_\mu)(x), \quad z=(x,v) \in \bbr^{2d}. \]
For $\mu_0$-a.e. $z=(x,v)$, we define a forward characteristic flow 
\[ Z_{\mu,t}  = (X_{\mu,t}, V_{\mu, t}) \]
as the solution to the following system:
\begin{equation}\label{C-12}
\begin{cases}
\displaystyle \dot X_\mu=V_\mu, \quad t > 0, \vspace{8pt}\\
\displaystyle\dot V_\mu=\cA[\mu_t](X_\mu,V_\mu)-(\nabla W*\rho_{\mu_t})(X_\mu), \vspace{8pt} \\
\displaystyle  Z_\mu(0;z)=z.
\end{cases}
\end{equation}

\begin{definition}[Lagrangian measure-valued solution]\label{D3.1}
Let $T \in (0, \infty]$ and $\mu_0\in\mathcal P_2(\bbr^{2d})$. A measure-valued map $\mu\in C([0,T);\mathcal P_2(\bbr^{2d}))$ is a  Lagrangian weak  solution to \eqref{A-4} if there is a Borel flow $Z_\mu=(X_\mu,V_\mu)$ generated by \eqref{C-12} such that
\[
\mu_t=(Z_\mu(t, \cdot))_\#\mu_0.
\]
In particular, $\mu$ satisfies \eqref{A-4} against every test function in ${\mathcal C}_c^1([0,T)\times\bbr^{2d})$. One can easily check that a Lagrangian weak solution $\mu_t$ is also a weak solution to \eqref{A-4}.
\end{definition}

\begin{theorem}[Existence, uniqueness, and Osgood-type stability]\label{T3.2}
Suppose that the framework $({\mathcal F}_A1) - ({\mathcal F}_A3)$ holds. Then, the following assertions hold. 
\begin{enumerate}[label=\textnormal{(\roman*)}]
\item For every centered $\mu_0\in\mathcal P_2(\bbr^{2d})$, there exists a global Lagrangian weak solution $\mu\in C([0,\infty);\mathcal P_2(\bbr^{2d}))$. In particular, for an initial datum $f_0\in L^1\cap L^\infty$, there exists a weak solution to \eqref{A-4}. If $\mu_0=f_0\dd z$ with $f_0\in L^1\cap L^\infty$, then
\[  \dd \mu_t(z) =f_t(z) \dd z \quad \text{and} \quad \|f_t \|_{L^\infty}\leq e^{d\kappa\bar{c}_\phi t}\|f_0\|_{L^\infty}, \quad t\geq0.
\]
\item For $T,M,a>0$, we assume that $\mu_0,\nu_0\in\mathcal P_2(\bbr^{2d})$ are centered and satisfy the bound condition:
\begin{equation}\label{C-13}
\max_{\lambda\in\{\mu_0,\nu_0\}}
\int_{\bbr^{2d}}\left(|x|^2+|v|^2+e^{a h_\lambda(x,v)}\right)\dd\lambda(x,v)
\leq M,
\end{equation}
and let $\mu$ and $\nu$ be Lagrangian weak solutions  to \eqref{A-4} with initial data $\mu_0$ and $\nu_0$. Then, there exists a continuous, concave and increasing function $G_T$ with $G_T(0)=0$ such that 
\begin{equation} \label{C-14}
\sup_{0\le t\le T}\mathcal{W}_1(\mu_t,\nu_t) \le G_T \Big (\mathcal{W}_1(\mu_0, \nu_0) \Big).
\end{equation}
Here, $G_T$ depends on $T,M,a,d,\kappa,\bar{c}_\phi,\|\phi'\|_{L^\infty}$ and the structural constants of $W$. In particular, Lagrangian weak solutions are unique in the exponential mechanical-energy class \eqref{C-13}.
\end{enumerate}
\end{theorem}
\begin{proof}
(i)  Define
\[
\cF_W[\mu](x):=-(\nabla W*\rho_\mu)(x).
\]
Then, the Hessian bound for $W$ gives
\begin{equation}\label{C-15}
|\cF_W[\mu](x)-\cF_W[\nu](y)|
\leq  \|D^2W \|_{\infty}\bigl(|x-y|+\mathcal{W}_1(\mu,\nu)\bigr).
\end{equation}
Thus the potential is a globally Lipschitz perturbation of the KCS vector field and has at most linear growth. Let $\mu_t^R $ denote the approximate solution corresponding to the compactly supported approximation indexed by $R>0$. The exact energy identity derived by Remark \ref{R2.1} (1) for these approximate solutions yields, uniformly on every finite time interval,
\begin{equation}\label{C-16}
	\sup_{R>0}\sup_{0\le t\le T}
	\int_{\mathbb R^{2d}} |z|^2\dd\mu_t^R(z)<\infty.
\end{equation}
More precisely, the forces satisfy a linear bound:
\[ |\cA[\mu](x,v)|\leq C\left(|v|+\int_{\bbr^{2d}}|v_*|\dd\mu(z_*) \right), \quad  |\cF_W[\mu](x)|\leq C\left(|x|+\int_{\bbr^d} |x_*|\dd\rho_\mu(x_*)\right). \]
Hence the associated vector fields have uniformly controlled linear
growth. The standard compactness argument
for characteristic measure solutions, together with
\eqref{C-15}--\eqref{C-16}, allows us to extract a subsequence converging
to a global Lagrangian weak solution of \eqref{A-4}; see
\cite{K-M-T-2015,A-G-2008}. \newline

\noindent (ii)~The stability estimate is proved in Appendix~\ref{app-C}. The argument is similar to velocity-truncation method of \cite{H-W-2026}. Note that estimate \eqref{C-15} contributes only a linear multiple of the Lagrangian deviation and therefore does not change the modulus. Setting $\mu_0=\nu_0$ in \eqref{C-14} gives uniqueness.
\end{proof}
\begin{remark}\label{R3.2}
Finite second moments are sufficient for existence, but the stability proof requires an exponential velocity tail. In the present model, this tail follows on every finite time interval from \eqref{C-13} and the characteristic mechanical-energy estimate. A merely polynomial tail leads, after truncation, to a non-Osgood modulus and does not by itself yield uniqueness. Accordingly, the polynomial decay results below apply to every global Lagrangian weak solution furnished by part~\textnormal{(i)}, whereas the exponential-tail dynamics is uniquely determined by its initial datum.
\end{remark}
\subsection{Main results} \label{sec:3.3}
In this subsection, we present two main results on the quantitative weak flocking estimates for \eqref{A-4} under the framework $({\mathcal F}_A)$ introduced in Section \ref{sec:3.1}.
\begin{theorem}\label{T3.3}
Suppose that the framework $({\mathcal F}_A1) - ({\mathcal F}_A3)$ holds, and let $f$ be a global Lagrangian weak solution whose existence is guaranteed by  Theorem~\ref{T3.2} \textnormal{(i)}. Then, depending on the decay rate of communication weight $\phi$ in \eqref{C-2}, we have the following quantitative flocking estimates:
\begin{enumerate}[label=\textnormal{(\roman*)}]
\item If $\beta=0$, then weak flocking emerges exponentially fast:
\begin{equation*}\label{C-17}
\frac{\min \Big \{2, \frac{1}{\bar{c}_W}  \Big \}\cF(0)}{\max \Big \{2, \frac{1}{\underbar{c}_W}  \Big \}} e^{-2\kappa\bar c_\phi t}\leq \cF(t)
\leq   \frac{\bar{c}_F \cF(0)}{ \underbar{c}_F}  e^{-\tilde{\lambda} t},\quad t\ge0.
\end{equation*}
Here, $\underbar{c}_F$, $\bar{c}_F$, and $\tilde{\lambda}$  are positive constants defined in \eqref{ND-11}, \eqref{D-12}, and \eqref{D-13}, respectively.

\vspace{.1cm}

\item
 If $0<\beta\leq2$ and $\cM_q(f_0)<\infty$ for some $q>1$, then weak flocking emerges at least polynomially fast:
\begin{equation}\label{C-18}
\cF(t) \lesssim 
 (1+t)^{-2(q-1)/\beta}, \quad  t\ge0.
\end{equation}
\vspace{.1cm}

\item If $0<\beta\leq2$ and $\cM_{e,a}(f_0)<\infty$ for some $a>0$, then there exist constants $C,c>0$ such that
\begin{equation}\label{C-19}
\frac{\min \Big \{2, \frac{1}{\bar{c}_W}  \Big \}\cF(0)}{\max \Big \{2, \frac{1}{\underbar{c}_W}  \Big \}} e^{-2\kappa\bar c_\phi t}
\leq \cF(t)\leq  Ce^{-ct},
\qquad t\geq0.
\end{equation}
\end{enumerate}
\end{theorem}
\begin{proof}
Since the proofs for the above assertions are very lengthy, we leave them in Section \ref{sec:4}. However, for readers' convenience, we outline a proof strategy only for the second assertion as follows.  \newline

Suppose that parameters and initial datum satisfy the following conditions:
\[ 0<\beta\leq 2, \quad \cM_q(f_0)<\infty, \quad \mbox{for some $q > 1$}, \]
and let $f$ be a global Lagrangian weak solution to \eqref{A-4}.
\vspace{0.1cm}
\begin{itemize}
\item
\textbf{Step A} (Energy dissipation): It follows from Lemma \ref{L2.2} and Lemma \ref{L2.3} that 
\begin{equation} \label{C-20}
 \cF(t) \asymp  (\cE_K + \cE_{I,e})(t), \quad \frac{\dd}{\dd t} (\cE_K + \cE_{I,e})(t)=-\Lambda(t).
\end{equation}
\item
\textbf{Step B} (Microscopic mechanical-energy control): We introduce a microscopic energy density:
\[
h_t(z):=\frac12|v|^2+(W*\rho_t)(x), \quad \rho_t(x) = \int_{\bbr^d} f_t(x,v) \dd v.
\]
Along every characteristic $Z(t) = (X(t), V(t))$, the boundedness of $D^2W$, momentum conservation, and the energy coercivity yield the at-most-linear growth of $h_t$ along the characteristics (see Proposition \ref{P4.1}):
\[
h_t(Z(t)) \leq h_0(z)+ C_* t.
\]
This estimate controls the propagation of the initial mechanical-energy tails despite the exchange between kinetic and potential energies.  
\vspace{0.2cm}
\item
\textbf{Step C} (Time-varying effective regions):
Let $\gamma \geq 1$ be a positive parameter. For the prescribed energy window
\[
R_{A, \gamma}(t)=1+A(1+t)^\gamma,
\]
we introduce time-varying effective region and mass outside the effective region:
\[
\hspace{1cm} \Omega_{A, \gamma}(t):=\{z \in \bbr^{2d}: h_t(z)\leq R_{A, \gamma}(t) \}, \quad  \mathfrak{M}_{A,\gamma}(t):=
\int_{\big(\Omega_{A,\gamma}(t)\big)^c} \Big (1+h_t(z) \Big )f_t(z)\dd z.
\]
By the properties of quadratic confinement for $W$ and the decay of initial mechanical-energy tail, we have
\[
\phi(|x-x_*|)  \gtrsim (1+t)^{-\gamma\beta/2}\quad   \mbox{on $\Omega_{A, \gamma}$}, \quad {\mathfrak M}_{A,\gamma}(t)\leq U_A(1+t)^{-\gamma(q-1)}, \quad t \gg 1.
\]
See Lemma \ref{L4.6} and Lemma \ref{L4.8}. 
\vspace{0.2cm}
\item
\textbf{Step D} (Lower bound estimate of $\Lambda$): Using the decay estimate of $\mathfrak{M}_{A,\gamma}$, the dissipation functional $\Lambda$ satisfies 
\[
\Lambda(t) \geq  |{\mathcal O}(1)| (1+t)^{-\frac{\gamma\beta}{2}} \cE_K(t) -  |{\mathcal O}(1)|  (1 + t)^{-\gamma (  \frac{\beta}{2} + q-1)}, \quad t  \gg 1.
\]
This and \eqref{C-20} imply 
\begin{align}
\begin{aligned} \label{C-21}
& \frac{\dd}{\dd t}  (\cE_K + \cE_{I,e})(t) \\
& \hspace{1cm} \leq -|{\mathcal O}(1)| (1+t)^{-\frac{\gamma\beta}{2}} \cE_K(t)  +  |{\mathcal O}(1)|  (1 + t)^{-\gamma (  \frac{\beta}{2} + q-1)}, \quad t  \gg 1.
\end{aligned}
\end{align}
See Proposition \ref{P4.2}.   As aforementioned in Introduction, the above differential inequality \eqref{C-21} is not sufficient to derive the decay estimates of $\cE_K + \cE_{I,e}$.
\vspace{0.2cm}
\item
\textbf{Step E} (Macroscopic hypocoercivity):~Note that the right-hand side of \eqref{C-21} does not contain $\cE_{I,e},$ so that we cannot use Gr\"onwall's differential inequality for $\cE_K + \cE_{I,e}$. To overcome this drawback, we introduce the compensated energy functional by adding a new term containing a cross term $\cE_{I,s}$ as in \cite{LiChen2025}:
\[
\cE(t)= \cE_K(t) + \cE_{I,e}(t)+\omega(t) \cE_{I,s}(t),
\]
where the functional $\cE_{I,s}$ contains the mixed moment $\int_{\bbr^{2d}} x\cdot vf_t \dd z$ and a primitive of the communication weight which make us use the jargon ``{\it method of hypocoercivity}". The primitive term cancels the alignment contribution in the derivative of the mixed moment, while the strict virial inequality controls the potential energy. As a result, we obtain  the Gr\"onwall's differential inequality for $\cE$:
\[
\cE'(t)
\leq
- |{\mathcal O}(1)|(1+t)^{-\gamma\beta/2}\cE(t)
+  |{\mathcal O}(1)| (1 + t)^{-\gamma (  \frac{\beta}{2} + q-1)}.
\]
See Lemma \ref{L4.9}.
\vspace{0.2cm}
\item
\textbf{Step F} (Polynomial tails): We first apply Gr\"onwall's lemma (Lemma \ref{L2.5}) to derive 
\[ \cE(t)  \leq {\tilde C} (1+t)^{-\gamma (q-1)}, \quad t \geq t_0. \]
For $\beta \in (0, 2)$, we choose $\gamma=2/\beta \in (1, \infty)$ and use a bootstrap argument and Lemma \ref{L2.7} (1) to derive algebraic decay rate of $\cE$:
\[
\cE(t)\lesssim(1+t)^{-2(q-1)/\beta}.
\]
For the critical case $\beta=2$, one first obtains a positive algebraic decay, which makes the characteristic-energy shift sublinear, and then repeats the argument with an arbitrarily small linear window. See Section \ref{sec:4.2.4}.
\end{itemize}
\end{proof}

\noindent Next, we show that the convergence rate of polynomial tails in \eqref{C-18} is optimal.
\begin{theorem}\label{T3.4}
Suppose that parameters and communication weight function satisfy 
\[
d\ge 2, \quad  q>1, \quad  0<\beta\leq2, \quad \phi_\beta(r):=\frac{1}{(1+r^2)^{\beta/2}}.
\]
There exists a fixed radially symmetric confining potential $W_{\rm rs}\in C^\infty(\bbr^d)$ satisfying $({\mathcal F}_A1)$ with $D^2W_{\rm rs}$ having a negative eigenvalue somewhere, and a symmetric global solution of a countably infinite weighted particle system with communication weight function $\phi_\beta$ satisfying the following properties:
\begin{enumerate}
\item
Empirical measure $\mu_t$ has fully noncompact phase-space support and 
\[
\int_{\bbr^{2d}}(1+h_{\mu_0}(z))^q\dd\mu_0(z) <\infty.
\]
\item
There exists a sequence $t_n\to\infty$ and $c>0$ for which
\begin{equation}\label{C-22}
\cF_\mu(t_n)\geq\frac{c}{n^2}(1+t_n)^{-2(q-1)/\beta}.
\end{equation}
\item
For every $\delta>0$,
\begin{equation}\label{C-23}
\limsup_{t\to\infty}(1+t)^{2(q-1)/\beta+\delta}\cF_\mu(t)=\infty.
\end{equation}
\end{enumerate}
\end{theorem}
\begin{proof}
We construct an example to show that the algebraic power in Theorem \ref{T3.3} (ii) is optimal.  Since the proof is very lengthy, we leave it in Section \ref{sec:5}.
\end{proof}
\vspace{.2cm}

\subsection{Comments on previous and main results} \label{sec:3.4} In this subsection, we discuss the frameworks and the results in Theorem \ref{T3.2}, Theorem \ref{T3.3},  and Theorem \ref{T3.4} in the sequel.\newline

	\begin{enumerate}
		\item
		The upper bounds in ~Theorem~\ref{T3.3} \textnormal{(i)}--\textnormal{(iii)} hold for every corresponding Lagrangian weak solution. Under the assumptions of~Theorem~\ref{T3.2} \textnormal{(ii)}, Lagrangian weak solution is unique and depends continuously on the initial law via the Osgood-type stability estimate \eqref{C-14}.
		\vspace{0.2cm}
		\item  At the well-posedness level, the references \cite{HaLiu2009,C-F-R-T-2010}  provide measure-theoretical frameworks and uniform-in-time $\mathcal{W}_1$-stability for compactly supported KCS solutions. The fully noncompact  stability  of \cite{H-W-2026} replaces this by an Osgood modulus under exponential velocity tails. ~Theorem \ref{T3.2} extends that mechanism to the present attractive model: the nonconvex potential is harmless for stability precisely because $D^2W\in L^\infty$, while quadratic confinement supplies the uniform second moments needed in the velocity-tail truncation.
		\vspace{0.2cm}
		\item Classical compact-support arguments yield a positive communication lower bound from the spatial diameter and lead to strong flocking \cite{C-F-R-T-2010,ShuTadmor2020}. Such arguments do not apply when both projected supports are unbounded. Therefore, we need to use the method of time-varying effective region developed in references \cite{H-W-2025,H-W-2026,H-W-CS-2026}. In particular, Ha-Wang-Xue treated noncompact  data without a confining force and proved weak velocity alignment with tail-dependent estimates \cite{HaWangXue2025,H-W-CS-2026}. In the current setting, the potential both generates spatial collapse and destroys monotonicity of high velocity moments. The microscopic mechanical energy, rather than separate position and velocity moments, is the correct tail variable. In particular, we can also verify the optimal convergence rate of weak flocking with the help of confinement force.
		\vspace{0.2cm}
		\item Li, Chen, and Yin treat compactly supported data and radial superquadratic attractions, obtaining weak and strong consensus for high-order power laws \cite{LiChen2025,ChenYin2023}. The high-order power-law results of Li and Chen allow $U''$ to change sign under radial monotonicity assumptions, but their strong-consensus argument uses compact support and a bounded spatial diameter \cite{ChenYin2023,LiChen2025}. Their diameter control is unavailable here, whereas our bounded-Hessian assumption is tailored to quadratic-growth forces and noncompact tails. 
		\vspace{0.2cm}
		\item 
	Uniform convexity provides confinement, force regularity, and virial coercivity through a single assumption. Here we separate these roles through \eqref{C-1}, allowing a genuinely indefinite Hessian. For a fixed nonconvex potential, we prove the optimal polynomial decay exponent. In the exponential-tail regime, the high-energy descent mechanism produces a fixed effective region and yields the optimal exponential weak flocking estimate \eqref{C-19}.
	\vspace{0.2cm}
	\item The sharp and delicate construction in Theorem ~\ref{T3.4} relies on persistent rotating modes
	and therefore requires \(d\geq2\). In one dimension, such rotating
	configurations are unavailable: confined trajectories repeatedly cross
	the central region, where the communication strength is substantially
	larger than its far-field value. This suggests that the
	dimension-independent upper bound in Theorem~\ref{T3.3} (ii) may not be
	optimal in \(d=1\).  Determining the
	optimal one-dimensional decay rate is left for future work.
	\end{enumerate}
	
	\vspace{0.2cm}

In the following two sections, we provide detailed proofs for Theorem \ref{T3.3} and Theorem \ref{T3.4}, respectively.

\section{Emergence of weak flocking}\label{sec:4}
\setcounter{equation}{0}
In this section, we provide a rigorous justification of each assertion in Theorem \ref{T3.3} for $\beta=0$, polynomial and exponentially decaying tail classes, respectively:
\[ \mbox{I}:~\beta = 0, \quad \mbox{II}:~0 < \beta \leq 2, \quad  \cM_q(f_0)<\infty, \quad \mbox{III}:~0 < \beta \leq 2, \quad  \cM_{e,a}(f_0)<\infty. \]
\subsection{First assertion}\label{sec:4.1}
In this subsection, we consider all-to-all communication weight with $\beta = 0$ which is bounded by two constants below and above:
\begin{equation} \label{D-1}
 \beta = 0, \quad \kappa \underbar{c}_\phi-1>0, \quad \mbox{and} \quad   \underbar{c}_\phi \leq \phi(r) \leq \bar{c}_\phi, \quad r \geq 0. 
 \end{equation}
 Recall the energy functionals:
 \begin{align*}
 	\begin{aligned}
 		& {\mathcal E}_K :=\int_{\bbr^{2d}} |v|^2 f_t(z)\dd z, \quad {\mathcal E}_{I,e} :=\int_{\bbr^{2d}} W(x-x_*)\rho_t(x)\rho_t(x_*)\dd x\dd x_*, \\
 		& \Phi(r) :=\kappa\int_0^r s\phi(s)\dd s, \quad  \cE_{I,s}(t) :=2\int_{\bbr^{2d}} (x\cdot v) f_t(z)\dd z
 		+\int_{\bbr^{2d}} \Phi(|x-x_*|)\rho_t(x)\rho_t(x_*)\dd x\dd x_*.
 	\end{aligned}
 \end{align*}
 \subsubsection{Preparatory lemmas} \label{sec:4.1.1}
 In this part, we study time-evolution of $\cE_{I,s}$ along the dynamics of \eqref{A-4} and equivalence between $\cF$ and $\cE$.
\begin{lemma} \label{L4.1}
Suppose that the conditions \eqref{D-1} hold, and let $f$ be a global Lagrangian weak solution to \eqref{A-4} whose existence is guaranteed by  Theorem~\ref{T3.2} \textnormal{(i)}. 
Then the functional $\cE_{I,s}$ defined in \eqref{B-13} satisfies the following estimates: 
\begin{align*}
\begin{aligned}
& (i)~\frac{(\kappa \underbar{c}_{\phi}-1)}{2\bar{c}_W}\cE_{I,e} - \cE_K  \leq  \cE_{I,s} \leq   \frac{(\kappa \bar{c}_{\phi} +1)}{2\underbar{c}_W}\cE_{I,e} + \cE_K.\\
& (ii)~ \frac{\dd}{\dd t}  \cE_{I,s}  \leq  2\cE_K-\nu_W\cE_{I,e}.
\end{aligned}
\end{align*}
\end{lemma}
\begin{proof}
	\noindent
	\textnormal{(i)}
	By the Cauchy--Schwarz inequality, Young's inequality, and the definition of \(\mathcal{I}_x\), we have
	\begin{align}\label{D-2}
	\left|
	2\int_{\bbr^{2d}}x\cdot vf_t(z)\dd z
	\right|
	\leq
	\int_{\bbr^d}|x|^2\rho_t(x)\dd x
	+
	\int_{\bbr^{2d}}|v|^2f_t(z)\dd z=
	\mathcal{I}_x(t)+\cE_K(t).
	\end{align}
	Moreover, the relation
	\[
	\underbar c_{\phi}
	\leq
	\phi(r)
	\leq
	\bar c_{\phi},
	\qquad r\geq0,
	\]
	implies
	\[
	\frac{\kappa\underbar c_{\phi}}{2}r^2
	\leq
	\Phi(r)
	\leq
	\frac{\kappa\bar c_{\phi}}{2}r^2.
	\]
	Hence, by Lemma~\ref{L2.2} (1),
	\begin{align}\label{D-3}
	\begin{aligned}
		\kappa\underbar c_{\phi}\mathcal{I}_x(t)
		&=
		\frac{\kappa\underbar c_{\phi}}{2}
		\int_{\bbr^{2d}}
		|x-x_*|^2
		\rho_t(x)\rho_t(x_*)
		\dd x\dd x_*
		\\
		&\leq
		\int_{\bbr^{2d}}
		\Phi(|x-x_*|)
		\rho_t(x)\rho_t(x_*)
		\dd x\dd x_*
		\\
		&\leq
		\frac{\kappa\bar c_{\phi}}{2}
		\int_{\bbr^{2d}}
		|x-x_*|^2
		\rho_t(x)\rho_t(x_*)
		\dd x\dd x_*=
		\kappa\bar c_{\phi}\mathcal{I}_x(t).
	\end{aligned}
	\end{align}
	Therefore, we combine \eqref{D-2} and \eqref{D-3} to see
	\begin{align}\label{D-4}
	\cE_{I,s}(t)
	\geq
	\bigl(\kappa\underbar c_{\phi}-1\bigr)\mathcal{I}_x(t)
	-\cE_K(t),
	\end{align}
	and
	\begin{align}\label{D-5}
	\cE_{I,s}(t)
	\leq
	\bigl(\kappa\bar c_{\phi}+1\bigr)\mathcal{I}_x(t)
	+\cE_K(t).
	\end{align}
	On the other hand, Lemma~\ref{L2.2} (2) gives
	\[
	2\underbar c_W \mathcal{I}_x(t)
	\leq
	\cE_{I,e}(t)
	\leq
	2\bar c_W \mathcal{I}_x(t).
	\]
	Since the coupling condition in \eqref{D-1} yields
	\[
	\kappa\underbar c_{\phi}-1>0,
	\]
	we obtain
\[
	\bigl(\kappa\underbar c_{\phi}-1\bigr)\mathcal{I}_x(t)
	\geq
	\frac{\kappa\underbar c_{\phi}-1}{2\bar c_W}
	\cE_{I,e}(t).
\]
	Similarly,
	\[
	\bigl(\kappa\bar c_{\phi}+1\bigr)\mathcal{I}_x(t)
	\leq
	\frac{\kappa\bar c_{\phi}+1}{2\underbar c_W}
	\cE_{I,e}(t).
	\]
Finally, we combine the above estimates, \eqref{D-4}, and \eqref{D-5} to see
	\begin{align*}\label{D-6}
	\frac{\kappa\underbar c_{\phi}-1}{2\bar c_W}\cE_{I,e}(t)
	-\cE_K(t)
	\leq
	\cE_{I,s}(t)
	\leq
	\frac{\kappa\bar c_{\phi}+1}{2\underbar c_W}\cE_{I,e}(t)
	+\cE_K(t).
\end{align*}
	\noindent
	\textnormal{(ii)}
	We combine Lemma \ref{L2.3} (iii) and  $\eqref{C-1}_3$ to get the desired estimate:
	\begin{align*}
		\begin{aligned}
			\frac{\dd}{\dd t}  \cE_{I,s}(t) &= 2 \cE_K(t) - \int_{\bbr^{2d}} (x-x_*)\cdot\nabla W(x-x_*)\rho_t(x)\rho_t(x_*)\dd x\dd x_* \\
			&\leq  2 \cE_K(t) - \nu_W  \int_{\bbr^{2d}} W(x-x_*)\rho_t(x)\rho_t(x_*)\dd x\dd x_*.
		\end{aligned}
	\end{align*}	
\end{proof}

\noindent We first choose the time-dependent weight $\omega$ to be the constant:
\[ \omega(t) \equiv \omega_0 :=  \min \Big \{ \frac{1}{4}, \frac{\kappa \underbar{c}_{\phi}}{2} \Big \},  \quad \forall~t \geq 0. \]
Therefore, we use the above definition and $(\cF_A 3)$  to obtain
\[
1 - \omega_0 > 0, \quad  \kappa {\underbar c}_{\phi}-\omega_0>0,  \quad \mbox{and} \quad   \kappa {\underbar c}_{\phi} - 1 > 0.
\]
Then, for such $\omega_0$, we define total energy functional $\cE$ as follows.
\begin{equation} \label{D-8}
{\mathcal E}(t) := {\mathcal E}_K(t) + \omega_0 {\mathcal E}_{I,s}(t) + {\mathcal E}_{I,e}(t).
\end{equation}
\begin{lemma} \label{L4.2}
Suppose the framework $(\cF_A)$ with $\beta = 0$ holds, and let $f$ be a global Lagrangian weak solution to \eqref{A-4}.  Then, the flocking functional $\cF$ is equivalent to the energy functional $\cE$:~there exist positive constants $\underbar{c}_F>0$ and $\bar{c}_F>0$ independent of $t$ such that 
\[  \underbar{c}_F \cE(t) \leq \cF(t) \leq  \bar{c}_F \cE(t), \quad \forall~t \geq 0.
\]
\end{lemma}
\begin{proof}
\noindent $\bullet$~Case A (Derivation of the first inequality):  we use \eqref{B-7} and \eqref{B-12} to get 
\begin{align}
\begin{aligned} \label{D-9}
{\mathcal F} &=  2 {\mathcal I}_x + 2 {\mathcal E}_K \geq  \frac{1}{\bar{c}_W} {\mathcal E}_{I,e} +  2 \cE_K
\geq \min \Big \{2, \frac{1}{\bar{c}_W}  \Big \}   (\cE_K + \cE_{I, e}).  \\
\end{aligned}
\end{align}
On the other hand, it follows from Lemma \ref{L4.1} and \eqref{D-8} that 
\begin{align}
\begin{aligned} \label{D-10}
  \cE &= {\mathcal E}_K +  {\mathcal E}_{I,e} +  \omega_0 {\mathcal E}_{I,s} \\
  &\leq {\mathcal E}_K +  {\mathcal E}_{I,e} + \omega_0 \Big( \frac{(\kappa \bar{c}_{\phi} +1)}{2\underbar{c}_W}\cE_{I,e} + \cE_K \Big) \\
  &=  (1 +\omega_0) \cE_K + \Big(  1 +  \frac{\omega_0(\kappa \bar{c}_{\phi} +1)}{2\underbar{c}_W}  \Big ) \cE_{I,e} \\
  &\leq \max \Big \{ 1 + \omega_0,~   1 +  \frac{\omega_0(\kappa \bar{c}_{\phi} +1)}{2\underbar{c}_W} \Big \} ( \cE_K +  \cE_{I,e}).
\end{aligned}
\end{align}
We combine \eqref{D-9} and \eqref{D-10} to get 
\begin{equation} \label{ND-11}
{\mathcal F} \geq \frac{\min \Big \{2, \frac{1}{\bar{c}_W}  \Big \}}{  \max \Big \{ 1 + \omega_0,~   1 +  \frac{\omega_0(\kappa \bar{c}_{\phi} +1)}{2\underbar{c}_W} \Big \}    }   \cE =: \underbar{c}_F \cE.
\end{equation}

\noindent $\bullet$~Case B (Derivation of the second inequality): 
It follows from Lemma \ref{L4.1} and \eqref{D-8} that 
\begin{align*}
	\begin{aligned}
		\cE &= {\mathcal E}_K +  {\mathcal E}_{I,e} +  \omega_0 {\mathcal E}_{I,s} \\
		&\geq {\mathcal E}_K +  {\mathcal E}_{I,e} + \omega_0 \Big( \frac{(\kappa \underbar{c}_{\phi}-1)}{2\bar{c}_W}\cE_{I,e} - \cE_K  \Big) \\
		&=  (1 - \omega_0) \cE_K + \Big(  1 +  \frac{\omega_0(\kappa \underbar{c}_{\phi} -1)}{2\bar{c}_W}  \Big ) \cE_{I,e} \\
		&\geq \min \Big \{ 1 - \omega_0,~   1 +  \frac{\omega_0(\kappa \underbar{c}_{\phi} -1)}{2\bar{c}_W} \Big \} ( \cE_K +  \cE_{I,e}).
	\end{aligned}
\end{align*}
We use the normalization condition:
\[
\int_{\bbr^{2d}} f_t(z)\dd z=1,
\quad
\int_{\bbr^{2d}} vf_t(z)\dd z=0,
\quad
\int_{\bbr^{2d}} xf_t(z)\dd z=0, \quad t > 0,
\]
and use \eqref{B-8}, \eqref{B-9}, and \eqref{B-10} to find 
\begin{align}
\begin{aligned} \label{D-12}
{\mathcal F}(t) &= \int_{\bbr^{4d}}
\bigl(|x-x_*|^2+|v-v_*|^2\bigr)
 f_t(z) f_t(z_*)\dd z \dd z_* \\
 &= 2 {\mathcal I}_x(t) + 2 {\mathcal E}_K(t) \leq \frac{1}{\underbar{c}_W} {\mathcal E}_{I,e}(t) + 2 {\mathcal E}_K(t)  \\
 & \leq \max \Big \{2, \frac{1}{\underbar{c}_W}  \Big \}  (\cE_K + \cE_{I, e})(t) \leq \frac{\max \Big \{2, \frac{1}{\underbar{c}_W}  \Big \}}{ \min \Big \{ 1 - \omega_0,~   1 +  \frac{\omega_0(\kappa \underbar{c}_{\phi} -1)}{2\bar{c}_W} \Big \}}  \cE(t) =:  \bar{c}_F \cE(t).
\end{aligned}
\end{align}

\end{proof}
\begin{lemma} \label{L4.3}
Suppose the framework $(\cF_A)$ with $\beta = 0$ holds, and let $f$ be a global Lagrangian weak solution to \eqref{A-4}. Then, the following assertions hold.
\begin{enumerate}
\item
The dissipation functional $\Lambda$ is bounded below by the constant multiple of $\cE_K$:
\[ 2\kappa{\underbar c}_{\phi} \cE_K(t) \leq \Lambda(t) \leq 2\kappa {\bar c}_{\phi} \cE_K(t) , \quad t \geq 0. \]
\item
The energy functional decays exponentially fast: there exists a positive constant $\tilde{\lambda}$ such that 
\[    \min \Big \{ 1 - \omega_0,~   1 +  \frac{\omega_0(\kappa \underbar{c}_{\phi} -1)}{2\bar{c}_W} \Big \}(\cE_K + \cE_{I,e})(0) e^{-2 \kappa {\bar c}_{\phi} t} \leq \cE(t) \leq e^{-\tilde{\lambda} t} \cE(0), \quad t \geq 0. \]
In fact, the decay exponent $\tilde{\lambda}$ is given by the following explicit formulation:
\begin{align} \label{D-13}
	 \tilde{\lambda} := \frac{\min \Big\{ 2( {\kappa \underbar c}_{\phi} - \omega_0 ),~  \omega_0 \nu_W  \Big \} }{\max \Big \{ 1 + \omega_0,~   1 +  \frac{\omega_0(\kappa \bar{c}_{\phi} +1)}{2\underbar{c}_W} \Big \}  },\quad \text{where} \quad \omega_0=  \min \Big \{ \frac{1}{4}, \frac{\kappa \underbar{c}_{\phi}}{2} \Big \}. 
	 \end{align}
\end{enumerate}
\end{lemma}
\begin{proof}
\noindent (1)~We use \eqref{D-1} to get 
\begin{align}
\begin{aligned} \label{D-14}
\Lambda(t) &= \kappa\int_{\bbr^{4d}}\phi(|x-x_*|)|v-v_*|^2f_t(z)f_{t}(z_*)\dd z \dd z_* \\
&\geq \kappa \underbar{c}_{\phi}\int_{\bbr^{4d}} |v-v_*|^2f_t(z)f_{t}(z_*)\dd z \dd z_* = 2 \kappa {\underbar c}_{\phi} \cE_K(t).
\end{aligned}
\end{align}
Similarly,  we have
\begin{align}
\begin{aligned} \label{D-15}
\Lambda(t) &= \kappa\int_{\bbr^{4d}}\phi(|x-x_*|)|v-v_*|^2f_t(z)f_{t}(z_*)\dd z \dd z_* \\
&\leq \kappa \bar{c}_{\phi}\int_{\bbr^{4d}} |v-v_*|^2f_t(z)f_{t}(z_*)\dd z \dd z_* = 2 \kappa {\bar c}_{\phi} \cE_K(t).
\end{aligned}
\end{align}
\noindent (2)~For an upper bound estimate, we use Lemma \ref{L4.1} (ii), \eqref{D-10}, and \eqref{D-14} to get 
\begin{align*}
\begin{aligned}  \label{D-16}
{\dot \cE} &= ({\dot \cE}_K + {\dot \cE}_{I,e}) +\omega_0 {\dot \cE}_{I,s} \leq -\Lambda + \omega_0 \Big(   2\cE_K-\nu_W\cE_{I,e}        \Big)  \\
&\leq  -2 \kappa {\underbar c}_{\phi} \cE_K +  \omega_0 \Big(   2\cE_K-\nu_W\cE_{I,e}        \Big) =- 2( \kappa {\underbar c}_{\phi} - \omega_0 ) \cE_K - \omega_0 \nu_W  \cE_{I,e}  \\
& \leq -\min \Big\{ 2( \kappa {\underbar c}_{\phi} - \omega_0 ),~  \omega_0 \nu_W  \Big \}  (\cE_K + \cE_{I,e}) \\
& \leq  - \frac{\min \Big\{ 2( {\kappa \underbar c}_{\phi} - \omega_0 ),~  \omega_0 \nu_W  \Big \} }{\max \Big \{ 1 + \omega_0,~   1 +  \frac{\omega_0(\kappa \bar{c}_{\phi} +1)}{2\underbar{c}_W} \Big \}  } \cE =: -\tilde{\lambda} \cE.
\end{aligned}
\end{align*}
This yields the desired upper bound estimate for $\cE$.  \newline

On the other hand, for the lower bound estimate of $\cE$, we take a detour route, since we do not have a lower bound estimate for ${\dot \cE}_{I, s}$.  First, note that a priori estimate between $\cE$ and $\cE_K + \cE_{I,e}$ in Case B of Lemma \ref{L4.2}:
\begin{equation} \label{D-17}
\cE \geq  \min \Big \{ 1 - \omega_0,~   1 +  \frac{\omega_0(\kappa \underbar{c}_{\phi} -1)}{2\bar{c}_W} \Big \}(\cE_K + \cE_{I,e}).
\end{equation}
Then, it follows from Remark \ref{R2.1} and  \eqref{D-15} that 
\[
\frac{\dd}{\dd t} (\cE_K + \cE_{I,e}) = -\Lambda(t) \geq  - 2 \kappa {\bar c}_{\phi} \cE_K(t) \geq -2 \kappa {\bar c}_{\phi}  (\cE_K + \cE_{I,e}), \quad t > 0. 
\]
This  implies
\begin{equation} \label{D-18}
(\cE_K + \cE_{I,e})(t) \geq (\cE_K + \cE_{I,e})(0) e^{-2 \kappa {\bar c}_{\phi} t}, \quad t \geq 0.
\end{equation}
Finally, we combine \eqref{D-17} and \eqref{D-18} to get the lower bound estimate of $\cE$.
\end{proof}
\vspace{0.2cm}

\subsubsection{Proof of the first assertion} \label{sec:4.1.2}
~Suppose that the framework $({\mathcal F}_A)$ holds with $\beta = 0$, and let $f$ be a global Lagrangian weak solution to \eqref{A-4}. \newline

\noindent $\bullet$~Step A (Upper bound estimate of $\cF$):~It follows from Lemma \ref{L4.2} and Lemma \ref{L4.3} that 
\[
\cF(t) \leq \bar{c}_F \cE(t) \leq \bar{c}_F  e^{-\tilde{\lambda} t} \cE(0) \leq  \frac{\bar{c}_F \cF(0)}{ \underbar{c}_F}  e^{-\tilde{\lambda} t}, \quad t \geq 0.
\]
\noindent $\bullet$~Step B (Lower bound estimate of $\cF$): ~
By the exact energy identity and the upper bound \(\phi\leq\bar c_\phi\),
\[
\frac{\dd}{\dd t}
\left(
\cE_K+\cE_{I,e}
\right)=
-\Lambda(t)
\geq
-2\kappa\bar c_\phi\cE_K(t)
\geq
-2\kappa\bar c_\phi
\left(
\cE_K+\cE_{I,e}
\right)(t).
\]
Thus, Gr\"onwall inequality  yields
\[
\left(
\cE_K+\cE_{I,e}
\right)(t)
\geq
\left(
\cE_K+\cE_{I,e}
\right)(0)
e^{-2\kappa\bar c_\phi t}.
\]
Using Lemma~\ref{L2.2} (4), we obtain the lower bound
\[\cF(t)\ge \min \Big \{2, \frac{1}{\bar{c}_W}  \Big \} \left(
\cE_K+\cE_{I,e}
\right)(t)\ge\cF(0)\frac{\min \Big \{2, \frac{1}{\bar{c}_W}  \Big \}}{\max \Big \{2, \frac{1}{\underbar{c}_W}  \Big \}}e^{-2\kappa\bar c_\phi t}. \]
 This completes the proof of the first assertion. 

\subsection{Second assertion} \label{sec:4.2}
In this subsection, we consider the situation:
\[
0 < \beta \leq 2, \quad  \cM_q(f_0)<\infty.
\]

\subsubsection{Microscopic mechanical energy}\label{sec:4.2.1}
Consider the characteristic flow $(X,V)$ issued from $z = (x, v) \in \bbr^{2d}$:
\begin{equation}\label{D-21}
\begin{cases}
\displaystyle \dot X=V, \qquad t > 0, \vspace{6pt}\\
\displaystyle \dot V=\cA[f](t,X,V)-(\nabla W*\rho)(X), \vspace{6pt}\\
(X,V)\big|_{t=0}=(x,v).
\end{cases}
\end{equation}
Define an energy density along the characteristic flow:
\begin{equation}\label{D-22}
h(t):=\frac12|V(t)|^2+(W*\rho_t)(X(t)) = \frac12|V(t)|^2+ \int_{\bbr^d} W(X(t) - x_*) \rho_t(x_*) \dd x_*.
\end{equation}
Next, we estimate the time-evolution of $h(t)$ along the characteristic flow \eqref{D-21}.
\begin{lemma}\label{L4.4}
For a global Lagrangian weak solution $f = f_t(z)$ to \eqref{A-4} with normalizations \eqref{B-2}, let $(X(t), V(t))$ be the characteristic flow to \eqref{D-21}. Then, one has 
\[ \frac{\dd h(t)}{\dd t}=V(t) \cdot\cA[f_t](X(t),V(t)) -\int_{\bbr^d} \nabla W(X(t) -x_*)\cdot j_t(x_*)\dd x_*, \quad t > 0. \]
\end{lemma}
\begin{proof} 
\noindent We differentiate \eqref{D-22} with respect to $t$ and use the continuity equation:
 \[ \partial_t \rho_t(x)+\nabla\cdot j_t(x)=0 \]
to get the desired estimate:
\begin{align*}
\begin{aligned} \label{D-23}
\frac{\dd h(t)}{\dd t} &= V(t) \cdot {\dot V}(t) +  \int_{\bbr^d}  \Big(  \nabla W(X(t) - x_*) \cdot V(t)  \rho_t(x_*) + W(X(t) - x_*) \partial_t \rho_t(x_*)  \Big) \dd x_* \\
& =   V(t) \cdot \Big[ \cA[f_t](X(t),V(t))-(\nabla W*\rho_t)(X(t)) \Big] + V(t) \cdot (\nabla W* \rho_t)(X(t)) \\
&\quad - (W * \nabla \cdot j_t(X(t)))  \\
&=V(t) \cdot\cA[f_t](X(t),V(t)) -\int_{\bbr^{2d}} \Big( \nabla W(X(t)-x_*)\cdot v_* \Big) f_t(z_*)\dd z_*, \quad t > 0.
\end{aligned}
\end{align*}
\end{proof}
\begin{lemma}\label{L4.5}
Let $f = f_t(z)$ be a global Lagrangian weak solution to \eqref{A-4} satisfying normalizations \eqref{B-2}. Then, the following assertions hold.
\begin{enumerate}
\item
For every $x \in\bbr^d$, there exists a structural constant $C_W$ such that
\begin{equation*}
\left|\int_{\bbr^{2d}} \nabla W(x-x_*)\cdot v_* f_t(z_*)\dd z_* \right|
\leq C_W (\cE_K + \cE_{I,e})(t),
\end{equation*}
where $C_W$ is a positive constant defined by 
\[ C_W =\max \Big \{ \frac{1}{4\underbar{c}_W},~ \frac{1}{2}\Big \}  \|D^2W \|_{\infty} . \]
\item
Along the characteristic flow \eqref{D-21}, we have
\[ V(t) \cdot\cA[f_t](X(t),V(t))\leq  \frac{ \kappa\bar{c}_\phi}{4} \cE_K(t). \]
\end{enumerate}
\end{lemma}
\begin{proof} 
\noindent (1)~For $x \in \bbr^{d}$, we use $\int_{\bbr^{2d}} v_* f_t(z_*) \dd z_*=0$ to get 
\begin{equation} \label{D-24}
\int_{\bbr^{2d}} \nabla W(x)\cdot v_* f_t(z_*)\dd z_* = 0.
\end{equation}
Then, \eqref{D-24}, Lemma \ref{L2.2} and the Cauchy--Schwarz inequality yield
\begin{align*}
\begin{aligned}
& \Big| \int_{\bbr^{2d}} \nabla W(x-x_*)\cdot v_* f_t(z_*)\dd z_* \Big| \\
& \hspace{1cm}  = \Big|
\int_{\bbr^{2d}} \bigl[\nabla W(x-x_*)-\nabla W(x)\bigr]\cdot v_*f_t(z_*)\dd z_* \Big| \\
& \hspace{1cm} \leq  \int_{\bbr^{2d}} |\nabla W(x-x_*)-\nabla W(x)| |v_*| f_t(z_*)\dd z_* \\
& \hspace{1cm} \leq  \|D^2W \|_{\infty} \int_{\bbr^{2d}}  |x_*| |v_*| f_t(z_*)\dd z_* \\
&  \hspace{1cm} \leq   \|D^2W \|_{\infty} \Big( \int_{\bbr^{2d}}  |x_*|^2  f_t(z_*)\dd z_* \Big)^{\frac{1}{2}}  \Big( \int_{\bbr^{2d}} |v_*|^2 f_t(z_*)\dd z_* \Big)^{\frac{1}{2}} \\
& \hspace{1cm} \leq \frac{1}{2}  \|D^2W \|_{\infty}  2 \Big( \frac{1}{2\underbar{c}_W}  \cE_{I,e}  \Big)^{\frac{1}{2}} \Big(  \cE_K \Big)^{\frac{1}{2}} \\
& \hspace{1cm} \leq  \frac{1}{2}  \|D^2W \|_{\infty} \Big(  \frac{1}{2\underbar{c}_W}  \cE_{I,e}  +   \cE_K  \Big) \\
& \hspace{1cm} \leq  \max \Big \{ \frac{1}{4\underbar{c}_W},~  \frac{1}{2}  \Big \}    \|D^2W \|_{\infty} ( \cE_K + \cE_{I,e} )=: C_W (   \cE_K + \cE_{I,e}).
\end{aligned}
\end{align*}
\noindent (2)~For $x \in \bbr^d$, we set 
\begin{equation*} \label{D-25}
{\varrho}_{\phi}(t,x):=\int_{\bbr^{2d}} \phi(|x-x_*|)f_t(z_*)\dd z_*,
\quad
{m}_{\phi}(t,x):=\int_{\bbr^{2d}} \phi(|x-x_*|) v_* f_t(z_*)\dd z_*.
\end{equation*}
Then we use the above ansatz, the Cauchy--Schwarz inequality and $\phi(r) \leq \bar{c}_\phi$ to find 
\begin{align*}
\begin{aligned} 
\cA[f_t](z) &=\kappa\int_{\bbr^{2d}}\phi(|x-x_*|)(v_*-v)f_t(z_*)\dd z_* = \kappa \Big (m_{\phi}(t,x) - \varrho_{\phi}(t,x) v \Big ), \\
|m_{\phi}(t,x)|^2 &\leq  \Big|  \int_{\bbr^{2d}} \phi(|x-x_*|) v_* f_t(z_*)\dd z_* \Big|^2 \leq \Big( \int_{\bbr^{2d}} \phi(|x-x_*|)  |v_*| f_t(z_*) \dd z_*  \Big)^2 \\
&\leq \Big(  \int_{\bbr^{2d}} |\phi(|x-x_*|)|^2  f_t(z_*) \dd z_* \Big)  \Big(  \int_{\bbr^{2d}} |v_*|^2 f_t(z_*) \dd z_* \Big)  \\
& \leq   \bar{c}_{\phi} {\varrho_{\phi}}(t,x) \cE_K(t).
\end{aligned}
\end{align*}
Again, we use the above relation and Young's inequality to find  
\begin{align*}
\begin{aligned}
 & V(t) \cdot\cA[f_t](X(t),V(t))  \\
 & \hspace{1cm} = \kappa  V(t) \cdot m_{\phi}(t,X(t)) - \kappa  \varrho_{\phi}(t,X(t))  | V(t) |^2 \\
 &\hspace{1cm}\leq  \kappa \sqrt{2 \varrho_\phi(t,X(t))} |V(t)| \frac{ |m_{\phi}(t,X(t))|}{\sqrt{2 \varrho_\phi(t, X(t))}} - \kappa  \varrho_{\phi}(t,X(t))  | V(t) |^2 \\
 & \hspace{1cm} \leq \frac{\kappa}{2} \Big( 2 \varrho_{\phi}(t, X(t))  |V(t)|^2 + \frac{|m_{\phi}(t,X(t))|^2}{2 \varrho_\phi(t, X(t))} \Big) - \kappa  \varrho_{\phi}(t,X(t))  | V(t) |^2 \\
 & \hspace{1cm} = \frac{\kappa |m_{\phi}(t,X(t))|^2}{4 \varrho_\phi(t, X(t))} \leq \frac{ \kappa\bar{c}_\phi}{4}  \cE_K(t).
 \end{aligned}
 \end{align*}
Note that when $\varrho_\phi(t, X(t))=0$, we have $|m_{\phi}(t,X(t))|^2=0$ and $\cA[f_t](X(t),V(t)) =0$. Therefore the second estimate holds.
 \end{proof}
\begin{proposition}\label{P4.1}
Let $f = f_t(z)$ be a global Lagrangian weak solution to \eqref{A-4} satisfying the normalizations \eqref{B-2}. Then, there exists $C_* >0$ such that
\begin{align}\label{D-27}
\begin{aligned} 
& (i)~h_t(X(t),V(t))
\leq h_0(z) +  \bigg(\frac{ \kappa\bar{c}_\phi}{4}  + C_W\bigg)  \int_0^t (\cE_K + \cE_{I,e})(s) \dd s. \\
& (ii)~h_t(X(t),V(t))
\leq h_0(z) + C_* t, \quad C_* :=\bigg(\frac{ \kappa\bar{c}_\phi}{4} + C_W\bigg) (\cE_K + \cE_{I,e})(0).
\end{aligned}
\end{align}
\end{proposition}
\begin{proof}
\noindent (i)~It follows from Lemma \ref{L4.4} and Lemma \ref{L4.5} that 
\begin{align*}
\begin{aligned}
 \frac{\dd h(t)}{\dd t} &=V(t) \cdot\cA[f_t](X(t),V(t)) -\int_{\bbr^d} \nabla W(X(t) -x_*)\cdot j_t(x_*)\dd x_*  \\
 &\leq \frac{ \kappa\bar{c}_\phi}{4} \cE_K(t) + C_W (\cE_K + \cE_{I,e}) (t) \\
 &\leq \bigg(\frac{ \kappa\bar{c}_\phi}{4} + C_W\bigg)  (\cE_K + \cE_{I,e})(t).
\end{aligned}
\end{align*}
We integrate the above relation from $0$ to $t$ to get the desired first assertion. \newline

\noindent (ii)~We use \eqref{B-22}:
\[  (\cE_K + \cE_{I, e})(t)\leq  (\cE_K + \cE_{I, e})(0) \]
and the first assertion to find 
\[
 \frac{\dd h(t)}{\dd t} \leq \bigg(\frac{ \kappa\bar{c}_\phi}{4}  + C_W\bigg)  (\cE_K + \cE_{I,e})(t) \leq  \bigg(\frac{ \kappa\bar{c}_\phi}{4}  + C_W\bigg)  (\cE_K + \cE_{I,e})(0) = C_*.
\]
This yields the desired second assertion. 
\end{proof}

\subsubsection{Time-varying effective region}\label{sec:4.2.2}
Now, we move on to \textbf{Step C} in the proof of Theorem \ref{T3.3} (ii). For $A>0$, $\gamma>0$, we introduce a set of functions:
\begin{align*}
\begin{aligned} \label{D-28}
& R_{A,\gamma}(t):= 1 +A(1+t)^\gamma, \quad  h_t(z) :=\frac12|v|^2+ (W*\rho_t)(x), \\
& \Omega_{A,\gamma}(t):= \Big \{z \in \bbr^{2d}:~h_t(z)\leq R_{A,\gamma}(t) \Big \}, \quad \mathfrak{M}_{A,\gamma}(t):=
\int_{\big(\Omega_{A,\gamma}(t)\big)^c} \Big (1+h_t(z) \Big )f_t(z)\dd z.
\end{aligned}
\end{align*}
\begin{lemma}\label{L4.6}
Suppose that parameters and initial datum satisfy the following conditions:
\[  \mbox{Either} \quad \gamma>1 \quad \mbox{or} \quad \gamma=1~~\mbox{and}~~A=2C_*+2; \qquad \cM_q(f_0)<\infty, \]
where $C_*$ is defined in \eqref{D-27}, and let $f$ be a global Lagrangian weak solution to \eqref{A-4}. Then, there exists a positive constant $U_A$ such that 
\begin{equation*}\label{D-29}
{\mathfrak M}_{A,\gamma}(t)\leq U_A(1+t)^{-\gamma(q-1)}, \quad t \gg 1.
\end{equation*}
Here, 
\begin{equation*}\label{D-30}
	U_A=	\max \{1, C_* \} \Big(1 + \Big(\frac{A}{2}\Big)^{-1}\Big)\Big(\frac{A}{2}\Big)^{-(q-1)}  \cM_q(f_0). 
\end{equation*}
\end{lemma}
\begin{proof}  Since the proof is very lengthy, we split its proof into several steps. \newline

\noindent $\bullet$~Step A (Identification of a time-dependent effective region):~For $z \in \bbr^{2d}$, we denote the characteristic flow associated with $f$ as follows:
	\begin{equation*}
		Z_t(z)
		:=
		\bigl(X(t;z),V(t;z)\bigr).
	\end{equation*}
 Since $f_t= (Z_t)_{\#}f_0,$ it follows from the definition of ${\mathfrak M}_{A,\gamma}(t)$ that 
	\begin{align*}
		{\mathfrak M}_{A, \gamma}(t)
		&=
		\int_{\bbr^{2d}}
		\Bigl(1+h_t\bigl(Z_t(z)\bigr)\Bigr)
		\mathbf{1}_{\{
				h_t(Z_t(z))> R_{A,\gamma}(t)
		\}}
		f_0(z)\dd z.
	\end{align*}
	By \eqref{D-27}, we have
	\begin{equation*} \label{D-31}
		h_t\bigl(Z_t(z)\bigr)
		\leq
		h_0(z)+C_*t.
	\end{equation*}
If 
	\begin{equation*}
		h_t\bigl(Z_t(z)\bigr)		
		> R_{A,\gamma}(t),
	\end{equation*}
	then we have
	\begin{equation*}
		h_0(z)
		>
		R_{A,\gamma}(t)-C_*t.
	\end{equation*}
	We set 
	\begin{equation*}
		{\hat R}(t)
		:= R_{A,\gamma}(t)-C_*t=
		1+A(1+t)^\gamma-C_*t.
	\end{equation*}
	Moreover, we have
	\begin{equation*}
		1+h_t\bigl(Z_t(z)\bigr)
		\leq
		1+h_0(z)+C_*t.
	\end{equation*}
	Therefore, we have
	\begin{equation}\label{D-32}
		{\mathfrak M}_{A, \gamma}(t)
		\leq
		\int_{\{h_0(z) > {\hat R}(t) \}}
		\bigl(1+h_0(z)+C_*t\bigr)
		f_0(z)\dd z.
	\end{equation}
	Next, we derive a lower bound for ${\hat R}(t)$ depending on $\gamma$:
\[ \mbox{Either}~~\gamma > 1 \quad \mbox{or} \quad \gamma = 1. \]
\noindent $\diamond$~\textbf{Case 1.1}:~Suppose that
\[ \gamma>1. \]
In this case, we have
	\begin{equation*}
		\frac{C_*t}{(1+t)^\gamma}
		\longrightarrow 0
		\qquad\text{as }t\to\infty.
	\end{equation*}
	Thus, there exists $t_0>0$ such that
	\begin{equation*}
		C_*t
		\leq
		\frac{A}{2}(1+t)^\gamma,
		\qquad
		t\geq t_0,
	\end{equation*}
	and consequently
	\begin{equation*}
		{\hat R}(t)
		\geq
		1+\frac{A}{2}(1+t)^\gamma
		\geq
		\frac{A}{2}(1+t)^\gamma,
		\qquad
		t\geq t_0.
	\end{equation*}
	\vspace{0.2cm}
\noindent $\diamond$~\textbf{Case 1.2}:~Suppose that 
\[ \gamma =1 \quad \mbox{and} \quad  A > C_*. \]
In this case, there exists a positive constant $c_0$ such that 
	\begin{align*}
		{\hat R}(t)
		=
		1+A(1+t)-C_*t
		=
		1+A+(A-C_*)t
		\geq
		c_0(1+t).
	\end{align*}
	
For simplicity of notation, we can take $A=2C_*+2$. Therefore, it follows from Case 1.1 and Case 1.2 that there exist constants $c_0:=\frac{A}{2}>0$ and $t_0>0$ such that
	\begin{equation}\label{D-33}
		{\hat R}(t)
		\geq
		c_0(1+t)^\gamma,
		\qquad
		t\geq t_0.
	\end{equation}
If necessary, we take $t_0 \gg 1$ such that 
	\[ {\hat R}(t)\geq1, \quad t \geq t_0  \gg 1. \]
\noindent $\bullet$~Step B (Decay estimate of weighted exterior tail):~We split the right-hand side of \eqref{D-32} to obtain
	\begin{align*}
	\begin{aligned}
		{\mathfrak M}_{A, \gamma}(t)
		&\leq
		\int_{{h_0(z) > {\hat R}(t)}}
		(1+h_0(z))f_0(z)\dd z
		+C_*t
		\int_{{h_0> {\hat R}(t)}}
		f_0(z)\dd z \\
		& =:\mathcal{I}_{21}(t)+C_*t \times \mathcal{I}_{22}(t).
	\end{aligned}
	\end{align*}
Below, we estimate the terms ${\mathcal I}_{2i}$ one by one. \newline	
	
\noindent $\diamond$ {\bf Case 2.1} (Estimate of $\mathcal{I}_{21}$): On the set $\{h_0(z) > {\hat R}(t)\}$, we have
	\begin{equation*}
		1+h_0(z)
		\geq
		1+ {\hat R}(t).
	\end{equation*}
	Since $q>1$, it follows that
	\begin{equation*}
		1+h_0(z) =
		(1+h_0(z))^q(1+h_0(z))^{-(q-1)}
		\leq
		(1+{\hat R}(t))^{-(q-1)}(1+h_0(z))^q.
	\end{equation*}
	Therefore, we have
	\begin{align*}
	\mathcal{I}_{21}(t)
		\leq
		(1+ {\hat R}(t))^{-(q-1)}
		\int_{\bbr^{2d}}
		(1+h_0(z))^qf_0(z)\dd z
		=
		(1+ {\hat R}(t))^{-(q-1)}
		\cM_q(f_0).
	\end{align*}
	Since ${\hat R}(t)\geq1$, we have
	\begin{equation}\label{D-34}
		\mathcal{I}_{21}(t)
		\leq {\hat R}(t)^{-(q-1)}\cM_q(f_0).
	\end{equation}
	\vspace{0.1cm}
	
\noindent $\diamond$ \textbf{Case 2.2} (Estimate of $\mathcal{I}_{22}$):~On the set $\{h_0(z) > {\hat R}(t)\}$, we have
	\begin{equation*}
		1 = (1 + h_0(z))^{-q} (1 + h_0(z))^{q}  \leq 
		(1+ {\hat R}(t))^{-q}(1+h_0)^q.
	\end{equation*}
	Hence, we have
	\begin{align*}
		\mathcal{I}_{22}(t)
		\leq
		(1+ {\hat R}(t))^{-q}
		\int_{\bbr^{2d}}
		(1+h_0(z))^qf_0(z)\dd z
		=
		(1+ {\hat R}(t))^{-q}\cM_q(f_0).
	\end{align*}
This yields
	\begin{equation}\label{D-35}
		\mathcal{I}_{22}(t)
		\leq
		{\hat R}(t)^{-q}\cM_q(f_0).
	\end{equation}
We combine \eqref{D-34} and \eqref{D-35} to find
	\begin{equation}\label{D-36}
		{\mathfrak M}_{A, \gamma}(t)
		\leq
		\max \{1, C_* \} \left(
		{\hat R}(t)^{-(q-1)}
		+t {\hat R}(t)^{-q}
		\right)\cM_q(f_0).
	\end{equation}
	To compare two terms, note that
	\begin{align*}
		t {\hat R}(t)^{-q}
		=
		{\hat R}(t)^{-(q-1)}
		\frac{t}{ {\hat R}(t)}
		\leq
		\frac{1}{c_0}  {\hat R}(t)^{-(q-1)}
		(1+t)^{1-\gamma}, \quad t \gg 1,
	\end{align*}
	where we used \eqref{D-33}. Since $\gamma\geq1$,
	\begin{equation*}
		(1+t)^{1-\gamma}\leq1.
	\end{equation*}
	Consequently, we have
	\begin{equation*} \label{D-37}
		t {\hat R}(t)^{-q}
		\leq
		\frac{1}{c_0} {\hat R}(t)^{-(q-1)}, \quad t \gg 1.
	\end{equation*}
	It follows from \eqref{D-36} that
	\begin{equation*}
		{\mathfrak M}_{A, \gamma}(t)
		\leq
		\max \{1, C_* \} (1 + c_0^{-1})  {\hat R}(t)^{-(q-1)}\cM_q(f_0).
	\end{equation*}
	Finally, we use \eqref{D-33} to obtain the desired estimate:
	\begin{align*}
	\begin{aligned}
		{\mathfrak M}_{A, \gamma}(t)
		&\leq
		\max \{1, C_* \} (1 + c_0^{-1}) c_0^{-(q-1)}  \cM_q(f_0) (1+t)^{-\gamma(q-1)} \\
		&=: U_A(C_*, c_0, \cM_q(f_0)) (1+t)^{-\gamma(q-1)}.
\end{aligned}
\end{align*}
\end{proof}
Next, we characterize the effective region $\Omega_{A,\gamma}$ in terms of $R_{A,\gamma}$.
\begin{lemma} \label{L4.7}
Suppose that $z = (x,v)$ belongs to time-varying effective region $\Omega_{A,\gamma}$ with $A>0$ and $\gamma>0$. Then we have
\[
|x|\leq\sqrt{\frac{R_{A,\gamma}(t)}{\underbar{c}_W}}
\quad \mbox{and} \quad |v|\leq\sqrt{2R_{A,\gamma}(t)}.
\]
\end{lemma}
\begin{proof}
\noindent Inside time-varying effective region $\Omega_{A,\gamma}$, note that 
\begin{align}
\begin{aligned} \label{D-38}
 z \in \Omega_{A,\gamma} & \Longleftrightarrow \quad  h_t(z) \leq R_{A,\gamma}(t)  \\
 &\Longleftrightarrow \quad \frac12|v|^2+ (W*\rho_t)(x) \leq  R_{A,\gamma}(t) \\
 & \Longrightarrow \quad \frac12|v|^2 <  R_{A,\gamma}(t), \quad  (W*\rho_t)(x) \leq R_{A,\gamma}(t).
\end{aligned}
\end{align}
On the other hand, we use the normalization conditions:
\[ \int_{\bbr^d} \rho_t(x) \dd x = 1 \quad\text{and}\quad \int_{\bbr^d} x \rho_t(x) \dd x = 0 \]
to get 
\begin{align}
\begin{aligned} \label{D-39}
(W*\rho_t)(x) &= \int_{\bbr^{d}} W(x - x_*) \rho_t(x_*) \dd x_* \geq \underbar{c}_W  \int_{\bbr^{d}} |x-x_*|^2 \rho_t(x_*) \dd x_* \\
&= \underbar{c}_W \Big(  |x|^2 + \int_{\bbr^{d}} |x_*|^2 \rho_t(x_*) \dd x_* - 2 x \cdot \int_{\bbr^d} x_* \rho_t(x_*) \dd x_*    \Big) \\
& =  \underbar{c}_W \Big(  |x|^2 + \int_{\bbr^{d}} |x_*|^2 \rho_t(x_*) \dd x_* \Big) \geq  \underbar{c}_W  |x|^2.
\end{aligned}
\end{align}
Thus, it follows from \eqref{D-38} and \eqref{D-39} that 
\[ z \in \Omega_{A,\gamma} \quad  \Longrightarrow \quad  \underbar{c}_W  |x|^2 \leq R_{A,\gamma}(t), \quad  \frac12|v|^2 \leq  R_{A,\gamma}(t). \]
 These imply the desired estimates. 
\end{proof}
\begin{lemma} \label{L4.8}
Suppose that parameters and initial datum satisfy the following conditions:
\[ 0<\beta\leq 2, \quad \cM_q(f_0)<\infty, \quad \mbox{for some $q > 1$}, \]
and let $f$ be a global Lagrangian weak solution to \eqref{A-4}. Then, the following assertions hold. \newline
\begin{enumerate}
\item
(Localized communication):~There exist $t_1 = t_1(A, \gamma) \geq1$ and $b_{A,\beta}>0$ such that if 
\[ (x,v),(x_*,v_*)\in \Omega_{A,\gamma}(t), \]
then 
\begin{equation*} \label{D-40}
\phi(|x-x_*|)  \geq b_{A,\beta}(1+t)^{-\frac{\gamma\beta}{2}}, \quad t \geq t_1, \quad 
b_{A, \beta} := \underbar{c}_{\phi} \Big( \frac{16A}{\underbar{c}_W} \Big)^{-\frac{\beta}{2}} . \end{equation*}
\item
(Exterior two-point correlation):~
\begin{equation*}\label{D-41}
\int_{(\Omega_{A,\gamma}\times \Omega_{A,\gamma})^c}|v-v_*|^2 f_t(z)f_{t}(z_*)\dd z \dd z_*
\leq  8{\mathfrak M}_{A, \gamma}(t).
\end{equation*}
\end{enumerate}
\end{lemma}
\begin{proof}
(1)~Suppose that 
 \[ (x,v),(x_*,v_*)\in \Omega_{A,\gamma}(t). \]
 Then, by Lemma \ref{L4.7}, we have
 \[
\max \{ |x|,~|x_*| \} \leq\sqrt{\frac{R_{A,\gamma}(t)}{\underbar{c}_W}}, \quad \max \{|v|, |v_*| \} \leq\sqrt{2R_{A,\gamma}(t)}.
\]
Since $A(1 + t)^{\gamma} \to \infty$ as $t \to \infty$, we can choose $t_* = t_*(A, \gamma)$ sufficiently large such that 
\[
A(1+t_*)^\gamma \geq 1.
\]
In fact, we can choose 
\[ t_* (A, \gamma) := \max \Big\{ 0,~e^{-\frac{\ln A}{\gamma}} - 1 \Big \}. \]
Then, we have 
\[ A(1+t)^\gamma \geq 1, \quad t \geq t_*(A, \gamma). \]
These imply 
\begin{equation} \label{D-42}
|x-x_*|^2 \leq 2 ( |x|^2 + |x_*|^2 ) \leq  \frac{4 R_{A,\gamma}(t)}{\underbar{c}_W} = \frac{4(1+A(1+t)^\gamma)}{\underbar{c}_W} \leq  \frac{8A(1+t)^\gamma}{\underbar{c}_W}, \quad t \geq  t_*(A, \gamma).
\end{equation}
We substitute the estimate \eqref{D-42} into \eqref{C-2} to obtain the desired estimate:
\begin{equation} \label{D-43}
\phi(|x-x_*|) \geq  \frac{\underbar{c}_\phi}{(1+ |x- x_*|^2)^{\beta/2}} \geq  \underbar{c}_{\phi} \Big(  1+   \frac{8A(1+t)^\gamma}{\underbar{c}_W}   \Big)^{-\frac{\beta}{2}}.
\end{equation}
We choose $t_{**}$ sufficiently large such that 
\begin{equation} \label{D-44}
 \frac{8A(1+t)^\gamma}{\underbar{c}_W}  \geq 1, \quad t \geq  t_{**} := \max \Big \{0,~e^{\frac{1}{\gamma} \ln \frac{\underbar{c}_W}{8A}} - 1 \Big \}.
\end{equation}
Finally, we set 
\begin{equation} \label{D-45}
 t_1 := \max \{1, t_*, t_{**} \}
 \end{equation}
and combine \eqref{D-43}, \eqref{D-44} and \eqref{D-45} to get 
\[
\phi(|x-x_*|) \geq \underbar{c}_{\phi} \Big( \frac{16A}{\underbar{c}_W} \Big)^{-\frac{\beta}{2}} (1 + t)^{-\frac{\gamma \beta}{2}} =:  b_{A,\beta}(1+t)^{-\frac{\gamma\beta}{2}}, \quad t \geq t_1.
\]
\vspace{.2cm}

\noindent (2)~ For simplicity, we suppress the dependence on \(A\) and \(\gamma\), and write
\[
\Omega:=\Omega_{A,\gamma}(t),
\qquad
R(t):=R_{A,\gamma}(t),
\]
and we set
\begin{align*}
\begin{aligned}
& m_{\rm out}(t):=\int_{\Omega^c}f_t(z)\dd z, \quad
K_{\rm out}(t):=\int_{\Omega^c}|v|^2f_t(z)\dd z, \\
& K_{\rm in}(t):=\int_{\Omega}|v|^2f_t(z)\dd z, \quad 
H_{\rm out}(t):=\int_{\Omega^c}h_t(z)f_t(z)\dd z.
\end{aligned}
\end{align*}
By symmetry and the inclusion
\[
(\Omega\times\Omega)^c
\subset
(\Omega^c\times\bbr^{2d})
\cup
(\bbr^{2d}\times\Omega^c),
\]
we have
\begin{align*}
	\int_{(\Omega\times\Omega)^c}
	|v-v_*|^2f_t(z)f_t(z_*)\dd z\dd z_*\leq
	2\int_{\Omega^c\times\bbr^{2d}}
	|v-v_*|^2f_t(z)f_t(z_*)\dd z\dd z_*.
\end{align*}
Using the mass and momentum normalizations \eqref{C-3} we obtain
\begin{align}
	&\int_{\Omega^c\times\bbr^{2d}}
	|v-v_*|^2f_t(z)f_t(z_*)\dd z\dd z_*
	\nonumber\\
	&\qquad=
	\int_{\Omega^c}|v|^2f_t(z)\dd z
	+
	\left(\int_{\Omega^c}f_t(z)\dd z\right)
	\left(\int_{\bbr^{2d}}|v_*|^2f_t(z_*)\dd z_*\right)
	\nonumber\\
	&\qquad=
	K_{\rm out}(t)+m_{\rm out}(t)\cE_K(t).
	\label{D-46}
\end{align}
Since
\[
h_t(z)=
\frac12|v|^2+(W*\rho_t)(x)
\geq\frac12|v|^2,
\]
we have
\[
K_{\rm out}(t)\leq2H_{\rm out}(t).
\]
Moreover, since \(h_t(z)\geq R(t)\) on \(\Omega^c\),
\[
R(t)m_{\rm out}(t)\leq H_{\rm out}(t).
\]
On the other hand, Lemma~\ref{L4.7} implies
\[
|v|^2\leq2R(t)
\qquad\text{on }\Omega,
\]
and hence
\[
K_{\rm in}(t)
\leq
2R(t)\bigl(1-m_{\rm out}(t)\bigr).
\]
Therefore,
\begin{align*}
	m_{\rm out}(t)\cE_K(t)
	&=
	m_{\rm out}(t)
	\bigl(K_{\rm in}(t)+K_{\rm out}(t)\bigr)\\
	&\leq
	2R(t)m_{\rm out}(t)
	\bigl(1-m_{\rm out}(t)\bigr)
	+
	2m_{\rm out}(t)H_{\rm out}(t)\\
	&\leq
	2\bigl(1-m_{\rm out}(t)\bigr)H_{\rm out}(t)
	+
	2m_{\rm out}(t)H_{\rm out}(t)\\
	&=
	2H_{\rm out}(t).
\end{align*}
Combining this estimate with \eqref{D-46}, we conclude that
\begin{align*}
\begin{aligned}
	&\int_{(\Omega\times\Omega)^c}
	|v-v_*|^2f_t(z)f_t(z_*)\dd z\dd z_* \leq
	2K_{\rm out}(t)
	+
	2m_{\rm out}(t)\cE_K(t)\\
	&\quad\leq
	4H_{\rm out}(t)+4H_{\rm out}(t)=
	8H_{\rm out}(t) \leq
	8\int_{\Omega^c}
	\bigl(1+h_t(z)\bigr)f_t(z)\dd z=
	8\mathfrak M_{A,\gamma}(t).
\end{aligned}
\end{align*}
\end{proof}
\begin{corollary}[Sublinear growth along characteristics]
	\label{C4.1}
	Let \(q>1\), and assume that
	\[
	\cM_q(f_0)<\infty.
	\]
	Let \(f\) be a global Lagrangian weak solution to \eqref{A-4}
	such that
	\begin{equation}\label{D-47}
		h_t\bigl(Z_t(z)\bigr)
		\leq h_0(z)+r(t)
		\qquad
		\text{for \(f_0\)-a.e. }z\in\bbr^{2d} \qquad \mbox{and} \quad  \lim_{t\to\infty}\frac{r(t)}{1+t}=0.
	\end{equation}
	Then, for every fixed \(A>0\) and every \(\gamma\geq1\), there exist
	\(t_{A,\gamma}>0\) and \(\tilde{U}_A>0\) such that
	\begin{equation}\label{D-49}
		\mathfrak M_{A,\gamma}(t)
		\leq
		\tilde{U}_A(1+t)^{-\gamma(q-1)},
		\qquad
		t\geq t_{A,\gamma}.
	\end{equation}
	More precisely, one may take
	\begin{equation}\label{D-50}
		\tilde{U}_A=
		2\left(\frac{A}{2}\right)^{-(q-1)}
		\cM_q(f_0)=
		2^qA^{-(q-1)}\cM_q(f_0).
	\end{equation}
\end{corollary}
\begin{proof}
	Since
	\[
	f_t=(Z_t)_{\#}f_0,
	\]
	we have
	\begin{align*}
		\mathfrak M_{A,\gamma}(t)
		=
		\int_{\bbr^{2d}}
		\Bigl(1+h_t\bigl(Z_t(z)\bigr)\Bigr)
		\mathbf{1}_{\{
				h_t(Z_t(z))>R_{A,\gamma}(t)
		\}}
		f_0(z)\dd z.
	\end{align*}
	By \eqref{D-47}, we have
	\[
	h_t\bigl(Z_t(z)\bigr)
	\leq h_0(z)+r(t).
	\]
	Consequently,
	\[
	h_t\bigl(Z_t(z)\bigr)>R_{A,\gamma}(t)
	\quad\Longrightarrow\quad
	h_0(z)>\widehat R_{A,\gamma}(t),
	\]
	where
	\[
	\widehat R_{A,\gamma}(t)
	:=
	R_{A,\gamma}(t)-r(t)=
	1+A(1+t)^\gamma-r(t).
\]
	Therefore,
	\begin{equation}\label{D-51}
		\mathfrak M_{A,\gamma}(t)
		\leq
		\int_{{h_0>\widehat R_{A,\gamma}(t)}}
		\bigl(1+h_0(z)+r(t)\bigr)f_0(z)\dd z.
	\end{equation}
	Since \(r(t)=o(t)\), we have, for every fixed \(\gamma\geq1\),
	\[
	\frac{r(t)}{(1+t)^\gamma}=
	\frac{r(t)}{1+t}(1+t)^{1-\gamma}
	\longrightarrow0
	\qquad\text{as }t\to\infty.
	\]
	Hence, for every fixed \(A>0\), there exists \(t_{A,\gamma}>0\)
	such that
	\begin{equation}\label{D-52}
		r(t)
		\leq
		\frac{A}{2}(1+t)^\gamma,
		\qquad
		t\geq t_{A,\gamma}.
	\end{equation}
	It follows that
	\begin{equation}\label{D-53}
		\widehat R_{A,\gamma}(t)
		\geq
		1+\frac{A}{2}(1+t)^\gamma
		\geq
		\frac{A}{2}(1+t)^\gamma,
		\qquad
		t\geq t_{A,\gamma}.
	\end{equation}
	We split the right-hand side of
	\eqref{D-51} as
	\begin{align*}
		\mathfrak M_{A,\gamma}(t)
		\leq
		\int_{{h_0>\widehat R_{A,\gamma}(t)}}
		(1+h_0(z))f_0(z)\dd z+
		r(t)
		\int_{{h_0>\widehat R_{A,\gamma}(t)}}
		f_0(z)\dd z.
	\end{align*}
	Using \(q>1\) and Chebyshev's inequality, we obtain
	\begin{align*}
		\int_{{h_0>\widehat R_{A,\gamma}(t)}}
		(1+h_0(z))f_0(z)\dd z
		&\leq
		\bigl(1+\widehat R_{A,\gamma}(t)\bigr)^{-(q-1)}
		\cM_q(f_0),\\
		\int_{{h_0>\widehat R_{A,\gamma}(t)}}
		f_0(z)\dd z
		&\leq
		\bigl(1+\widehat R_{A,\gamma}(t)\bigr)^{-q}
		\cM_q(f_0).
	\end{align*}
	Thus,
	\begin{align*}
		\mathfrak M_{A,\gamma}(t)
		&\leq
		\left[
		\bigl(1+\widehat R_{A,\gamma}(t)\bigr)^{-(q-1)}
		+
		r(t)
		\bigl(1+\widehat R_{A,\gamma}(t)\bigr)^{-q}
		\right]
		\cM_q(f_0).
	\end{align*}
	By \eqref{D-52} and
	\eqref{D-53},
	\[
	r(t)
	\leq
	\frac{A}{2}(1+t)^\gamma
	\leq
	\widehat R_{A,\gamma}(t)
	\leq
	1+\widehat R_{A,\gamma}(t).
	\]
	Therefore,
	\begin{align*}
		\mathfrak M_{A,\gamma}(t)
		\leq
		2\bigl(1+\widehat R_{A,\gamma}(t)\bigr)^{-(q-1)}
		\cM_q(f_0)\leq
		2\left(\frac{A}{2}\right)^{-(q-1)}
		\cM_q(f_0)
		(1+t)^{-\gamma(q-1)}.
	\end{align*}
	This yields the desired estimate \eqref{D-49}.
\end{proof}

Now, we need to estimate  the term $\Lambda$ using Lemma \ref{L4.6} -- Lemma \ref{L4.8}.
\begin{proposition}[Localized dissipation]\label{P4.2}
Suppose that parameters and initial datum satisfy the following conditions:
\[ 0<\beta\leq 2, \quad \cM_q(f_0)<\infty, \quad \mbox{for some $q > 1$}, \]
and let $f$ be a global Lagrangian weak solution to \eqref{A-4}.  Then, there exists some large ${\tilde t}_1$ such that 
\begin{equation*}\label{D-54}
\Lambda(t) \geq  2 \kappa b_{A,\beta}(1+t)^{-\frac{\gamma\beta}{2}} \cE_K(t) - 8 U_A \kappa  b_{A,\beta} (1 + t)^{-\gamma (  \frac{\beta}{2} + q-1)}, \quad t  \geq {\tilde t}_1.
\end{equation*}
\end{proposition}
\begin{proof}
We use Lemma \ref{L4.6}, Lemma \ref{L4.7}, Lemma \ref{L4.8}, and normalizations to get the desired estimate:
\begin{align*}
\Lambda(t) &= \kappa \int_{\bbr^{4d}}\phi(|x-x_*|)|v-v_*|^2f_t(z)f_{t}(z_*)\dd z \dd z_* \\
&\geq  \kappa  \int_{\Omega_{A, \gamma}^2}\phi(|x-x_*|)|v-v_*|^2f_t(z)f_{t}(z_*)\dd z \dd z_* \\
&\geq  \kappa  b_{A,\beta}(1+t)^{-\frac{\gamma\beta}{2}}
\int_{ \Omega_{A, \gamma}^2  }|v-v_*|^2f_t(z)f_{t}(z_*)\dd z \dd z_* \\
&=\kappa b_{A,\beta}(1+t)^{-\frac{\gamma\beta}{2}}
\Bigg(
\int_{\bbr^{4d}}|v-v_*|^2f_t(z)f_t(z_*)\dd z\dd z_*
\\[-1mm]
&\hspace{4.8cm}
-\int_{(\Omega_{A,\gamma}(t)^2)^c}|v-v_*|^2f_t(z)f_t(z_*)\dd z\dd z_*
\Bigg)\\
&=\kappa b_{A,\beta}(1+t)^{-\frac{\gamma\beta}{2}}
\Bigg(
2\int_{\bbr^{2d}}|v|^2f_t(z)\dd z
\\[-1mm]
&\hspace{4.8cm}
-\int_{(\Omega_{A,\gamma}(t)^2)^c}|v-v_*|^2f_t(z)f_t(z_*)\dd z\dd z_*
\Bigg)\\
&= \kappa   b_{A,\beta}(1+t)^{-\frac{\gamma\beta}{2}}
\left(2\cE_K(t)-
\int_{(\Omega_{A,\gamma}(t)^2)^c}|v-v_*|^2f_t(z)f_{t}(z_*)\dd z \dd z_* \right) \\
&\geq  \kappa   b_{A,\beta}(1+t)^{-\frac{\gamma\beta}{2}}
\Big ( 2 \cE_K(t)-8  \mathfrak{M}_{A,\gamma}(t) \Big) \\
& \geq  \kappa  b_{A,\beta}(1+t)^{-\frac{\gamma\beta}{2}}
\Big ( 2\cE_K(t)-8  U_A(1+t)^{-\gamma(q-1)}\Big)  \\
& = 2 \kappa  b_{A,\beta}(1+t)^{-\frac{\gamma\beta}{2}} \cE_K(t) - 8 U_A  \kappa  b_{A,\beta} (1 + t)^{-\gamma (  \frac{\beta}{2} + q-1      )}.
\end{align*}
\end{proof}

\subsubsection{Macroscopic hypocoercivity}
\label{sec:4.2.3}
In this part, we move on to \textbf{Step E} in the proof of
Theorem~\ref{T3.3} \textnormal{(ii)}. Set
\[
p:=\frac{\gamma\beta}{2},
\]
and let \(0<\eta\leq1\) be a gauge parameter. We define
\begin{equation*}\label{D-55}
	\omega(t)
	:=
	\frac{\eta\kappa b_{A,\beta}}{2}(1+t)^{-p},
	\qquad
	\cE(t)
	:=
	\cE_K(t)+\cE_{I,e}(t)+\omega(t)\cE_{I,s}(t).
\end{equation*}
Then we have
\begin{equation*}\label{D-56}
	0\leq\omega(t)
	\leq
	\frac{\eta\kappa b_{A,\beta}}{2},
	\qquad
	\omega'(t)=
	-\frac{\eta\kappa b_{A,\beta}p}{2}
	(1+t)^{-p-1}.
\end{equation*}
We next establish the equivalence between \(\cE\) and \(\cF\), and
derive a differential inequality for \(\cE\).
\begin{lemma}[Gauge-dependent Lyapunov inequality]
	\label{L4.9}
	Set
	\begin{equation*}\label{D-57}
		p:=\frac{\gamma\beta}{2},
		\quad
		\omega_0:=
		\frac{\eta\kappa b_{A,\beta}}{2}.
	\end{equation*}
	Assume that \(0<\eta\leq1\) is sufficiently small so that $\max\{\omega_0,\frac{1}{2\underbar{c}_W}\omega_0\}<\frac{1}{2}.$	Define
	\begin{equation*}\label{D-58}
		c_{L,\eta}
		:=
			\min\left\{1-\omega_0,1-\frac{1}{2\underbar{c}_W}\omega_0\right\},
		\quad
		C_{L,\eta}
		:=
		\max\left\{
			1+\omega_0,
			1+	\frac{\kappa\bar{c}_{\phi}+1}{2\underbar{c}_W} ~\omega_0
			\right\}.
	\end{equation*}
	Moreover, there exists \(\widetilde t_*>0\) such that the following
	assertions hold:
	\begin{enumerate}
		\item
		For every \(t\geq0\),
		\begin{equation}\label{D-59}
			\frac{
				\min\left\{
					2,\frac{1}{\bar{c}_W}
					\right\}
			}{
				C_{L,\eta}
			}
			\cE(t)
			\leq
			\cF(t)
			\leq
			\frac{
				\max\left\{
					2,\frac{1}{\underbar{c}_W}
					\right\}
			}{
				c_{L,\eta}
			}
			\cE(t).
		\end{equation}

		\item
		For every \(t\geq\widetilde t_*\),
		\begin{equation}\label{D-60}
			\cE'(t)
			\leq
			-\kappa\mu_{A} 
			(1+t)^{-\frac{\gamma\beta}{2}}
			\cE(t)
			+
			8U_A\kappa b_{A,\beta}
			(1+t)^{-\gamma\left(\frac{\beta}{2}+q-1\right)},
		\end{equation}
		where
		\begin{equation*}\label{D-61}
			\mu_{A} 
			:=
			\frac{
				b_{A,\beta}
				\min\left\{
				\frac12,\frac{\eta\nu_W}{4}
				\right\}
			}{	\max\left\{
				1+\omega_0,
				1+	\frac{\kappa\bar{c}_{\phi}+1}{2\underbar{c}_W} ~\omega_0
				\right\}
			}.
		\end{equation*}
	\end{enumerate}
\end{lemma}

\begin{proof}
	(1) For convenience, set
	\[
	\cB(t):=\cE_K(t)+\cE_{I,e}(t).
	\]
We use a similar argument as in Lemma \ref{L4.1} to find
	\begin{equation}\label{D-62}
		-\frac{1}{2\underbar{c}_W}\cE_{I,e}(t)-\cE_K(t)
		\leq
		\cE_{I,s}(t)
		\leq
			\frac{\kappa\bar{c}_{\phi}+1}{2\underbar{c}_W}\cE_{I,e}(t)+\cE_K(t).
	\end{equation}
	Using the lower bound in
	\eqref{D-62}, we obtain
	\begin{align*}
		\cE(t)
		&=
		\cE_K(t)+\cE_{I,e}(t)
		+\omega(t)\cE_{I,s}(t)\\
		&\geq
		\cE_K(t)+\cE_{I,e}(t)
		+\omega(t)
		\left(
		-\frac{1}{2\underbar{c}_W}\cE_{I,e}(t)-\cE_K(t)
		\right)\\
		&=
		\bigl(1-\omega(t)\bigr)\cE_K(t)
		+
		\Bigg(1-\frac{1}{2\underbar{c}_W}\omega(t)\Bigg)\cE_{I,e}(t).
	\end{align*}
	Since $
	0\leq\omega(t)\leq\omega_0$
	we have
$
	1-\omega(t)\geq1-\omega_0.
$
	Consequently, we have
	\begin{equation}\label{D-63}
		\cE(t)
		\geq
		c_{L,\eta}
		\left(
		\cE_K(t)+\cE_{I,e}(t)
		\right)
		=
		c_{L,\eta}\cB(t).
	\end{equation}
	Condition $\max\{\omega_0,\frac{1}{2\underbar{c}_W}\omega_0\}<\frac{1}{2}$ ensures that
	\[
		c_{L,\eta}=	\min\left\{1-\omega_0,1-\frac{1}{2\underbar{c}_W}\omega_0\right\}\geq\frac12.
	\]
	Similarly, using the upper bound in \eqref{D-62}, we obtain
	\begin{align*}
		\cE(t)
		&\leq
		\cE_K(t)+\cE_{I,e}(t)
		+\omega(t)
		\left(
			\frac{\kappa\bar{c}_{\phi}+1}{2\underbar{c}_W}\cE_{I,e}(t)+\cE_K(t)
		\right)\\
		&=
		\bigl(1+\omega(t)\bigr)\cE_K(t)
		+
		\Bigg(1+	\frac{\kappa\bar{c}_{\phi}+1}{2\underbar{c}_W}\omega(t)\Bigg)\cE_{I,e}(t)\leq
		C_{L,\eta}
		\left(
		\cE_K(t)+\cE_{I,e}(t)
		\right).
	\end{align*}
	Therefore,
	\begin{equation}\label{D-64}
		c_{L,\eta}\cB(t)
		\leq
		\cE(t)
		\leq
		C_{L,\eta}\cB(t),
		\qquad
		t\geq0.
	\end{equation}
	On the other hand, Lemma~\ref{L2.2} gives
	\begin{equation*}\label{D-65}
		\min\left\{
			2,\frac{1}{\bar{c}_W}
			\right\}
		\cB(t)
		\leq
		\cF(t)
		\leq
		\max\left\{
			2,\frac{1}{\underbar{c}_W}
			\right\}
		\cB(t).
	\end{equation*}
Then, we obtain
	\[
	\frac{
		\min\left\{
			2,\frac{1}{\bar{c}_W}
			\right\}
	}{
		C_{L,\eta}
	}
	\cE(t)
	\leq
	\cF(t)
	\leq
	\frac{
		\max\left\{
			2,\frac{1}{\underbar{c}_W}
			\right\}
	}{
		c_{L,\eta}
	}
	\cE(t),
	\]
	which proves \eqref{D-59}.\newline 
	
	\noindent (2) We next derive the differential inequality. By Remark~\ref{R2.1} (1), Lemma~\ref{L2.3} (iii), the virial condition \eqref{C-1}$_3$, and Proposition~\ref{P4.2}, we have
	\begin{align}
		\cE'(t)
		\leq
		-\Lambda(t)
		+\omega'(t)\cE_{I,s}(t)
		+\omega(t)
		\left(
		2\cE_K(t)-\nu_W\cE_{I,e}(t)
		\right).
		\label{D-66}
	\end{align}
	Since
	\[
	\omega'(t)
	=
	-\frac{\eta\kappa b_{A,\beta}p}{2}
	(1+t)^{-p-1}
	\leq0,
	\]
	we use the lower bound for \(\cE_{I,s}\) in
	\eqref{D-62}. Hence,
	\begin{align}\label{D-67}
		\begin{aligned}
		\omega'(t)\cE_{I,s}(t)
		&\leq
		\omega'(t)
		\left(
		-\frac{1}{2\underbar{c}_W}\cE_{I,e}(t)-\cE_K(t)
		\right)
	\\
		&=
		\frac{\eta\kappa b_{A,\beta}p}{2}
		(1+t)^{-p-1}\cE_K(t)
		+
		\frac{\eta\kappa b_{A,\beta}p}{4\underbar{c}_W}
		(1+t)^{-p-1}\cE_{I,e}(t).	
		\end{aligned}
	\end{align}
	Moreover, Proposition~\ref{P4.2} implies
	\begin{align}
		-\Lambda(t)
		\leq
		-2\kappa b_{A,\beta}
		(1+t)^{-p}\cE_K(t)
		+
		8U_A\kappa b_{A,\beta}
		(1+t)^{-\gamma\left(\frac{\beta}{2}+q-1\right)}.
		\label{D-68}
	\end{align}
	Using
\[
	\omega(t)
	=
	\frac{\eta\kappa b_{A,\beta}}{2}
	(1+t)^{-p},
\]
	and substituting \eqref{D-67} and \eqref{D-68}
	into \eqref{D-66}, we obtain
	\begin{align*}\label{D-69}
		\begin{aligned}
		\cE'(t)
		&\leq
		-\kappa b_{A,\beta}(1+t)^{-p}
		\left(
		2-\eta-\frac{\eta p}{2(1+t)}
		\right)
		\cE_K(t)\\
		&\quad
		-
		\frac{\eta\kappa b_{A,\beta}}{2}
		(1+t)^{-p}
		\left(
		\nu_W-\frac{p}{2\underbar{c}_W(1+t)}
		\right)
		\cE_{I,e}(t)+
		8U_A\kappa b_{A,\beta}
		(1+t)^{-\gamma\left(\frac{\beta}{2}+q-1\right)}.
			\end{aligned}
	\end{align*}
	We choose \(\widetilde t_*>0\) sufficiently large so that
	Proposition~\ref{P4.2} is applicable and
	\begin{equation*}\label{D-70}
		\frac{\eta p}{2(1+t)}
		\leq\frac12,
		\qquad
		\frac{p}{2\underbar{c}_W(1+t)}
		\leq\frac{\nu_W}{2},
		\qquad
		t\geq\widetilde t_*.
	\end{equation*}
	Since \(0<\eta\leq1\), the first condition yields
	\[
	2-\eta-\frac{\eta p}{2(1+t)}
	\geq
	2-1-\frac12
	=
	\frac12.
	\]
	The second condition gives
	\[
	\nu_W-	\frac{p}{2\underbar{c}_W(1+t)}
	\geq
	\frac{\nu_W}{2}.
	\]
	Therefore, for \(t\geq\widetilde t_*\),
	\begin{align*}
		\cE'(t)
		&\leq
		-\kappa b_{A,\beta}(1+t)^{-p}
		\left(
		\frac12\cE_K(t)
		+
		\frac{\eta\nu_W}{4}\cE_{I,e}(t)
		\right)
		\nonumber+
		8U_A\kappa b_{A,\beta}
		(1+t)^{-\gamma\left(\frac{\beta}{2}+q-1\right)}
		\nonumber\\
		&\leq
		-\kappa b_{A,\beta}
		\min\left\{
			\frac12,\frac{\eta\nu_W}{4}
			\right\}
		(1+t)^{-p}\cB(t)
	+
		8U_A\kappa b_{A,\beta}
		(1+t)^{-\gamma\left(\frac{\beta}{2}+q-1\right)}.
		\label{D-71}
	\end{align*}
	By the upper bound in \eqref{D-64},
	\[
	\cB(t)
	\geq
	\frac{1}{C_{L,\eta}}\cE(t).
	\]
	Consequently, we have
	\begin{align*}\label{D-72}
		\cE'(t)
		&\leq
		-\frac{
			\kappa b_{A,\beta}
			\min\left\{
				\frac12,\frac{\eta\nu_W}{4}
				\right\}
		}{
			C_{L,\eta}
		}
		(1+t)^{-p}\cE(t)
	+
		8U_A\kappa b_{A,\beta}
		(1+t)^{-\gamma\left(\frac{\beta}{2}+q-1\right)}
		\nonumber\\
		&=:
		-\kappa\mu_{A} 
		(1+t)^{-\frac{\gamma\beta}{2}}
		\cE(t)
		+
		8U_A\kappa b_{A,\beta}
		(1+t)^{-\gamma\left(\frac{\beta}{2}+q-1\right)}.
	\end{align*}
	This establishes \eqref{D-60}.
	
\end{proof}

\begin{remark}\label{R4.1}
	Assume that
	\[
	0<\beta\le2,
	\qquad
	\gamma=\frac{2}{\beta}.
	\]
	Then
	\[
	p=\frac{\gamma\beta}{2}=1,
	\qquad
	\omega(t)
	=
	\frac{\eta\kappa b_{A,\beta}}{2(1+t)}.
	\]
	For \(t\geq0\), define the time-dependent effective damping
	coefficient by
	\begin{equation}\label{D-73}
		\kappa\mu_{A} (t)
		:=
		\frac{
			\kappa b_{A,\beta}
			\min\left\{
				\frac12,\frac{\eta\nu_W}{4}
				\right\}
		}{
			\max\left\{
				1+\omega(t),
				1+
				\frac{\kappa\bar c_\phi+1}
				{2\underbar c_W}
				\omega(t)
				\right\}	}.
	\end{equation}
We set
	\[
	m_\eta
	:=
	\min\left\{
	\frac12,\frac{\eta\nu_W}{4}
	\right\},
	\qquad
	r_*:=\frac{2(q-1)}{\beta},
	\]
	and
	\[
	M_\kappa
	:=
	\max\left\{
	1,
	\frac{\kappa\bar c_\phi+1}
	{2\underbar c_W}
	\right\}.
	\]
	Then
	\[
	\max\left\{
	1+\omega(t),
	1+
	\frac{\kappa\bar c_\phi+1}
	{2\underbar c_W}
	\omega(t)
	\right\}
	=
	1+M_\kappa\omega(t),
	\]
	and hence
	\begin{equation}\label{D-74}
		\kappa\mu_{A} (t)
		=
		\frac{
			\kappa b_{A,\beta}m_\eta
		}{
			1+M_\kappa\omega(t)
		}.
	\end{equation}
	We now choose \(A>0\) explicitly as
	\begin{equation}\label{D-75}
		A
		:=
		\frac{\underbar c_W}{16}
		\left(
		\frac{
			\kappa\underbar c_\phi m_\eta
		}{
			2r_*
		}
		\right)^{\frac{2}{\beta}}
		=
		\frac{\underbar c_W}{16}
		\left(
		\frac{
			\beta\kappa\underbar c_\phi \min\left\{
			\frac12,\frac{\eta\nu_W}{4}
			\right\}
		}{
			4(q-1)
		}
		\right)^{\frac{2}{\beta}}.
	\end{equation}
	Since
	\[
	b_{A,\beta}
	=
	\underbar c_\phi
	\left(
	\frac{16A}{\underbar c_W}
	\right)^{-\frac{\beta}{2}},
	\]
	the choice \eqref{D-75} yields
	\begin{equation}\label{D-76}
		\kappa b_{A,\beta}m_\eta
		=
		\kappa\underbar c_\phi m_\eta
		\left(
		\frac{
			\kappa\underbar c_\phi m_\eta
		}{
			2r_*
		}
		\right)^{-1} =
		2r_*
		=
		\frac{4(q-1)}{\beta}.
	\end{equation}
	Because
	\[
	\omega(t)
	=
	\frac{\eta\kappa b_{A,\beta}}{2(1+t)}
	\longrightarrow0
	\qquad
	\text{as }t\to\infty,
	\]
	we may choose \(T_A\geq0\) such that
	\begin{equation}\label{D-77}
		T_A
		\geq
		\max\left\{
		\eta\kappa b_{A,\beta}M_\kappa-1,	\frac{\eta\kappa b_{A,\beta}M_\kappa}{2\underbar{c}_W}-1,
		0
		\right\}.
	\end{equation}
	Then, for every \(t\geq T_A\),
	\begin{equation}\label{D-78}
		\omega(t)
		=
		\frac{\eta\kappa b_{A,\beta}}{2(1+t)}
		\leq
		\frac{1}{2M_\kappa}
		\leq
		\frac12\quad \text{and} \quad 	\frac{1}{2\underbar{c}_W}\omega(t)\le \frac{1}{2}.
	\end{equation}
	In particular,
	\begin{equation*}\label{D-79}
	\min\left\{1-\omega(t),1-\frac{1}{2\underbar{c}_W}\omega(t)\right\}\geq\frac12,
		\qquad
		t\geq T_A.
	\end{equation*}
	Moreover, \eqref{D-78} implies
	\[
	1+M_\kappa\omega(t)
	\leq
	\frac32.
	\]
	Thus, by \eqref{D-74} and \eqref{D-76},
	\begin{equation*}\label{D-80}
		\kappa\mu_{A} (t)
		=
		\frac{
			\kappa b_{A,\beta}m_\eta
		}{
			1+M_\kappa\omega(t)
		}
		\geq
		\frac{
			2r_*
		}{
			\frac32
		}
		=
		\frac43r_*
		>
		r_*
		=
		\frac{2(q-1)}{\beta},
		\quad
		t\geq T_A.
	\end{equation*}
	Consequently, by first choosing \(A\) as in
	\eqref{D-75} and then choosing \(T_A\) as in
	\eqref{D-77}, we simultaneously obtain
	\[
	\min\left\{1-\omega(t),1-\frac{1}{2\underbar{c}_W}\omega(t)\right\}\geq\frac12
	\]
	and
	\[
	\kappa\mu_{A} (t)
	>
	\frac{2(q-1)}{\beta},
	\qquad
	t\geq T_A.
	\]
\end{remark}

\subsubsection{Proof of the second assertion} \label{sec:4.2.4}
Now, we are ready to prove the second assertion in Theorem \ref{T3.3}. Suppose that $\beta$ and initial datum $f_0$ satisfy 
\[ 0<\beta \leq 2, \quad \cM_q(f_0)<\infty \quad \mbox{for some $q>1$}. \]
It follows from Lemma \ref{L4.9} that 
\begin{equation}
\begin{cases} \label{D-81}
\displaystyle 	\frac{
	\min\left\{
	2,\frac{1}{\bar{c}_W}
	\right\}
}{
	C_{L,\eta}
}
\cE(t)
\leq
\cF(t)
\leq
\frac{
	\max\left\{
	2,\frac{1}{\underbar{c}_W}
	\right\}
}{
	c_{L,\eta}
}
\cE(t), \quad t \geq 0, \vspace{8pt}\\
\displaystyle \cE^{\prime}(t) \leq -\kappa \mu_{A}  (1+t)^{-\frac{\gamma\beta}{2}} \cE(t) +  8 U_A   \kappa  b_{A,\beta} (1 + t)^{-\frac{\gamma \beta}{2}}  (1 + t)^{-\gamma(q-1)}. 
\end{cases}
\end{equation}
Next, we consider two cases:
\[ 0 < \beta < 2 \quad \mbox{and} \quad \beta = 2. \]
\noindent $\bullet$~Case 1 $(0 < \beta < 2)$:~We first choose $ \gamma=1,~A=2C_*+2$. In order to apply Lemma \ref{L2.5} to $\eqref{D-81}_2$, we set 
\[
c := \kappa \mu_{A}, \quad \alpha := \frac{\gamma\beta}{2}<1, \quad C :=  8 U_A \kappa  b_{A,\beta}, \quad \lambda := \gamma(q-1)  \]
to get 
\[ 
\cE(t) \leq  {\tilde C}  (1+t)^{- \gamma(q-1)}, \quad t \gg 1.
\]
Note that by Proposition \ref{P4.1}, \eqref{D-63} and the above decay estimates, we have 
\begin{align*}	
	h_t\bigl(Z_t(z)\bigr)
&\leq h_0(z)+\left( \frac{\kappa\bar c_\phi}{4}+C_W \right)\int_0^{t}(\cE_{K}+\cE_{I,e})(s)\dd s\\ &\le	h_0(z)+\frac{\left( \frac{\kappa\bar c_\phi}{4}+C_W \right)}{c_{L,\eta}}\int_0^{t}\cE(s)\dd s\le h_0(z)+o(t).
\end{align*}
Combining this with Corollary \ref{C4.1} shows that, for any sufficiently small $A$, we still have 
\[\cE^{\prime}(t) \leq -\kappa \mu_{A}(t) (1+t)^{-\frac{\gamma\beta}{2}} \cE(t) +  8 	\tilde{U}_A \kappa  b_{A,\beta} (1 + t)^{-\frac{\gamma \beta}{2}}  (1 + t)^{-\gamma(q-1)}.\]
Then, we choose $\gamma=\frac{2}{\beta}$ and \[A:=	\frac{\underbar c_W}{16}
\left(
\frac{
	\beta\kappa\underbar c_\phi \min\left\{
	\frac12,\frac{\eta\nu_W}{4}
	\right\}
}{
	4(q-1)
}
\right)^{\frac{2}{\beta}},
\]
as in Remark \ref{R4.1}. Then we have $\kappa \mu_{A}(t)>\frac{2(q-1)}{\beta}$ for all $t\ge T_A$. For simplicity of notation, we use $\kappa \tilde{\mu}_{A}:=\inf_{t\ge T_A}\kappa \mu_{A}(t)$ to replace $\kappa \mu_{A}(t)$. We redefine \(\omega\) and \(\cE\) using this new choice of \(A\) and \(\gamma\).
Now, we use Lemma \ref{L2.7} (1) and  $\kappa \tilde{\mu}_{A}>\frac{2(q-1)}{\beta}$ to obtain the desired decay rate of $\cE$:
\[
\cE(t) \leq  {\tilde C}  (1+t)^{- \frac{2(q-1)}{\beta}}, \quad t \gg T_A.
\]
Therefore, the flocking functional $\cF$ satisfies 
\begin{equation*} \label{D-82}
 {\mathcal F}(t) \lesssim \cE(t)\lesssim {\tilde C} (1+t)^{- \frac{2(q-1)}{\beta}}, \quad t \gg 1.
\end{equation*}
\vspace{.1cm}

\noindent $\bullet$~Case 2 $(\beta = 2)$:~In this case, again we first choose parameters $A$ and $\gamma$ as 
\[
A=2C_*+2=
2\left(
\frac{\kappa\bar c_\phi}{4}+C_W
\right)
(\mathcal E_K+\mathcal E_{I,e})(0)+2, \quad \gamma=1.
\]
In this case, all parameters appearing in $\eqref{D-81}_2$ become
\begin{align*}
\begin{aligned}
b_{A, 2} &= \underbar{c}_{\phi} \Big( \frac{16A}{\underbar{c}_W} \Big)^{-1}, \quad \mu_A =
\frac{
	b_{A,\beta}
	\min\left\{
	\frac12,\frac{\eta\nu_W}{4}
	\right\}
}{	\max\left\{
	1+\omega_0,
	1+	\frac{\kappa\bar{c}_{\phi}+1}{2\underbar{c}_W} ~\omega_0
	\right\}
}, \quad c_0 = \frac{A}{2}=C_*+1, \\
U_A&= \max \{1, C_* \} (1 + c_0^{-1}) c_0^{-(q-1)}  \cM_q(f_0).
\end{aligned}
\end{align*}
Then, we can rewrite $\eqref{D-81}_2$ as 
\begin{equation} \label{D-83}
\cE^{\prime}(t) \leq - \frac{\kappa \mu_A}{1+t} \cE(t) +  \frac{8 U_A \kappa  b_{A,\beta}}{1 + t}  (1 + t)^{-(q-1)}.
\end{equation}
Now, we apply Lemma \ref{L2.7} (1) for \eqref{D-83} with 
\[ c_1 := \kappa \mu_A, \quad C_1 :=  8 U_A  \kappa  b_{A,\beta}, \quad g(t) : = (1 + t)^{-(q-1)}, \quad \lambda = q-1  \]
to find 
\[
\cE(t) \lesssim \begin{cases}
(1+t)^{-\min\{\kappa \mu_A, q-1 \}}, \quad &\kappa \mu_A \neq q-1,\\
(1+t)^{-(q-1)}\log(2+t), \quad  & \kappa \mu_A= q-1.
\end{cases}
\]
Moreover, we can use the same bootstrap argument in Case 1 to find some sufficiently small  \[A:=	\frac{\underbar c_W}{16}
\left(
\frac{
	\beta\kappa\underbar c_\phi \min\left\{
	\frac12,\frac{\eta\nu_W}{4}
	\right\}
}{
	4(q-1)
}
\right)^{\frac{2}{\beta}}
\] such that $\kappa \tilde{\mu}_A> q-1$ (see Remark \ref{R4.1}) and 
\begin{equation} \label{D-84}
	\cE^{\prime}(t) \leq - \frac{\kappa \tilde{\mu}_A}{1+t} \cE(t) +  \frac{8 \tilde{U}_A   \kappa  b_{A,\beta}}{1 + t}  (1 + t)^{-(q-1)}, \quad t \gg T_A.
\end{equation}
Here, $\tilde{U}_A$ and $\tilde{\mu}_A=\inf_{t\ge T_A}\mu_A(t)$ are defined in \eqref{D-50} and \eqref{D-73}, respectively. Now, we again apply Lemma \ref{L2.7} (1) for \eqref{D-84} with new constants
\[ c_1 := \kappa \tilde{\mu}_A, \quad C_1 :=  8 \tilde{U}_A   \kappa  b_{A,\beta}, \quad g(t) : = (1 + t)^{-(q-1)}, \quad \lambda = q-1  \]
to see
\[
\cE(t) \lesssim 
	(1+t)^{- (q-1)}.
\]
By $\eqref{D-81}_1$ and Remark \ref{R4.1}, the flocking functional $\cF$ satisfies 
\begin{equation*} \label{D-85}
 {\mathcal F}(t) \lesssim \cE(t)\lesssim  (1+t)^{-(q-1 )}.
\end{equation*}
This finishes the proof of  Theorem \ref{T3.3} (ii).

\subsection{Third assertion}\label{sec:4.3}
In this subsection, we derive the optimal weak flocking of the KCS model \eqref{A-4} with exponential mechanical-energy tails. The proof has two stages. First, the time-varying effective-region method gives an integrable preliminary decay. This makes the total characteristic mechanical-energy shift uniformly bounded in time. Second, we use the resulting uniform exponential moment of the \emph{current} mechanical energy to prove a high-energy descent estimate along characteristics. The latter produces a fixed effective region and ultimately restores an exponential hypocoercive relaxation.

\subsubsection{Preparatory lemmas}  \label{sec:4.3.1}
In the sequel, we provide several elementary estimates to be used in the proof of the third assertion. 
\begin{lemma} \label{L4.10}
	Suppose that parameters and initial datum satisfy the following conditions:
	\[ 0<\beta\leq 2, \quad \cM_{e,a}(f_0)<\infty, \]
	and let $f$ be a global Lagrangian weak solution to \eqref{A-4}. Then, the following assertions hold.
\begin{enumerate}
\item
\emph{(Coarse exponential tail)}:~Suppose that parameters and initial datum satisfy the following conditions:
\[  \mbox{Either} \quad \gamma>1 \quad \mbox{or} \quad \gamma=1~~\mbox{and}~~A>C_*, \]
where $C_*$ is defined in \eqref{D-27}. Then, we have
\begin{equation}\label{D-86}
		{\mathfrak M}_{A, \gamma}(t)\leq C\exp\left[-c(1+t)^\gamma\right], \quad t \gg 1.
	\end{equation}
\item
\emph{(Uniform energy shift after integrability)}:~~Suppose that parameters satisfy 
\[
\gamma > 0, \quad A > 0,
\]
and let $f$ be a global Lagrangian weak solution to \eqref{A-4} satisfying a priori condition:
\begin{equation} \label{D-87}
\int_0^\infty (\cE_K + \cE_{I,e})(s) \dd s<\infty, \quad S_\infty:=\bigg(\frac{ \kappa\bar{c}_\phi}{4} + C_W\bigg)\int_0^\infty (\cE_K + \cE_{I,e})(s) \dd s. 
\end{equation}
Then, for $t \geq 0$, we have
\begin{align}
\begin{aligned} \label{D-88}
& (i)~h_t(X(t),V(t))\leq h_0(x,v)+S_\infty. \\
& (ii)~ {\mathfrak M}_{A, \gamma}(t)\leq C\exp\left[-c(1+t)^\gamma\right].
\end{aligned}
\end{align} 
\end{enumerate}	
\end{lemma}
\begin{proof}
\noindent (1)~We set
\[  {\hat R}(t)= R_{A,\gamma}(t)-C_*t\geq c(1+t)^\gamma. \]
Then, the push-forward formula gives
	\[
	{\mathfrak M}_{A, \gamma}(t)
	\leq\int_{\{h_0(z) > {\hat R}(t)\}}(1+h_0(z) +C_*t)f_0(z) \dd z.
	\]
	In the admissible set, $h_0(z)> {\hat R}(t)$ implies 
\[ t\leq C(1+h_0(z)). \]
Hence the prefactor is bounded by $Ce^{ah_0(z)/2}$, while
	\[
	e^{ah_0(z)/2}\mathbf{1}_{\{h_0(z) > {\hat R}(t)\}}
	\leq e^{-a {\hat R}(t)/2}e^{ah_0(z)}.
	\]
	This proves \eqref{D-86} without using any  a priori decay of $\cE_K + \cE_{I,e}$. \newline
	
\noindent (2)~The first statement in \eqref{D-88} follows directly from Proposition \ref{P4.1}. We set
	\[
	{\hat R}_\infty(t):= R_{A,\gamma}(t)-S_\infty.
	\]
	For all sufficiently large $t$, 
\[  {\hat R}_\infty(t)\geq A(1+t)^\gamma/2. \]
If a characteristic lies outside $\Omega_{A,\gamma}(t)$, then  the a priori condition \eqref{D-87} implies 
\[ h_0> {\hat R}_\infty(t). \]
Therefore, we have
	\begin{align*}
		{\mathfrak M}_{A, \gamma}(t)
		&\leq\int_{\{h_0(z) > {\hat R}_\infty(t)\}}(1+h_0(z) +S_\infty)f_0(z) \dd z\\
		&\leq C_{a,S_\infty}e^{-a{\hat R}_\infty(t)/2}
		\int_{\bbr^{2d}} e^{ah_0(z)}f_0(z) \dd z
		\leq C\exp\left[-c(1+t)^\gamma\right].
	\end{align*}
\end{proof}
\begin{lemma}  \label{L4.11} Suppose that parameters and initial datum satisfy the following conditions:
\[  \mbox{Either} \quad \gamma>1 \quad \mbox{or} \quad \gamma=1~~\mbox{and}~~A>C_*; \quad \cM_{e,a}(f_0)<\infty, \]
where $C_*$ is defined in \eqref{D-27}, and let $f$ be a global Lagrangian weak solution to \eqref{A-4}. Then, there exists some large constant ${\tilde t}_1$ such that 
\begin{equation}\label{D-89}
\Lambda(t) \geq   2\kappa b_{A,\beta}(1+t)^{-\frac{\gamma\beta}{2}} \cE_K(t) - C \kappa  b_{A,\beta} (1 + t)^{-\frac{\gamma \beta}{2}} \exp\left[-c(1+t)^\gamma\right], \quad t  \geq {\tilde t}_1.
\end{equation}
\end{lemma}
\begin{proof}
We use the same arguments in Proposition \ref{P4.2} to find
\begin{align*}
\begin{aligned}
\Lambda(t)  &\geq  \kappa   b_{A,\beta}(1+t)^{-\frac{\gamma\beta}{2}}
\Big ( 2 \cE_K(t)-8  \mathfrak{M}_{A,\gamma}(t) \Big) \\
& \geq  \kappa  b_{A,\beta}(1+t)^{-\frac{\gamma\beta}{2}}
\Big ( 2 \cE_K(t)- C\exp\left[-c(1+t)^\gamma\right] \Big)  \\
& =  2\kappa b_{A,\beta}(1+t)^{-\frac{\gamma\beta}{2}} \cE_K(t) - C \kappa  b_{A,\beta} (1 + t)^{-\frac{\gamma \beta}{2}} \exp\left[-c(1+t)^\gamma\right].
\end{aligned}
\end{align*}
\end{proof}
\begin{remark}\label{R4.2}
	Similar to Corollary \ref{C4.1},  if 	\begin{equation*}
		h_t\bigl(Z_t(z)\bigr)
		\leq h_0(z)+C(1+t)^{r}
		\qquad
		\text{for \(f_0\)-a.e. }z\in\bbr^{2d},
	\end{equation*} then for any $\gamma>r>0$ and $A>0$ we can derive the similar estimates as in  \eqref{D-89}.
\end{remark}
\begin{lemma}[Gauge-dependent Lyapunov inequality]\label{L4.12}
	Suppose that parameters and initial datum satisfy the following conditions:
\[ 0<\beta\leq 2, \quad \cM_{e,a}(f_0)<\infty, \]
and let $f$ be a global Lagrangian weak solution to \eqref{A-4}. There exists a positive constant ${\tilde t}_*$ such that 
\begin{align*}
	\cE'(t) \leq  -\kappa \mu_A (1+t)^{-\frac{\gamma\beta}{2}} \cE(t)+  C \kappa  b_{A,\beta} (1 + t)^{-\frac{\gamma \beta}{2}} \exp\left[-c(1+t)^\gamma\right], \quad t \geq {\tilde t}_*. \end{align*}
\end{lemma}
\begin{proof} 
	We use the same arguments in the proof of Lemma \ref{L4.9} together with Lemma \ref{L4.11} to get the desired estimate. 
\end{proof}

\begin{proposition}[Preliminary integrable decay]\label{NP4.3}
Suppose that
\[
0<\beta\leq2,\qquad \cM_{e,a}(f_0)<\infty
\]
for some $a>0$, and let $f$ be a global Lagrangian weak solution to \eqref{A-4}.
Set
\[
\cB(t):=\cE_K(t)+\cE_{I,e}(t).
\]
Then
\begin{equation*}
\cB(t)\longrightarrow0,
\qquad
\int_0^\infty\cB(s)\dd s<\infty.
\end{equation*}
Moreover, there exist constants $C_B\geq1$, $c_B>0$, and $\theta_B>0$, depending only on the data and the structural parameters such that
\begin{equation}\label{ND-60}
\cB(t)\leq C_B\exp\left[-c_B(1+t)^{\theta_B}\right],
\qquad t\geq0.
\end{equation}
\end{proposition}

\begin{proof}
We split the proof into the case $0<\beta<2$ and the case $\beta=2$.\newline 

\noindent $\bullet$ Case 1 ($0<\beta<2$): We first choose the linear gauge:
\[  \gamma=1, \quad  A>C_*. \]
Then, it follows from the first assertion of Lemma \ref{L4.10} that
\[ 	{\mathfrak M}_{A, \gamma}(t)\leq C\exp\left[-ct\right], \quad t \gg 1. \]
On the other hand, it follows from Lemma \ref{L4.12} that 
\begin{equation} \label{D-92}
	\cE'(t) \leq  -\kappa \mu_A  (1+t)^{-\frac{\beta}{2}} \cE +  C \kappa  b_{A,\beta} (1 + t)^{-\frac{\beta}{2}} e^{-ct}, \quad t \geq {\tilde t}_*.
\end{equation}
Now, we set 
\[ c_0 = \kappa \mu_A, \quad \alpha = \frac{\beta}{2}, \quad C_0 =   C \kappa  b_{A,\beta}, \quad \sigma = c, \quad \lambda  = 1,  \]
and apply Lemma \ref{L2.6} for \eqref{D-92} to get the desired estimate:
\begin{equation} \label{D-93}
	\cB(t)=	(\cE_{K}+\cE_{I,e})(t)	\lesssim\cE(t) \leq  {\tilde C}_0\exp\left[-{\tilde c}_0(1+t)^{1-\frac{\beta}{2}} \right], \quad t \gg 1, 
\end{equation}
where we used the relation:
\[ 0 <  1 - \frac{\beta}{2} < 1. \]
In particular, the relation \eqref{D-93} implies the integrability of $\cE$ and $\cE_{K}+\cE_{I,e}$:
\[\int_0^t(\cE_{K}+\cE_{I,e})(s)\dd s\lesssim \int_0^{\infty} \cE(s) \dd s \leq {\tilde C}_0  \int_0^{\infty} \exp\left[-{\tilde c}_0(1+t)^{1-\frac{\beta}{2}} \right] \dd t < \infty. \]
\vspace{0.2cm}
	
\noindent $\bullet$~Case 2 ($\beta=2$):~Assume that $\beta$ and initial datum $f_0$ satisfy
\[  \beta = 2, \quad  \cM_{e,a}(f_0)<\infty. \]
We further split the proof  into two bootstrap steps. \vspace{0.2cm}

\noindent $\diamond$~Step B.1 (Derivation of rough decay estimate):~We first take a coarse linear gauge for $R_{A, \gamma}$:
\[  \gamma = 1, \quad A > C_*. \]
Then, it follows from Lemma \ref{L4.10} that 
\[ {\mathfrak M}_{A, 1}(t)\leq C e^{-c t}, \quad t \gg 1. \]
Note that $\cE$ satisfies 
\begin{equation} \label{D-98}
	\cE'(t) \leq  -\kappa \mu_A  (1+t)^{-1} \cE +  C \kappa  b_{A,\beta} (1 + t)^{-1} e^{-ct}, \quad t \geq {\tilde t}_*. 
\end{equation}
We apply Lemma \ref{L2.7} (2) for \eqref{D-98} to get 
\[
\cE(t)\leq C(1+t)^{-\bar{\mu}_A} \quad \mbox{for some $0<\bar{\mu}_A<\min\{\kappa\mu_A,1\}$}. 
\]
This implies 	
\[ \int_0^t(\cE_{K}+\cE_{I,e})(s)\dd s\lesssim \int_0^t\cE(s)\dd s\le C (1+t)^{1-\bar{\mu}_A}. \]

\noindent $\diamond$~Step B.2 (Derivation of decay estimate):~Similar to Lemma \ref{L4.6}, we fix  $A=2C_*+2>0$ and $\tilde{\gamma}=1-\frac{\bar{\mu}_A}{2}<1$ to see for sufficiently large $t$,
\[
{\mathfrak M}_{A,\tilde{\gamma}}(t)\leq Ce^{-c t^{\tilde{\gamma}}}.
\]
Repeating the proof of Lemma~\ref{L4.12}, now using
Remark~\ref{R4.2} in place of Lemma~\ref{L4.11}, we obtain
\begin{equation*}\label{D-99}
	\cE'(t) \leq  -\kappa \mu_A  (1+t)^{-\frac{\tilde{\gamma}\beta}{2}} \cE +  C \kappa  b_{A,\beta} (1 + t)^{-\frac{\tilde{\gamma}\beta}{2}} e^{-ct^{\tilde{\gamma}}}, \quad t \geq {\tilde t}_*.
\end{equation*}
Since $\frac{\tilde{\gamma}\beta}{2}<1$, we can use Lemma \ref{L2.6} to see
\begin{equation*}
\cB(t)=(\cE_{K}+\cE_{I,e})(t)	\lesssim \cE(t) \leq  {\tilde C}_1\exp\left[-{\tilde c}_1(1+t)^{1-\frac{\tilde{\gamma}\beta}{2}} \right], \quad t \gg 1.
\end{equation*}
This implies the integrability of $\cE$ and $\cE_{K}+\cE_{I,e}$:
\[\int_0^t(\cE_{K}+\cE_{I,e})(s)\dd s\lesssim  \int_0^{\infty} \cE(t) \dd t \leq {\tilde C}_1  \int_0^{\infty} \exp\left[-{\tilde c}_1(1+t)^{1-\frac{\tilde{\gamma}\beta}{2}} \right] \dd t < \infty. \]
In both cases, the displayed decay estimates hold for all sufficiently large $t$. Since $\cB$ is non-increasing by Remark~\ref{R2.1}, we may enlarge the multiplicative constant and obtain \eqref{ND-60} for every $t\geq0$.
\end{proof}
The integrability of $\cB$ allows the initial exponential moment to be propagated uniformly in the \emph{current} mechanical energy.

\begin{lemma}[Uniform current exponential moment]\label{NL4.13}
Define
\begin{equation}\label{ND-64}
S_\infty:=\left(\frac{\kappa\bar c_\phi}{4}+C_W\right)
\int_0^\infty\cB(s)\dd s<\infty
\end{equation}
and
\begin{equation*}
M_a^*:=e^{aS_\infty}\cM_{e,a}(f_0),
\qquad
K_a:=\max\left\{1,\frac{2}{a}e^{-1+a/2}\right\}.
\end{equation*}
Then
\begin{equation}\label{ND-65}
\sup_{t\geq0}\int_{\bbr^{2d}}e^{ah_t(z)}f_t(z)\dd z\leq M_a^*.
\end{equation}
Moreover, for every $R\geq0$ and every $t\geq0$,
\begin{align}
\int_{\{h_t>R\}}f_t(z)\dd z
&\leq M_a^*e^{-aR},\label{ND-66}\\
\int_{\{h_t>R\}}(1+h_t)f_t(z)\dd z
&\leq K_aM_a^*e^{-aR/2},\label{ND-67}\\
\int_{\{h_t>R\}}|v|f_t(z)\dd z
&\leq K_aM_a^*e^{-aR/2}.\label{ND-68}
\end{align}
\end{lemma}

\begin{proof}
By Proposition~\ref{P4.1} and Proposition~\ref{NP4.3}, for $f_0$-a.e. initial label $z$,
\[
h_t(Z_t(z))
\leq h_0(z)+\left(\frac{\kappa\bar c_\phi}{4}+C_W\right)
\int_0^t\cB(s)\dd s
\leq h_0(z)+S_\infty.
\]
Since $f_t=(Z_t)_\#f_0$,
\begin{align*}
\int_{\bbr^{2d}}e^{ah_t(z)}f_t(z)\dd z
=\int_{\bbr^{2d}}e^{ah_t(Z_t(z))}f_0(z)\dd z\leq e^{aS_\infty}\int_{\bbr^{2d}}e^{ah_0(z)}f_0(z)\dd z
=M_a^*.
\end{align*}
This proves \eqref{ND-65}. The mass estimate \eqref{ND-66} follows from
\[
\mathbf 1_{\{r>R\}}\leq e^{-aR}e^{ar}.
\]
For the energy-weighted tail \eqref{ND-67}, we observe that
\[
\sup_{r\geq0}(1+r)e^{-ar/2}
=\max\left\{1,\frac{2}{a}e^{-1+a/2}\right\}=K_a.
\]
Hence we have
\[
(1+r)\mathbf 1_{\{r>R\}}
\leq K_ae^{-aR/2}e^{ar}.
\]
This gives \eqref{ND-67} and \eqref{ND-68} follows from $|v|\leq1+h_t$.
\end{proof}

\begin{lemma}[Pointwise estimates used in the high-energy argument]\label{NL4.14}
Set
\[
U_t(x):=(W*\rho_t)(x),
\qquad
F_t(x):=(\nabla W*\rho_t)(x).
\]
Then the following estimates hold.
\begin{enumerate}[label=\textnormal{(\roman*)}]
\item For every $x\in\bbr^d$,
\begin{equation}\label{ND-69}
U_t(x)\geq\underbar c_W\bigl(|x|^2+\cI_x(t)\bigr)\geq\underbar c_W|x|^2.
\end{equation}
\item For every $x\in\bbr^d$,
\begin{equation}\label{ND-70}
x\cdot F_t(x)
\geq\nu_WU_t(x)-C_{\rm vir}\cB(t),
\qquad
C_{\rm vir}:=\frac{\|D^2W\|_{L^\infty}}{2\underbar c_W}.
\end{equation}

\item The force fields satisfy
\begin{equation*}
|\cA[f_t](X,V)|
\leq\kappa\bar c_\phi\bigl(|V|+\cE_K(t)^{1/2}\bigr),
\qquad
|F_t(X)|\leq \|D^2W\|_{L^\infty}\bigl(|X|+\cI_x(t)^{1/2}\bigr).
\end{equation*}
\end{enumerate}
\end{lemma}

\begin{proof}
(i) We use the normalization conditions to see
\[
\int_{\bbr^d}|x-x_*|^2\rho_t(x_*)\dd x_*=|x|^2+\cI_x(t).
\]
This combines $W(z)\geq\underbar c_W|z|^2$ yield \eqref{ND-69}.
\vspace{0.2cm}

\noindent(ii) Note that
\begin{align*}
x\cdot F_t(x)=\int_{\bbr^d}(x-x_*)\cdot\nabla W(x-x_*)\rho_t(x_*)\dd x_*+\int_{\bbr^d}x_*\cdot\nabla W(x-x_*)\rho_t(x_*)\dd x_*.
\end{align*}
The first term is bounded below by $\nu_WU_t(x)$ by \eqref{C-1}$_3$. Since $\int x_*\rho_t(x_*)\dd x_*=0$, we get
\begin{align*}
\left|\int_{\bbr^d}x_*\cdot\nabla W(x-x_*)\rho_t(x_*)\dd x_*\right|
&=\left|\int_{\bbr^d}x_*\cdot\bigl(\nabla W(x-x_*)-\nabla W(x)\bigr)\rho_t(x_*)\dd x_*\right|\\
&\leq \|D^2W\|_{L^\infty}\cI_x(t)
\leq \frac{\|D^2W\|_{L^\infty}}{2\underbar c_W}\cE_{I,e}(t)
\leq C_{\rm vir}\cB(t),
\end{align*}
where Lemma~\ref{L2.2} was used. This proves \eqref{ND-70}.
\vspace{0.2cm}

\noindent(iii) Finally, we combine the unit mass, $0\leq\phi\leq\bar c_\phi$, $\nabla W(0)=0$, and the Cauchy--Schwarz inequality to obtain the desired estimates 
\begin{equation*}
	\begin{aligned}
		|\cA[f_t](X,V)|
		\leq
		\kappa\bar c_\phi
		\int_{\bbr^{2d}}(|v_*|+|V|)f_t(z_*)\dd z_*
		\leq
		\kappa\bar c_\phi
		\left(
		|V|+\cE_K(t)^{1/2}
		\right)
	\end{aligned}
\end{equation*}
and 
\begin{equation*}
	\begin{aligned}
		|F_t(X)|
		&\leq
		\int_{\bbr^d}
		|\nabla W(X-x_*)|\rho_t(x_*)\dd x_*
		\leq
		\|D^2W\|_{L^\infty}
		\int_{\bbr^d}
		(|X|+|x_*|)\rho_t(x_*)\dd x_*\\
		&\leq
		\|D^2W\|_{L^\infty}
		\left(
		|X|+\cI_x(t)^{1/2}
		\right).
	\end{aligned}
\end{equation*}

\end{proof}

\subsubsection{High-energy descent along characteristics}\label{sec:4.3.2}
We now prove the key dynamical estimate. Throughout this subsection, we set
\begin{equation*}
\begin{aligned}
	 C_{\rm vir}:=\frac{\|D^2W\|_{L^\infty}}{2\underbar c_W},\quad
K_A:=\frac{2\kappa\bar c_\phi}{\sqrt{\underbar c_W}}, \quad C_h=\frac{ \kappa\bar{c}_\phi}{4} + C_W,
\end{aligned}
\end{equation*}
\begin{equation}\label{ND-71}
\quad
C_{XV}:=4\sqrt{\frac{2}{\underbar c_W}},\quad	\sigma
	:=\min\left\{
	\frac12,\quad
	\frac{\nu_W}{8(1+\nu_W/2)},\quad 
	\left(\frac{\nu_W}{8K_A}\right)^2
	\right\},
\end{equation}
\begin{equation*}
\varepsilon:=\frac{\sigma}{32},	\quad b_\varepsilon
	:=\underbar c_\phi
	\left(1+\frac{8+2\varepsilon}{\underbar c_W}\right)^{-\beta/2},
	\quad
	c_1:=\frac{3\kappa b_\varepsilon}{8},
	\quad
	c_2:=\sigma c_1,
\end{equation*}
\begin{equation*}
	\begin{aligned}
	\cB_0:=\cB(0),\quad	C_1:=8
		+\frac{2\kappa\bar c_\phi}{\sqrt{\underbar c_W}}
		\left(2\sqrt2+\sqrt{\cB_0}\right)+\frac{2\|D^2W\|_{L^\infty}}{\sqrt{\underbar c_W}}
		\left(\frac{2}{\sqrt{\underbar c_W}}
		+\sqrt{\frac{\cB_0}{2\underbar c_W}}\right),
	\end{aligned}
\end{equation*}
and choose
\begin{equation}\label{ND-72}
	\delta_*:=\min\left\{1,\frac{c_2}{4C_1}\right\}, \quad \alpha:=\frac{a\varepsilon}{2},
	\qquad
	D_{\rm tail}:=2\sqrt2\,\kappa\bar c_\phi K_aM_a^*.
\end{equation}
\begin{proposition}[High-energy descent]\label{NP4.4}
Under the assumptions of Proposition \ref{NP4.3}, we define
\begin{equation}\label{ND-73}
H_*:=\max\left\{1,\quad \frac{1}{a\varepsilon}\log(2M_a^*),\quad \frac{2}{\alpha}
\log\left(\max\left\{1,
\frac{8D_{\rm tail}}{c_2\sqrt{\alpha e}}
\right\}\right),\quad \left(8C_{XV}\delta_*\right)^{2/\beta}\right\},
\end{equation}
\begin{align}
B_*&:=\min\left\{H_*\min\left\{
\left(\frac{\nu_W}{8K_A}\right)^2,
\frac{\nu_W}{8C_{\rm vir}}
\right\},\quad H_*^{1-\beta/2}
\min\left\{
\frac{c_2}{8C_W},
\frac{\delta_*\nu_W}{4C_h}
\right\}\right\}.
\label{ND-74}
\end{align}
Then, for every $H\geq H_*$ and every characteristic $Z_t(z)=(X(t),V(t))$ satisfying
\begin{equation}\label{ND-75}
H\leq h_t(X(t),V(t))\leq4H
\end{equation}
for a.e. $t$ in an interval contained in $[T_*,\infty)$, the level-dependent compensated energy
\begin{equation*}
\mathscr H_H(t)
:=h_t(X(t),V(t))+\delta_*H^{-\beta/2}X(t)\cdot V(t)
\end{equation*}
satisfies
\begin{equation}\label{ND-76}
\frac{\dd}{\dd t}\mathscr H_H(t)
\leq-c_{\rm des}H^{1-\beta/2}
\end{equation}
for a.e. $t$ in that interval. Moreover,
\begin{equation}\label{ND-77}
\left|\mathscr H_H(t)-h_t(X(t),V(t))\right|
\leq\frac18H.
\end{equation}
Here, 
\begin{equation}\label{ND-79}
	c_{\rm des}:=\min\left\{\frac{c_2}{2},\frac{\delta_*\nu_W}{4}\right\}.
\end{equation}\begin{equation}\label{ND-78}
	T_*:=\max\left\{
	0,
	\left[
	\frac{1}{c_B}
	\log\left(\max\left\{1,\frac{C_B}{B_*}\right\}\right)
	\right]^{1/\theta_B}-1
	\right\}.
\end{equation}

\end{proposition}

\begin{proof}
We split the proof into five steps.

\medskip
\noindent \noindent $\bullet$~\textbf{Step A: consequences of the annular restriction.}
By Lemma~\ref{NL4.14} (i) and \eqref{ND-75}, we have
\begin{equation}\label{ND-80}
|X|\leq\frac{2}{\sqrt{\underbar c_W}}\sqrt H,
\qquad
|V|\leq2\sqrt2\,\sqrt H.
\end{equation}
Moreover, Remark~\ref{R2.1} and Lemma~\ref{L2.2} give
\begin{equation}\label{ND-81}
\cE_K(t)\leq\cB_0,
\qquad
\cI_x(t)\leq\frac{\cB_0}{2\underbar c_W}.
\end{equation}
Using Lemma~\ref{NL4.14} (iii), \eqref{ND-80}, and \eqref{ND-81}, we obtain
\begin{align}\label{ND-82}
	\begin{aligned}
\left|\frac{\dd}{\dd t}(X\cdot V)\right|
&\leq |V|^2+|X|\,|\cA[f_t](X,V)|+|X|\,|F_t(X)|\\
&\leq 8H
+\frac{2\kappa\bar c_\phi}{\sqrt{\underbar c_W}}
\left(2\sqrt2+\sqrt{\cB_0}\right)H
+\frac{2\|D^2W\|_{L^\infty}}{\sqrt{\underbar c_W}}
\left(\frac{2}{\sqrt{\underbar c_W}}
+\sqrt{\frac{\cB_0}{2\underbar c_W}}\right)H=:C_1H.
	\end{aligned}
\end{align}

\medskip
\noindent\noindent $\bullet$~\textbf{Step B: a uniformly large low-energy background.}
For $H\geq1$ define
\[
G_H(t):=\{(x_*,v_*)\in\bbr^{2d}:h_t(x_*,v_*)\leq\varepsilon H\}.
\]
By \eqref{ND-66}, 
\[
\int_{G_H(t)^c}f_t(z_*)\dd z_*
\leq M_a^*e^{-a\varepsilon H}.
\]
The definition of $H_{*}$ in $\eqref{ND-73}_2$ therefore gives
\begin{equation}\label{ND-83}
\int_{G_H(t)}f_t(z_*)\dd z_*\geq\frac12,
\qquad H\geq H_*,\quad t\geq0.
\end{equation}

\medskip
\noindent\noindent $\bullet$~\textbf{Step C: kinetic-dominated region.}
Assume
\begin{equation*}
|V|^2\geq\sigma H.
\end{equation*}
For $(x_*,v_*)\in G_H(t)$,
\[
|v_*|^2\leq2\varepsilon H
=\frac{\sigma H}{16}
\leq\frac{|V|^2}{16},
\]
so that
\begin{equation}\label{ND-84}
V\cdot(v_*-V)\leq-\frac34|V|^2.
\end{equation}
Also, Lemma~\ref{NL4.14} (i) gives
\[
|x_*|^2\leq\frac{\varepsilon H}{\underbar c_W}.
\]
Together with \eqref{ND-80}, we obtain
\[
|X-x_*|^2
\leq2|X|^2+2|x_*|^2
\leq\frac{8+2\varepsilon}{\underbar c_W}H.
\]
Since $H\geq1$, the lower bound in \eqref{C-2} yields
\begin{equation}\label{ND-85}
\phi(|X-x_*|)
\geq b_\varepsilon H^{-\beta/2},
\qquad (x_*,v_*)\in G_H(t).
\end{equation}
Combining \eqref{ND-83}, \eqref{ND-84}, and \eqref{ND-85}, we have 
\begin{align*}
\kappa\int_{G_H(t)}\phi(|X-x_*|)V\cdot(v_*-V)f_t(z_*)\dd z_*
&\leq-\frac{3\kappa b_\varepsilon}{8}H^{-\beta/2}|V|^2=:-c_1H^{-\beta/2}|V|^2.
\end{align*}
On $G_H(t)^c$, only the positive part is relevant. We combine Lemma~\ref{NL4.13}, \eqref{ND-80}, and the definition of $D_{\rm tail}$  to give
\begin{align*}
&\kappa\int_{G_H(t)^c}\phi(|X-x_*|)V\cdot(v_*-V)f_t(z_*)\dd z_*\nonumber\\
&\qquad\qquad\leq\kappa\bar c_\phi|V|
\int_{G_H(t)^c}|v_*|f_t(z_*)\dd z_*\leq D_{\rm tail}\sqrt H\,e^{-\alpha H}.
\end{align*}
Hence, by Lemma \ref{L4.4} and Lemma~\ref{L4.5} we have 
\begin{equation}\label{ND-86}
\frac{\dd}{\dd t}h_t(X,V)
\leq-c_2H^{1-\beta/2}
+D_{\rm tail}\sqrt H\,e^{-\alpha H}
+C_W\cB(t).
\end{equation}
We next verify that the two error terms in \eqref{ND-86} can be absorbed into the first term. Since $0<\beta\leq2$ and $H\geq1$,
\[
H^{(\beta-1)/2}\leq H^{1/2}.
\]
Furthermore,
\[
H^{1/2}e^{-\alpha H}
=\bigl(H^{1/2}e^{-\alpha H/2}\bigr)e^{-\alpha H/2}
\leq\frac{1}{\sqrt{\alpha e}}e^{-\alpha H/2}.
\]
Thus the definition of $H_{*}$ in $\eqref{ND-73}_3$ implies
\begin{equation}\label{ND-87}
D_{\rm tail}\sqrt H\,e^{-\alpha H}
\leq\frac{c_2}{8}H^{1-\beta/2},
\qquad H\geq H_*.
\end{equation}
By \eqref{ND-60} and \eqref{ND-78}, we have
\begin{align}\label{ND-88}
	\begin{aligned}
	\cB(t)
	\leq
	C_B
	\exp\left[
	-c_B(1+t)^{\theta_B}
	\right]\leq
	\frac{C_B}{
		\max\left\{
		1,\frac{C_B}{B_*}
		\right\}
	}
	\leq
	B_*,\quad \text{for any} \quad t\ge T_*.
	\end{aligned}
\end{align}
Therefore, we combine $\eqref{ND-74}_2$ and $H\geq H_*$ to give
\begin{equation}\label{ND-89}
C_W\cB(t)
\leq\frac{c_2}{8}H_*^{1-\beta/2}
\leq\frac{c_2}{8}H^{1-\beta/2}.
\end{equation}
Finally, by \eqref{ND-72} and \eqref{ND-82}, we get
\[
\delta_*H^{-\beta/2}
\left|\frac{\dd}{\dd t}(X\cdot V)\right|
\leq\frac{c_2}{4}H^{1-\beta/2}.
\]
Combining this with \eqref{ND-86}, \eqref{ND-87}, and \eqref{ND-89}, we obtain
\begin{equation}\label{ND-90}
\frac{\dd}{\dd t}\mathscr H_H(t)
\leq-\frac{c_2}{2}H^{1-\beta/2}.
\end{equation}

\medskip
\noindent \noindent $\bullet$~\textbf{Step D: potential-dominated region.}
Assume now
\begin{equation*}
|V|^2<\sigma H.
\end{equation*}
Then
\begin{equation*}
U_t(X)
=h_t(X,V)-\frac12|V|^2
\geq\left(1-\frac{\sigma}{2}\right)H.
\end{equation*}
We use Lemma~\ref{L4.4} and Lemma \ref{L4.5} to find
\begin{equation}\label{ND-91}
\frac{\dd}{\dd t}h_t(X,V)
\leq C_h\cB(t).
\end{equation}
Next, we utilize Lemma~\ref{NL4.14} (ii) and (iii), together with \eqref{ND-80} to derive
\begin{align*}
\frac{\dd}{\dd t}(X\cdot V)
&=|V|^2+X\cdot\cA[f_t](X,V)-X\cdot F_t(X)\\
&\leq-\nu_W\left(1-\frac{\sigma}{2}\right)H
+\sigma H+K_A\sqrt\sigma\,H
+K_AH^{1/2}\cB(t)^{1/2}
+C_{\rm vir}\cB(t).
\end{align*}
By the definition of $\sigma$ in $\eqref{ND-71}_2$,
\[
\left(1+\frac{\nu_W}{2}\right)\sigma+K_A\sqrt\sigma
\leq\frac{\nu_W}{4},
\]
so the first three terms are bounded above by $-\frac{3\nu_W}{4}H$.
On the other hand, $\eqref{ND-74}_1$, and \eqref{ND-88} imply
\[
K_AH^{1/2}\cB(t)^{1/2}
\leq\frac{\nu_W}{8}H,
\qquad
C_{\rm vir}\cB(t)
\leq\frac{\nu_W}{8}H.
\]
Consequently, we have
\begin{equation}\label{ND-92}
\frac{\dd}{\dd t}(X\cdot V)
\leq-\frac{\nu_W}{2}H.
\end{equation}
Furthermore,  we again use $\eqref{ND-74}_2$ and \eqref{ND-88} to see
\[
C_h\cB(t)
\leq\frac{\delta_*\nu_W}{4}H_*^{1-\beta/2}
\leq\frac{\delta_*\nu_W}{4}H^{1-\beta/2}.
\]
Thus \eqref{ND-91} and \eqref{ND-92} give
\begin{equation}\label{ND-93}
\frac{\dd}{\dd t}\mathscr H_H(t)
\leq-\frac{\delta_*\nu_W}{4}H^{1-\beta/2}.
\end{equation}
Combining \eqref{ND-90} and \eqref{ND-93}, we  obtain \eqref{ND-76} with the explicit constant $c_{\rm des}$ in \eqref{ND-79}
\begin{equation*}
	\frac{\dd}{\dd t}\mathscr H_H(t)
	\leq-\min\left\{\frac{c_2}{2},\frac{\delta_*\nu_W}{4}\right\}H^{1-\beta/2}=:-c_{\rm des}H^{1-\beta/2}.
\end{equation*}
\medskip
\noindent\noindent $\bullet$~\textbf{Step E: comparison of $\mathscr H_H$ and $h_t$.}
From \eqref{ND-80}, we have
\[
|X\cdot V|\leq C_{XV}H.
\]
Therefore, we obtain
\[
|\mathscr H_H-h_t(X,V)|
\leq C_{XV}\delta_*H^{1-\beta/2}.
\]
The definition of $H_{*}$ in $\eqref{ND-73}_4$ gives
\[
C_{XV}\delta_*H^{-\beta/2}\leq\frac18,
\qquad H\geq H_*,
\]
which is exactly \eqref{ND-77}.
\end{proof}

\begin{corollary}[Crossing time of an energy annulus]\label{NC4.2}

Let $H\geq H_*$ and $s\geq T_*$. If
\[
h_s(X(s),V(s))=2H,
\]
then the characteristic reaches the level $H$ at some time
\[
\tau\in\left[s,s+\frac{3}{2c_{\rm des}}H^{\beta/2}\right].
\]
Before reaching $H$, it cannot reach the level $4H$.
\end{corollary}

\begin{proof}
Since $t\mapsto h_t(X(t),V(t))$ is continuous, define the first exit time
\[
\tau:=\inf\left\{t\geq s:
h_t(X(t),V(t))\notin(H,4H)\right\}.
\]
We first prove that $\tau$ is finite and satisfies the asserted upper bound. Suppose the contrary holds, we have
\[
\tau>s+\frac{3}{2c_{\rm des}}H^{\beta/2}.
\]
Then \eqref{ND-75} holds on
$[s,s+\frac{3}{2c_{\rm des}}H^{\beta/2}]$. Since $h_s=2H$, \eqref{ND-77} gives
\[
\mathscr H_H(s)\leq\frac{17}{8}H.
\]
Integrating \eqref{ND-76}, we obtain
\begin{align*}
\mathscr H_H\left(s+\frac{3}{2c_{\rm des}}H^{\beta/2}\right)
&\leq\frac{17}{8}H
-c_{\rm des}H^{1-\beta/2}\frac{3}{2c_{\rm des}}H^{\beta/2}=\frac{17}{8}H-\frac32H
=\frac58H.
\end{align*}
However, as long as the characteristic remains in $[H,4H]$, \eqref{ND-77} implies
\[
\mathscr H_H(t)\geq h_t(X(t),V(t))-\frac18H\geq\frac78H,
\]
a contradiction. Thus, we have 
\[
\tau\leq s+\frac{3}{2c_{\rm des}}H^{\beta/2}.
\]
It remains to identify the exit boundary. If $h_\tau=4H$, then continuity and \eqref{ND-77} give
\[
\mathscr H_H(\tau)\geq\frac{31}{8}H,
\qquad
\mathscr H_H(s)\leq\frac{17}{8}H,
\]
whereas \eqref{ND-76} shows that $\mathscr H_H$ is non-increasing on $[s,\tau]$. This is impossible. Hence $h_\tau=H$, and the proof is complete.
\end{proof}

\subsubsection{Dyadic descent and the dynamically generated exterior tail}\label{sec:4.3.3}
We now iterate the annulus-crossing estimate. 

\begin{lemma}[Finite descent time and no high-energy re-entry]\label{NL4.15}
The following assertions hold.
\begin{enumerate}[label=\textnormal{(\roman*)}]
\item If
\[
E(z):=h_{T_*}(Z_{T_*}(z))>2H_*,
\]
then the characteristic reaches $\{h_t\leq2H_*\}$ no later than
\begin{equation*}
T_*+\frac{3}{2c_{\rm des}(2^{\beta/2}-1)}E(z)^{\beta/2}=:T_*+C_{\rm dy}E(z)^{\beta/2}.
\end{equation*}
\item If $s\geq T_*$ and $h_s(Z_s(z))\leq2H_*$, then
\begin{equation*}
h_t(Z_t(z))<4H_*,
\qquad t\geq s.
\end{equation*}
\end{enumerate}
\end{lemma}

\begin{proof}
\noindent (i).
Let $E:=E(z)>2H_*$ and define
\[
n_*:=\min\left\{n\in\bbn:\frac{E}{2^n}\leq2H_*\right\}.
\]
Then $n_*\geq1$, and by minimality, we get
\[
\frac{E}{2^j}>2H_*,
\qquad j=0,\ldots,n_*-1.
\]
Suppose the characteristic is at level $E/2^j$. Set
\[
H_j:=\frac{E}{2^{j+1}}.
\]
For $j=0,\ldots,n_*-1$, one has $H_j\geq H_*$, and the current level equals $2H_j$. Corollary~\ref{NC4.2} therefore shows that the characteristic reaches $E/2^{j+1}$ within time at most
\[
\frac{3}{2c_{\rm des}}H_j^{\beta/2}
=\frac{3}{2c_{\rm des}}\left(\frac{E}{2^{j+1}}\right)^{\beta/2}.
\]
Summing over $j=0,\ldots,n_*-1$ gives
\begin{align*}
\tau_{n_*}-T_*
&\leq \frac{3}{2c_{\rm des}}E^{\beta/2}
\sum_{j=0}^{n_*-1}2^{-(j+1)\beta/2}\leq \frac{3}{2c_{\rm des}}E^{\beta/2}
\sum_{j=1}^{\infty}2^{-j\beta/2}\\
&=\frac{\frac{3}{2c_{\rm des}}}{2^{\beta/2}-1}E^{\beta/2}
=:C_{\rm dy}E^{\beta/2}.
\end{align*}
At time $\tau_{n_*}$ guaranteed by Corollary~\ref{NC4.2}, the energy is $E/2^{n_*}\leq2H_*$, which proves Lemma \ref{NL4.15} (i).

\medskip
\noindent (ii).
Assume, by contradiction, that $h_s(Z_s(z))\leq2H_*$ but $h_t(Z_t(z))$ reaches $4H_*$ at a later time. Let
\[
\tau:=\inf\{t>s:h_t(Z_t(z))=4H_*\}.
\]
By continuity there exists a last time
\[
s_1:=\sup\{r\in[s,\tau):h_r(Z_r(z))=2H_*\}.
\]
Then
\[
2H_*\leq h_t(Z_t(z))\leq4H_*,
\qquad s_1\leq t\leq\tau.
\]
Proposition~\ref{NP4.4} with $H=H_*$ therefore applies on $[s_1,\tau]$, so $\mathscr H_{H_*}$ is non-increasing. However, \eqref{ND-77} gives
\[
\mathscr H_{H_*}(s_1)\leq\frac{17}{8}H_*,
\qquad
\mathscr H_{H_*}(\tau)\geq\frac{31}{8}H_*,
\]
a contradiction. Hence Lemma \ref{NL4.15} (ii) holds.
\end{proof}

\begin{corollary}[Dynamic exponential exterior tail]\label{NC4.3}
Define
\begin{equation*}
\Omega_*(t):=\{z\in\bbr^{2d}:h_t(z)\leq4H_*\},\quad
\mathfrak M_*(t):=\int_{\Omega_*(t)^c}(1+h_t(z))f_t(z)\dd z.
\end{equation*}
We set
\begin{equation*}
c_0:=\frac12(2C_{\rm dy})^{-2/\beta},
\quad
c_{\rm ext}:=\frac{ac_0}{2},\quad
t_{\rm ext}:=\max\left\{2T_*,\,2C_{\rm dy}(2S_\infty)^{\beta/2}\right\},
\end{equation*}
and
\begin{equation*}
C_{\rm ext}:=\max\left\{
(1+S_\infty)K_a\cM_{e,a}(f_0),\quad
(1+\cB_0)\exp\left(c_{\rm ext}t_{\rm ext}^{2/\beta}\right)
\right\}.
\end{equation*}
Then, we have 
\begin{equation}\label{ND-94}
\mathfrak M_*(t)\leq C_{\rm ext}\exp\left[-c_{\rm ext}t^{2/\beta}\right],
\qquad t\geq0.
\end{equation}
\end{corollary}

\begin{proof}
Fix $t\geq T_*$ and an initial label $z$ such that
\[
h_t(Z_t(z))>4H_*.
\]
By the no-reentry property in Lemma~\ref{NL4.15} (ii), this characteristic cannot have reached the level $2H_*$ at any time in $[T_*,t]$. In particular, we have
\[
E(z):=h_{T_*}(Z_{T_*}(z))>2H_*.
\]
Otherwise it would already lie in the inner region at time $T_*$ and could not be outside $4H_*$ at time $t$.

Lemma~\ref{NL4.15} (i) says that a characteristic with initial energy $E(z)$ at time $T_*$ must reach $\{h\leq2H_*\}$ by time
\[
T_*+C_{\rm dy}E(z)^{\beta/2}.
\]
Since this has not happened by time $t$,
\begin{equation}\label{ND-95}
E(z)>\left(\frac{t-T_*}{C_{\rm dy}}\right)^{2/\beta}.
\end{equation}
For $t\geq t_{\rm ext}$, we have $t-T_*\geq t/2$ and
\[
\left(\frac{t}{2C_{\rm dy}}\right)^{2/\beta}\geq2S_\infty.
\]
We use \eqref{ND-64} and Proposition \ref{P4.1} to derive
\[
E(z)=h_{T_*}(Z_{T_*}(z))\leq h_0(z)+S_\infty.
\]
Combining this with \eqref{ND-95}, for $t\geq t_{\rm ext}$ we obtain
\begin{align*}
h_0(z)
&>\left(\frac{t-T_*}{C_{\rm dy}}\right)^{2/\beta}-S_\infty\geq\left(\frac{t}{2C_{\rm dy}}\right)^{2/\beta}-S_\infty\geq\frac12\left(\frac{t}{2C_{\rm dy}}\right)^{2/\beta}
=c_0t^{2/\beta}.
\end{align*}
Therefore, for $t\geq t_{\rm ext}$,
\begin{equation*}
h_t(Z_t(z))>4H_*
\quad\Longrightarrow\quad
h_0(z)>c_0t^{2/\beta}.
\end{equation*}
Using $f_t=(Z_t)_\#f_0$ and \eqref{ND-64}, we get
\begin{align*}
\mathfrak M_*(t)
&=\int_{\{h_t(Z_t(z))>4H_*\}}
\bigl(1+h_t(Z_t(z))\bigr)f_0(z)\dd z\\
&\leq\int_{\{h_0>c_0t^{2/\beta}\}}
(1+h_0+S_\infty)f_0(z)\dd z\\
&\leq(1+S_\infty)K_a\cM_{e,a}(f_0)
\exp\left[-\frac{ac_0}{2}t^{2/\beta}\right].
\end{align*}
This proves \eqref{ND-94} for $t\geq t_{\rm ext}$. For $0\leq t\leq t_{\rm ext}$, note that
\[
\mathfrak M_*(t)
\leq\int_{\bbr^{2d}}(1+h_t)f_t(z)\dd z
=1+\frac12\cE_K(t)+\cE_{I,e}(t)
\leq1+\cB_0.
\]
The definition of $C_{\rm ext}$ then extends \eqref{ND-94} to every $t\geq0$.
\end{proof}

\subsubsection{Proof of the third assertion}\label{sec:4.3.4}
Inside $\Omega_*(t)$, the communication weight has a time-independent positive lower bound. Indeed, if $z=(x,v)$ and $z_*=(x_*,v_*)$ belong to $\Omega_*(t)$, then Lemma~\ref{NL4.14} (i) gives
\[
|x|^2,\ |x_*|^2\leq\frac{4H_*}{\underbar c_W},
\]
and hence
\begin{equation}\label{ND-96}
	\phi(|x-x_*|)\geq b_*:=\underbar c_\phi
	\left(1+\frac{16H_*}{\underbar c_W}\right)^{-\beta/2}>0.
\end{equation}

\begin{proposition}[Fixed-region localized dissipation]\label{NP4.5}
	The following estimate holds:
	\begin{equation}\label{ND-97}
		\Lambda(t)\geq 2\kappa b_*\cE_K(t)
		-2\kappa b_*\bigl(2+\cB_0\bigr)C_{\rm ext}\exp\left[-c_{\rm ext}t^{2/\beta}\right],
		\qquad t\geq0.
	\end{equation}
Here, $c_{\rm ext}$ and $C_{\rm ext}$ are defined in
	Corollary~\ref{NC4.3}.
\end{proposition}

\begin{proof}
	For simplicity, write
	\[
	\Omega_*:=\Omega_*(t).
	\]
	By \eqref{ND-96}, for every
	$z=(x,v),z_*=(x_*,v_*)\in\Omega_*$,
	\[
	\phi(|x-x_*|)\geq b_*.
	\]
	Hence,
	\begin{align*}
		\Lambda(t)
		&=
		\kappa
		\int_{\bbr^{4d}}
		\phi(|x-x_*|)
		|v-v_*|^2
		f_t(z)f_t(z_*)\dd z\dd z_*
	\geq
		\kappa b_*
		\int_{\Omega_*\times\Omega_*}
		|v-v_*|^2
		f_t(z)f_t(z_*)\dd z\dd z_*.
	\end{align*}
	Since the total mass is one and the mean velocity vanishes, we have
	\begin{align*}
		\int_{\Omega_*\times\Omega_*}
		|v-v_*|^2f_t(z)f_t(z_*)\dd z\dd z_*
		&=
		2\cE_K(t)
		-
		\int_{(\Omega_*\times\Omega_*)^c}
		|v-v_*|^2f_t(z)f_t(z_*)\dd z\dd z_*.
	\end{align*}
	It remains to estimate the exterior contribution. Since
	\[
	\mathbf 1_{(\Omega_*\times\Omega_*)^c}(z,z_*)
	\leq
	\mathbf 1_{\Omega_*^c}(z)
	+
	\mathbf 1_{\Omega_*^c}(z_*),
	\]
	and the integrand is symmetric under
	$(z,z_*)\leftrightarrow(z_*,z)$, we obtain
	\begin{align*}
		\int_{(\Omega_*\times\Omega_*)^c}
		|v-v_*|^2f_t(z)f_t(z_*)\dd z\dd z_*\leq
		2\int_{\Omega_*^c}|v|^2f_t(z)\dd z
		+
		2\cE_K(t)
		\int_{\Omega_*^c}f_t(z)\dd z.
		\label{ND-97-proof-6}
	\end{align*}
	Moreover, by the energy dissipation estimate
$	\cE_K(t)
	\leq
	\cB(t)
	\leq
	\cB_0
$
we have
	\begin{align*}
		2\int_{\Omega_*^c}|v|^2f_t(z)\dd z
		&\leq
		4\int_{\Omega_*^c}h_t(z)f_t(z)\dd z
		\leq
		4\mathfrak M_*(t),\\
		2\cE_K(t)
		\int_{\Omega_*^c}f_t(z)\dd z
		&\leq
		2\cB_0
		\int_{\Omega_*^c}f_t(z)\dd z
		\leq
		2\cB_0\mathfrak M_*(t).
	\end{align*}
	Hence,
	\begin{equation*}
		\int_{(\Omega_*\times\Omega_*)^c}
		|v-v_*|^2f_t(z)f_t(z_*)\dd z\dd z_*
		\leq
		2\bigl(2+\cB_0\bigr)\mathfrak M_*(t).
	\end{equation*}
	This implies
	\begin{align*}
		\Lambda(t)
		&\geq
		2\kappa b_*\cE_K(t)
		-
		2\kappa b_*
		\bigl(2+\cB_0\bigr)\mathfrak M_*(t).
	\end{align*}
	Finally, Corollary~\ref{NC4.3} gives
	\[
	\mathfrak M_*(t)
	\leq
	C_{\rm ext}
	\exp\left[-c_{\rm ext}t^{2/\beta}\right],
	\qquad t\geq0.
	\]
	Therefore, we obtain the desired estimate
	\[
	\Lambda(t)
	\geq
	2\kappa b_*\cE_K(t)
	-
	2\kappa b_*
	\bigl(2+\cB_0\bigr)
	C_{\rm ext}
	\exp\left[-c_{\rm ext}t^{2/\beta}\right].
	\]
\end{proof}
We can now complete the proof of Theorem~\ref{T3.3}~\textnormal{(iii)}. The global bound $0\leq\phi\leq\bar c_\phi$ and the nonnegativity of $\Phi$ imply
\begin{equation*}
	-\cE_K-\frac{1}{2\underbar c_W}\cE_{I,e}
	\leq\cE_{I,s}
	\leq\cE_K+\frac{1+\kappa\bar c_\phi}{2\underbar c_W}\cE_{I,e}.
\end{equation*}
Thus there exists $C_s>0$ such that
\[
|\cE_{I,s}(t)|\leq \max\left\{1+\frac{1}{2\underbar c_W}, 1+\frac{1+\kappa\bar c_\phi}{2\underbar c_W}\right\}\cB(t)=:C_s\cB(t).
\]
Choose a constant gauge $\omega_*>0$ satisfying
\begin{equation*}
	C_s\omega_*\leq\frac12,
	\qquad
	\omega_*\leq\frac{\kappa b_*}{2},
\end{equation*}
and define
\begin{equation*}
	\cE_*(t):=\cE_K(t)+\cE_{I,e}(t)+\omega_*\cE_{I,s}(t).
\end{equation*}
Then
\begin{equation}\label{ND-98}
	\frac12\cB(t)\leq\cE_*(t)\leq\frac32\cB(t).
\end{equation}
By Remark~\ref{R2.1} (1), Lemma~\ref{L2.3} (iii), the virial condition \eqref{C-1}$_3$, and Proposition~\ref{NP4.5}, we have
\begin{align*}
	\frac{\dd}{\dd t}\cE_*
	&\leq-\Lambda(t)+\omega_*\left(2\cE_K-\nu_W\cE_{I,e}\right)\\
	&\leq-2(\kappa b_*-\omega_*)\cE_K-\omega_*\nu_W\cE_{I,e}
	+2\kappa b_*
	\bigl(2+\cB_0\bigr)
	C_{\rm ext}
	\exp\left[-c_{\rm ext}t^{2/\beta}\right]\\
	&\leq-\min\left\{\kappa b_*, \omega_*\nu_W\right\}\cB(t)+2\kappa b_*
	\bigl(2+\cB_0\bigr)
	C_{\rm ext}
	\exp\left[-c_{\rm ext}t^{2/\beta}\right].
\end{align*}
Here, $c_{\rm ext}$ and $C_{\rm ext}$ are defined in
Corollary~\ref{NC4.3}. We set
\[
\lambda_*:=
\frac12\min\left\{\kappa b_*,\omega_*\nu_W\right\},
\qquad
	\tilde{C}_{\rm ext}:=
2\kappa b_*
\bigl(2+\cB_0\bigr)
C_{\rm ext}.
\]
Then, we have
\begin{equation*}
	\frac{\dd}{\dd t}\cE_*(t)
	+\lambda_*\cE_*(t)
	\leq
	\tilde{C}_{\rm ext}\exp\left[-c_{\rm ext}t^{2/\beta}\right].
\end{equation*}
Since \(0<\beta\leq2\), we have
$
p:=\frac{2}{\beta}\geq1.$
For every \(t\geq0\),
\[
t-t^p\leq1.
\]
Indeed, if \(t\geq1\), then \(t^p\geq t\), whereas if
\(0\leq t\leq1\), then \(t-t^p\leq t\leq1\).
Consequently,
\[
\exp\left[-c_{\rm ext}t^{2/\beta}\right]
=
e^{-c_{\rm ext}t}
e^{c_{\rm ext}(t-t^{2/\beta})}
\leq
e^{c_{\rm ext}}e^{-c_{\rm ext}t},
\qquad t\geq0.
\]
Hence we obtain
\begin{equation*}
	\frac{\dd}{\dd t}\cE_*(t)
	+\lambda_*\cE_*(t)
	\leq
		\tilde{C}_{\rm ext}e^{c_{\rm ext}}e^{-c_{\rm ext}t}.
\end{equation*}
Multiplying the above differential inequality by \(e^{\lambda_*t}\) and
integrating over \([0,t]\), we obtain
\begin{align*}
	\cE_*(t)
	&\leq
	e^{-\lambda_*t}\cE_*(0)
	+
	\tilde{C}_{\rm ext}e^{c_{\rm ext}}
	\int_0^t
	e^{-\lambda_*(t-s)}
	e^{-c_{\rm ext}s}\dd s\\
&	\le 	e^{-\lambda_*t}\cE_*(0)+
\frac{
	8\kappa b_*
	\bigl(2+\cB_0\bigr)
	C_{\rm ext}
	e^{c_{\rm ext}}
}{
	\min\left\{\kappa b_*,\omega_*\nu_W\right\}
}e^{-\min\left\{
\frac14
\kappa b_*,~\frac14\omega_*\nu_W,~\frac12
c_{\rm ext}
\right\}t}\\
&\le \left(\cE_*(0)
+
\frac{
	8\kappa b_*
	\bigl(2+\cB_0\bigr)
	C_{\rm ext}
	e^{c_{\rm ext}}
}{
	\min\left\{\kappa b_*,\omega_*\nu_W\right\}
}\right)e^{-\min\left\{
\frac14
\kappa b_*,~\frac14\omega_*\nu_W,~\frac12
c_{\rm ext}
\right\}t}\\
&=:C_{\rm exp}e^{-c_{\rm exp}t}.
\end{align*}
Here, we used
\[
\lambda_*=
\frac12\min\left\{\kappa b_*,\omega_*\nu_W\right\},
\qquad
\min\left\{\frac{\lambda_*}{2},\frac{c_{\rm ext}}{2}\right\}=\min\left\{
\frac14
\kappa b_*,~\frac14\omega_*\nu_W,~\frac12
c_{\rm ext}
\right\},
\]
to obtain
\[
\begin{aligned}
	\int_0^t
	e^{-\lambda_*(t-s)}e^{-c_{\rm ext}s}\dd s
	&\leq
	e^{-c_{\rm exp}t}
	\int_0^t
	e^{-(\lambda_*-c_{\rm exp})(t-s)}\dd s\\
	&\leq
	\frac{4}{
		\min\{\kappa b_*,\omega_*\nu_W\}
	}
	e^{-\min\left\{
		\frac14
		\kappa b_*,~\frac14\omega_*\nu_W,~\frac12
		c_{\rm ext}
		\right\}t}.
\end{aligned}
\]
Finally, we combine \eqref{ND-98} and Lemma~\ref{L2.2}~(4) to obtain the desired upper bound in \eqref{C-19}
\[
\cF(t)
\leq
2\max\left\{
2,\frac1{\underbar c_W}
\right\}
C_{\rm exp}e^{-c_{\rm exp}t}=:Ce^{-ct},
\qquad t\geq0.
\]
The proof of the lower bound in \eqref{C-19} is identical to that of Theorem~\ref{T3.3}~\textnormal{(i)}. Hence, the proof of Theorem~\ref{T3.3}~\textnormal{(iii)} is complete.\newline

\section{Optimality of polynomial convergence rate}\label{sec:5}
\setcounter{equation}{0}
In this section, we provide a proof of Theorem \ref{T3.4} by explicitly constructing a nonconvex refinement of the rotating-shell potential. Throughout this section, we take $\Omega>0$ to be a fixed constant.

\subsection{Preparatory estimates} \label{sec:5.1}
In this subsection, we study preparatory estimates to be crucially used in the proof of the second main result.  We begin with two-dimensional setting with $d=2$.

\begin{lemma}[Construction of a radial potential]\label{L5.1}
For a fixed constant $\Omega>0$, there exists a radially symmetric force potential $U\in \mathcal C_c^\infty(\bbr^2)$, supported in the annulus $\{1<|z|<2\}$ such that
\begin{align*}
\begin{aligned} \label{E-1}
& (i)~W_{\rm rs}(z)~~\mbox{is not convex}. \\
& (ii)~W_{\rm rs}(z):=\frac{\Omega^2}{2}|z|^2+U(z)~~\mbox{satisfies the conditions in $({\mathcal F}_A1)$}.
\end{aligned}
\end{align*}
\end{lemma}
\begin{proof}
We choose a smooth bump function $\chi\in \mathcal C_c^\infty((1,2))$ with $\chi\equiv1$ on $[4/3,5/3]$.  For a small $\delta>0$ and a large integer $k$, we set
\begin{equation} \label{E-2}
U(z)=u(|z|),
\quad
u(r):=\frac\delta k\chi(r)\cos(kr).
\end{equation}
For such $U$, we define a radially symmetric potential $W_{\rm rs}$:
\begin{equation} \label{E-3}
W_{\rm rs}(z):=\frac{\Omega^2}{2}|z|^2+U(z). 
\end{equation}
\noindent (i)~Note that the ansatz \eqref{E-2} clearly shows that $W_{\rm rs}$ is not convex. For all sufficiently large $k$, the interval $[4/3,5/3]$ contains a point $r_k$ with $kr_k=2\pi\ell$ for some integer $\ell$.  Since $\chi\equiv1$ near $r_k$,
\[
\frac{\dd^2}{\dd r^2}W_{\rm rs}(r_k)
=\Omega^2-\delta k<0
\]
when $k>\Omega^2/\delta$.  The radial Hessian eigenvalue is therefore negative at $r_k$. 

\vspace{0.2cm}

\noindent (ii)~Next, we claim that $W_{\rm rs}$ satisfies a set of conditions in $(\cF_A 1)$:
\begin{align}
\begin{aligned} \label{E-4}
& W_{\rm rs}(-x) = W_{\rm rs}(x), \quad \underbar{c}_W |x|^2\leq W_{\rm rs}(x)\leq \bar{c}_W |x|^2, \\
& x \cdot\nabla W_{\rm rs}(x)\geq \nu_W W_{\rm rs}(x), \quad \|D^2W_{\rm rs} \|_{L^{\infty}} < \infty.
\end{aligned}
\end{align}
\noindent $\bullet$~Case 1 (verification of $\eqref{E-4}_1$):~Since $W_{\rm rs}$ is radially symmetric,  it is clearly an even function:
\[ W_{\rm rs}(-x) = W_{\rm rs}(x). \]
\noindent $\bullet$~Case 2 (verification of $\eqref{E-4}_2$):~By \eqref{E-2}, one has 
\begin{equation} \label{E-5}
 ru^{\prime}(r) = \frac{\delta}{k}  r \chi^{\prime}(r) \cos(kr) - \delta \chi(r) r\sin (kr). 
\end{equation}
Because the support is separated from the origin, this defines a smooth radial function on $\bbr^2$. \newline

\noindent On $1\leq r\leq2$, we use \eqref{E-2} and \eqref{E-5} to see
\begin{align}
\begin{aligned} \label{E-6}
|u(r)| &= \Big |\frac\delta k\chi(r)\cos(kr) \Big | \leq\frac\delta k \quad \mbox{and} \\
|ru'(r)| &= \Big|  \frac{\delta}{k}  r \chi^{\prime}(r) \cos(kr) - \delta \chi(r) r\sin (kr)  \Big| \leq 2\delta\left(1+\frac{\norm{\chi'}_\infty}{k}\right).
\end{aligned}
\end{align}
We choose $\delta>0$ sufficiently small and sufficiently large $k$ such that 
\begin{equation} \label{E-7}
\delta < \frac{\Omega^2}{8} \quad \mbox{and} \quad    \frac{\delta}{k} \leq \min \Big \{1,~\frac{1}{ \|\chi^{\prime} \|_\infty} \Big \}  \frac{\Omega^2}{8}.
\end{equation}
Since $r = |z| \geq1$ on the support, we use \eqref{E-6} and \eqref{E-7} to find 
\begin{align}
\begin{aligned} \label{E-8}
W_{\rm rs}(z) &\geq \frac{\Omega^2}{2}|z|^2 - |u(|z|)| \geq \frac{\Omega^2}{2}|z|^2 - \frac{\delta}{k} \geq \frac{\Omega^2}{2}|z|^2 -   \frac{\Omega^2}{8} \\
& \geq  \frac{\Omega^2}{2}|z|^2  -  \frac{\Omega^2}{8}|z|^2 = \frac{3\Omega^2}{8}|z|^2. \\
\end{aligned}
\end{align}
\noindent $\diamond$~Outside the support of $U$, the potential $W_{\rm rs} = \frac{\Omega^2}{2}|z|^2$ is exactly harmonic. \newline

\noindent $\diamond$~On the support of $U$, the relations \eqref{E-6} and \eqref{E-7} yield
\begin{align}
\begin{aligned} \label{E-9}
W_{\rm rs}(z) &\leq  \frac{\Omega^2}{2}|z|^2 +  |u(|z|)| \leq \frac{\Omega^2}{2}|z|^2  +  \frac{\delta}{k} \leq  \frac{\Omega^2}{2}|z|^2  +  \frac{\Omega^2}{8} \\
&\leq \frac{\Omega^2}{2}|z|^2  +  \frac{\Omega^2}{8} |z|^2 = \frac{5\Omega^2}{8}|z|^2.
\end{aligned}
\end{align}
Finally, we combine \eqref{E-8} and \eqref{E-9} to get 
\begin{equation} \label{E-10}
\frac{3\Omega^2}{8}|z|^2 \leq W_{\rm rs}(z) \leq  \frac{5\Omega^2}{8}|z|^2.
\end{equation}
Thus, the relation  $\eqref{E-4}_2$ holds with 
\[ \underbar{c}_W = \frac{3\Omega^2}{8} \quad \mbox{and} \quad \bar{c}_W = \frac{5\Omega^2}{8}. \]

\noindent $\bullet$~Case 3 (verification of $\eqref{E-4}_3$):~We take a gradient of \eqref{E-3} with respect to $z$ to get 
\begin{equation} \label{E-11}
\nabla W_{\rm rs} = \nabla \Big(  \frac{\Omega^2}{2}|z|^2+U(z) \Big ) = \Omega^2 z + u^{\prime}(|z|) \frac{z}{r}.  
\end{equation}
\noindent $\diamond$ Outside of the support of $U$,  
\[ W_{\rm rs}(z) = \frac{\Omega^2}{2}|z|^2 \quad \mbox{and} \quad \nabla W_{\rm rs}(z) = \Omega^2 z. \]
This yields
\begin{equation} \label{E-12}
 z\cdot\nabla W_{\rm rs}(z) = \Omega^2 |z|^2 = 2 W_{\rm rs}(z).
\end{equation} 
\noindent $\diamond$~On the support of $U$, we use  \eqref{E-10}, \eqref{E-5} and \eqref{E-11} to get 
\begin{align}
\begin{aligned} \label{E-13}
z\cdot\nabla W_{\rm rs}(z) &= \Omega^2 |z|^2 + r u^{\prime}(|z|) =   \Omega^2 |z|^2 +  \frac{\delta}{k}  r \chi^{\prime}(r) \cos(kr) - \delta \chi(r) r\sin (kr) \\
&\geq \Omega^2 |z|^2 -  \Big|  \frac{\delta}{k}  r \chi^{\prime}(r) \cos(kr) \Big| - \Big|  \delta \chi(r) r\sin (kr) \Big| \\
& \geq  \Omega^2 |z|^2 -  \Big( \frac{\delta}{k}   \| \chi^{\prime} \|_{\infty}  +  \delta \Big) r  \geq \Omega^2 |z|^2 -  \Big( \frac{\delta}{k}   \| \chi^{\prime} \|_{\infty}  +  \delta \Big) r^2  \\
& \geq \Omega^2 |z|^2  - \frac{\Omega^2}{4} |z|^2 =  \frac{6\Omega^2}{8} |z|^2 = \frac{6}{5} \times \Big( \frac{5 \Omega^2}{8} |z|^2  \Big) \geq  \frac{6}{5} W_{\rm rs}(z),
\end{aligned}
\end{align}
where we used the fact $r^2 \geq r$ on the support of $U$. Finally, we combine \eqref{E-12} and \eqref{E-13} to get the virial estimate with $\nu_W = \frac{6}{5}$:
\[ z\cdot\nabla W_{\rm rs}(z) \geq \frac{6}{5} W_{\rm rs}(z). \]

\noindent $\bullet$~Case 4 (verification of $\eqref{E-4}_4$): On the outside of the support of $U$, we have
\[  \partial_{x_j} \partial_{x_i} W_{\rm rs}(x) = \Omega^2 \delta_{ij}. \]
On $\mbox{supp}(U) \subset \{x\in\mathbb R^2:1\leq |x|\leq2\}$, we have
\[
 \partial_{x_j} \partial_{x_i} W_{\rm rs}(x) = \Omega^2 \delta_{ij} +  \partial_{x_j} \partial_{x_i} u(|x|)  = \Omega^2 \delta_{ij} +  u^{\prime \prime}(r) \frac{x_i x_j}{r^2}
 +  u^{\prime}(r) \Big(  \frac{\delta_{ij}}{r} - \frac{x_i x_j}{r^3}  \Big),
\]
which is bounded. Here $\delta_{ij}$ is the Kronecker-delta function defined as follows.
\[ \delta_{ij} = \begin{cases}
1, \quad & i = j, \\
0, \quad & i \neq j.
\end{cases}
\]
Therefore, we have
\[  \|D^2W_{\rm rs} \|_{L^{\infty}} < \infty.\]
\end{proof}
\begin{remark}\label{NR5.1}
	    The construction in Lemma \ref{L5.1} can also be applied to the general case $d\ge2$. 
\end{remark}

\subsection{Explicit example for optimal decay} \label{sec:5.2}
In this subsection, we construct a concrete example to show that the polynomial decay rate in  Theorem \ref{T3.3} (ii) is in fact optimal. We first decompose the index set ${\mathbb Z}$ as follows:
\[ {\mathbb Z} = \{(n, +):~n \in \bbn \} \cup \{0 \}  \cup \{ (n,-):~ n \in \bbn \}.  \]
Then, we consider the Cauchy problem for the infinite CS model with confining potential $W_{\rm rs}$:
\begin{equation}\label{E-14}
	\begin{cases}
\displaystyle \dot x_i=v_i, \quad  t > 0, \quad i \in {\mathbb Z},  \vspace{6pt}\\
\displaystyle\dot v_i=\kappa\sum_{j \in \bbz} m_j\phi_\beta(|x_i-x_j|)(v_j-v_i)
-\sum_{j \in \bbz} m_j\nabla W_{\rm rs}(x_i-x_j),  \vspace{6pt}\\
\displaystyle  (x_i, v_i) \Big|_{t = 0} = (x_{i}^0, v^0_{i}).
	\end{cases}
\end{equation}
For a given infinite initial configuration $\{ (x^0_i, v^0_i) \}_{i \in \bbz}$, we define the empirical measure:
\[
\mu_0:=\sum_{i\in \bbz}m_i\delta_{(x_i^0,v_i^0)}.
\]
Then, it follows from Proposition \ref{P5.1} below that the Cauchy problem \eqref{E-14} admits a global solution:
\[
\{(x_i(t),v_i(t))\}_{i\in \bbz},
\]
and hence it generates a measure-valued solution to KCS model \eqref{A-4}:
\[
\mu_t:=\sum_{i\in \bbz}m_i\delta_{(x_i(t),v_i(t))}.
\]
For the centered configuration, we define the discrete energy at time $t$:
\begin{equation*}\label{E-15}
\mathscr E_\infty(t)
:=\sum_{i \in \bbz} m_i|v_i(t)|^2
+\sum_{i,j \in \bbz} m_im_jW_{\rm rs}(x_i(t)-x_j(t)), \quad t \geq 0.
\end{equation*}
\begin{proposition}[Countable dynamics]\label{P5.1}
Suppose that initial configuration satisfies 
\[ \sum_{i \in \bbz} m_i(|x_i^0|^2+|v_i^0|^2)<\infty. \]
Then, the following assertions hold. 
\vspace{0.1cm}
\begin{enumerate}
\item
Symmetric finite truncations have a subsequence converging coordinatewise and locally uniformly to a global solution to \eqref{E-14}. 
\vspace{0.2cm}
\item
For the symmetric masses $m_{n,+}=m_{n,-}$ and symmetric initial data, a solution satisfies
\begin{equation}\label{E-16}
\mathscr E_\infty(t)+\kappa\int_0^t
\sum_{i,j \in \bbz} m_im_j\phi_\beta(|x_i(s) -x_j(s)|)|v_i(s) -v_j(s)|^2\dd s
\leq\mathscr E_\infty(0),
\end{equation}
and it preserves the symmetries:
\[
x_{n,-}(t) =-x_{n,+}(t),
\quad
v_{n,-}(t)=-v_{n,+}(t),~~ n \in \bbn,
\quad
x_0=v_0=0.
\]
\end{enumerate}
\end{proposition}
\begin{proof}
Since the proof is very lengthy, we postpone it to Appendix \ref{app-D}.
\end{proof}

\noindent For a fixed \(q>1\), we first choose the total size of the exterior
shells, and we set 
\begin{equation}\label{E-17}
	0<\eta\leq \frac{1}{2\sum_{n\geq1} 1/n^{2}} = \frac{3}{\pi^2}.
\end{equation}
After the radii \(R_n\geq1\) have been chosen, we define
\begin{equation*}\label{E-18}
	m_n:=\frac{\eta}{n^2R_n^{2q}},
	\quad
	m_{n,+}=m_{n,-}:=\frac{m_n}{2},
	\quad
	m_0:=1-\sum_{n\geq1}m_n.
\end{equation*}
Since \(R_n\geq1\), we use \eqref{E-17} to see
\[
\sum_{n\geq1}m_n
\leq
\eta\sum_{n\geq1}\frac1{n^2}
\leq\frac12.
\]
Therefore, we have
\[  m_0\geq1/2. \]
We next prescribe the symmetric rotating initial configuration:
\begin{equation*}\label{E-19}
	x_{n,+}^0=R_ne_1,
	\quad
	v_{n,+}^0=\Omega R_ne_2,
	\quad
	(x_{n,-}^0,v_{n,-}^0)
	=-(x_{n,+}^0,v_{n,+}^0), \quad x_0^0=v_0^0=0.
\end{equation*}
The initial energy admits an upper bound independent of the
particular choice of the radii. Indeed, we use the upper quadratic bound:
\[
W_{\rm rs}(z)\leq \bar{c}_W |z|^2
\]
and the fact that the weighted center is zero to obtain
\begin{align*}
	\mathscr E_\infty(0)
	&\leq
	\Omega^2\sum_{n\geq1} m_n R_n^2
	+ \bar{c}_W \sum_{i,j}m_im_j|x_i^0-x_j^0|^2 =
	 \Big (\Omega^2 +2 \bar{c}_W \Big)\sum_{n\geq1}m_nR_n^2 \\
	&=
	 \Big (\Omega^2 +2 \bar{c}_W \Big )\eta
	\sum_{n\geq1}\frac{1}{n^2R_n^{2(q-1)}}
	\leq \Big( \Omega^2 +2 \bar{c}_W \Big )\eta\sum_{n\geq1}\frac1{n^2}.
\end{align*}
We therefore fix
\begin{equation*}\label{E-20}
	\overline E
	:= \Big (\Omega^2 +2 \bar{c}_W \Big )\eta\sum_{n\geq1}\frac1{n^2}.
\end{equation*}
With the above preparations, Figure~\ref{fig:1} provides a roadmap for the proof of Theorem \ref{T3.4}. We now proceed by following this roadmap.\newline 
	\begin{figure}[htbp]
	\centering
	\includegraphics[width=1\textwidth]{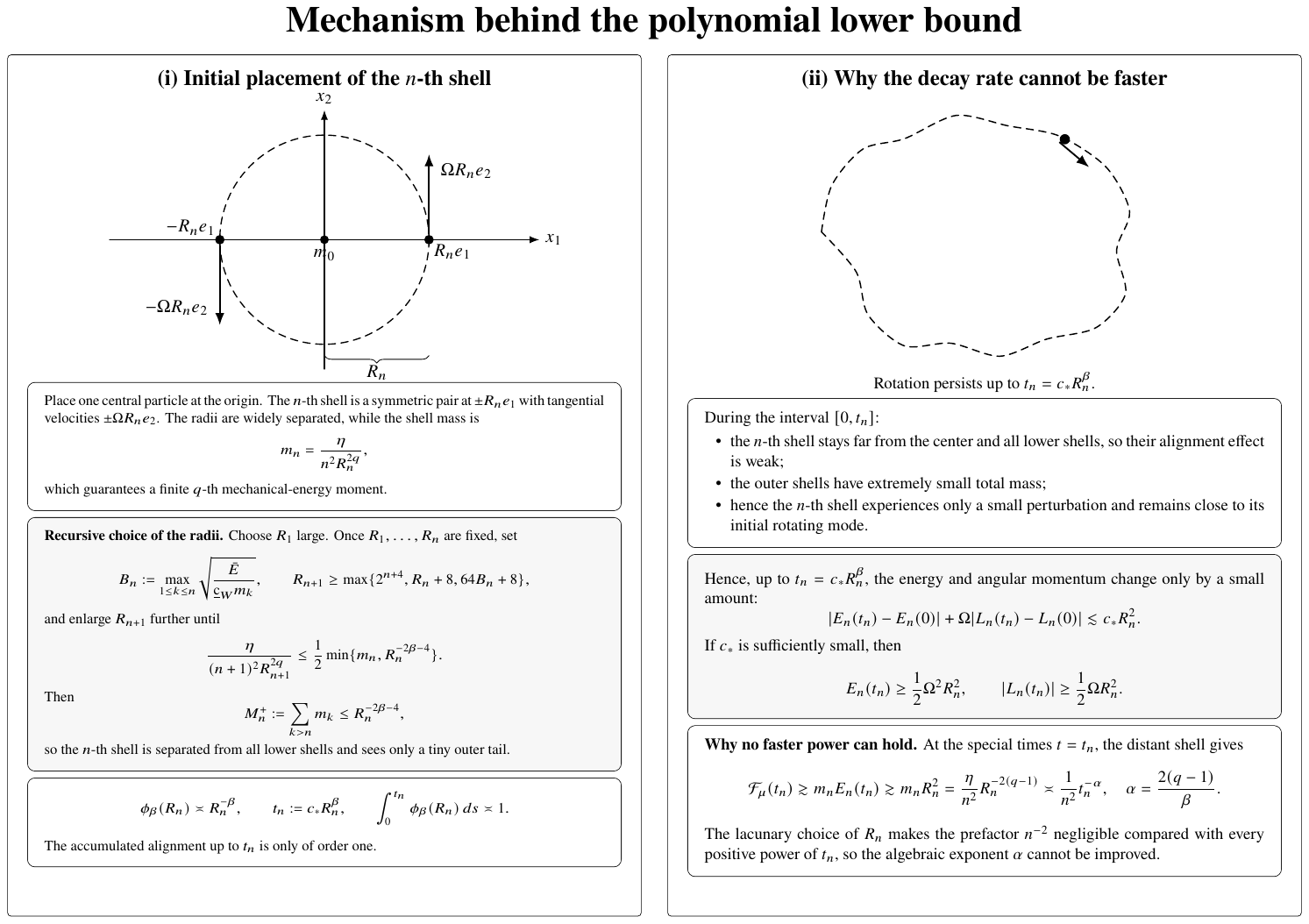}
	\caption{Mechanism behind the polynomial convergence lower bound.}
	\label{fig:1}
\end{figure}

The radii will now be chosen recursively. The following lemma records two properties needed later: separation from all lower shells and
smallness of the total mass of all higher shells.

\begin{lemma}[Recursive choice of lacunary scales]
	\label{L5.2}
	There exists an increasing sequence $(R_n)$  diverging to $\infty$ as $n \to \infty$ such that
	\begin{align}
	\begin{aligned} \label{E-21}
	&(i)~R_1\geq16, \quad
		R_n\geq2^{n+3}, \quad
		R_{n+1}\geq R_n+8. \\
	& (ii)~R_n \geq
		64\max_{1\leq k<n}
		\sqrt{\frac{\overline E}{\underbar{c}_W m_k}}
		+8, \quad n \geq 2. \\
      & (iii)~M_n^+ :=
		\sum_{k>n}m_k
		\leq
		R_n^{-2\beta-4},
		\quad n \geq 1. \\
     & (iv)~|R_n-R_m| \geq 8, \quad n \neq m.		
	\end{aligned}
	\end{align}
\end{lemma}
\begin{proof}
We first construct a required sequence $(R_n)$ satisfying the first three conditions \eqref{E-21} inductively, and then we show the constructed sequence $(R_n)$ satisfies the last condition in \eqref{E-21}. 
\vspace{0.1cm}

\noindent $\bullet$~Step A (Construction of $(R_n)$ satisfying the first three conditions \eqref{E-21}): Below, we construct the desired sequence $(R_n)$ satisfying the first three conditions in \eqref{E-21} inductively. \newline

\noindent $\diamond$~Step A.1 (Initial step):~We choose $R_1$ such that 
\[  R_1\geq16. \]

\noindent $\diamond$~Step A.2 (Inductive step):~Suppose that a finite sequence $(R_1,\ldots,R_n)$  has already been chosen to satisfy a set of the first two conditions in \eqref{E-21}.  At this stage, all quantities
	\[
	m_1,\ldots,m_n~\mbox{are fixed}.
	\]
Now, we introduce 	
	\begin{equation}\label{E-22}
		B_n
		:=
		\max_{1\leq k\leq n}
		\sqrt{\frac{\overline E}{ \underbar{c}_W m_k}}
		\quad \mbox{and} \quad 
		a_n
		:=
		\frac12\min\left\{
		m_n,\,
		R_n^{-2\beta-4}
		\right\}.
	\end{equation}
Next, we choose \(R_{n+1}\) sufficiently large so that
	\begin{equation}\label{E-23}
		R_{n+1}
		\geq
		\max\left\{
		2^{n+4},\,
		R_n+8,\,
		64B_n+8
		\right\} \quad \mbox{and} \quad m_{n+1}
		:=
		\frac{\eta}{(n+1)^2R_{n+1}^{2q}}
		\leq a_n.
	\end{equation}
Such a choice is always possible, because all quantities on the
	right-hand sides of
	\eqref{E-22}--\eqref{E-23}
	are already fixed, whereas we have
	\[
	 \lim_{R \to \infty} \frac{\eta}{(n+1)^2R^{2q}} = 0.
	\]
	The conditions (i) and (ii) in 
	\eqref{E-21} follow directly from
	\eqref{E-23}. Moreover, the relations 
	\eqref{E-23} imply
	\begin{equation}\label{E-24}
		m_{n+1}\leq\frac12m_n \quad \mbox{and} \quad m_{n+1}
		\leq
		\frac12R_n^{-2\beta-4}.
	\end{equation}
	Applying \eqref{E-24} repeatedly, we find
	\[
	m_{n+j}\leq2^{-(j-1)}m_{n+1},
	\qquad j\geq1.
	\]
	Consequently, we have $\eqref{E-21}_3$:
	\begin{align*}
		M_n^+
		=
		\sum_{j\geq1}m_{n+j}
		&\leq
		m_{n+1}\sum_{j\geq1}2^{-(j-1)}
		=
		2m_{n+1} \leq
		R_n^{-2\beta-4}.
	\end{align*}
\noindent $\bullet$~Step B (the sequence $(R_n)$ satisfies separation property $\eqref{E-21}_4$):~The center is at
	distance \(R_n\geq16\) from the \(n\)-th shell, and two particles
	on the same shell are separated by \(2R_n\geq32\). \newline
	
\noindent If \(k<n\), then
	particles on the same side satisfy
	\[
	|R_n-R_k|\geq8,
	\]
	while particles on opposite sides satisfy
	\[
	R_n+R_k\geq32.
	\]
	Thus every two distinct initial particles are separated by at least
	\(8\).
\end{proof}
\vspace{0.2cm}

\noindent 
\textbf{Proof of Theorem \ref{T3.4} (1).} Now, we set 
	\[
	M_{2,0}:=\sum_i m_i|x_i^0|^2.
	\]
	By the definition of the mass $(m_n)$ in $\eqref{E-23}_2$,
	\[
	M_{2,0}
	=
	\sum_{n\geq1}m_nR_n^2
	=
	\eta\sum_{n\geq1}
	\frac{1}{n^2R_n^{2(q-1)}}
	\leq
	\eta\sum_{n\geq1}\frac1{n^2}<\infty.
	\]
	Since the weighted spatial center is zero, we have
	\[
	\sum_jm_j|x_i^0-x_j^0|^2
	=
	|x_i^0|^2+M_{2,0}.
	\]
	Using the upper quadratic bound for \(W_{\rm rs}\), we obtain
	\[
	h_{\mu_0}(x_i^0,v_i^0)
	\leq
	\frac12|v_i^0|^2
	+ \bar{c}_W \bigl(|x_i^0|^2+M_{2,0}\bigr).
	\]
	For the \(n\)-th symmetric pair, we have
	\[
	1+h_{\mu_0}(x_{n,\sigma}^0,v_{n,\sigma}^0)
	\leq C(1+R_n^2).
	\]
	Consequently, we use \eqref{E-23} to get
\begin{align*}
\int_{\bbr^{2d}}(1+h_{\mu_0}(z))^q\dd\mu_0(z)
&\leq C+C\sum_{n\geq1}m_n(1+R_n^2)^q\\
&\leq C+C\sum_{n\geq1}m_nR_n^{2q}=C+C\eta\sum_{n\geq1}\frac1{n^2}<\infty.
\end{align*}
This completes the proof of Theorem \ref{T3.4} (1).

\subsection{Proof of Theorem \ref{T3.4}} \label{sec:5.3}
In this subsection, we provide a proof of second and third assertions in Theorem \ref{T3.4}. \newline 

 For each \(n\geq1\), we set
\[
y_n(t):=x_{n,+}(t),
\quad
z_n(t):=v_{n,+}(t),
\]
and define the oscillator energy and angular momentum:
\begin{equation*}\label{E-25}
	E_n(t)
	:=
	\frac12\Bigl(|z_n(t)|^2+\Omega^2|y_n(t)|^2\Bigr),
	\quad
	L_n(t)
	:=
	y_n(t)\wedge z_n(t).
\end{equation*}
Here, for \(a=(a_1,a_2),b=(b_1,b_2)\in\mathbb R^2\),
\[
a\wedge b:=a_1b_2-a_2b_1.
\]
\begin{proposition}[Persistence of the rotating modes]
	\label{P5.2}
	For \(\beta>0\), let $(R_n)$ be an increasing sequence whose existence is guaranteed by Lemma~\ref{L5.2}.  Then there exists a constant \(c_*>0\) independent of \(n\) such
	that
	\begin{equation}\label{E-26}
		t_n:=c_*R_n^\beta, \quad	E_n(t)\geq\frac12\Omega^2R_n^2,
		\quad
		|L_n(t)|\geq\frac12\Omega R_n^2,
		\quad
		0\leq t\leq t_n.
	\end{equation}
\end{proposition}

\begin{proof}
	Throughout the proof, the constant \(C>0\) denotes a generic constant depending only on
	the fixed parameters of the system, but not on \(n\). We split the proof into five steps.
	\begin{itemize}
		\item \textbf{Step A} (Bootstrap interval and annular bounds):~
		The lower bound on the angular momentum and the upper bound on the
		oscillator energy imply that both the position and velocity of the
		\(n\)-th mode remain in an annular region of size \(R_n\).
		
		\vspace{0.2cm}
		
		\item \textbf{Step B} (Separation from the center and all lower modes):~
		The gap of the radii, together with the global energy bound,
		separates the \(n\)-th mode from the center, its symmetric partner,
		and all lower modes. Consequently, the compactly supported
		perturbation \(U\) vanishes on all these interactions.
		
		\vspace{0.2cm}
		
		\item \textbf{Step C} (Equation for the \(n\)-th mode):
		Using conservation of total mass and of the weighted center, the
		harmonic part of the interaction produces the exact restoring force
		$
		-\Omega^2 y_n.$
		Thus the dynamics of the \(n\)-th mode can be written as a perturbed
		harmonic oscillator.
		
		\vspace{0.2cm}
		
		\item \textbf{Step D} (Estimate of the non-harmonic acceleration):
		We estimate the non-harmonic acceleration by separately treating:
		\begin{enumerate}
			\item the alignment interaction with the center and lower modes;
			\item the alignment interaction with the higher modes;
			\item the compact perturbation generated by the higher modes.
		\end{enumerate}
		
		\vspace{0.2cm}
		
		\item \textbf{Step E} (Variation of energy and angular momentum):
		The above bound implies that the oscillator energy and angular
		momentum vary by at most \(\mathcal{O}(R_n^2)\) on a time interval of length
		\(\mathcal{O}(R_n^\beta)\). Choosing the proportionality constant sufficiently
		small strictly improves the bootstrap assumptions and closes the
		argument.
	\end{itemize}
	\vspace{0.2cm} 
	The detailed arguments for each step will be given in Appendix \ref{app-E}. 	
\end{proof}
Therefore, the empirical measure generated by system \eqref{E-14} has zero first moments and remains fully noncompact at every finite time because all sufficiently large shells are still in their persistence intervals.  Moreover, we use the zero weighted center and zero momentum to find
\[
\cF_\mu(t)=2\sum_im_i\bigl(|x_i(t)|^2+  |v_i(t)|^2\bigr).
\]
At $t_n=c_*R_n^\beta$, we use Proposition~\ref{P5.2} and $\eqref{E-23}_2$ to obtain the second desired estimate \eqref{C-22}:
\[
\cF_\mu(t_n)\geq cm_n E_n(t_n)\geq cm_nR_n^2
=\frac{c\eta}{n^2}R_n^{-2(q-1)}
=\frac{c'}{n^2}t_n^{-2(q-1)/\beta}.
\]
Since $R_n\geq2^{n+3}$, multiplication by $t_n^\delta$ proves the third desired estimate \eqref{C-23}. \newline

For every $d>2$, the same construction is embedded in a fixed two-dimensional coordinate plane, with the radial potential defined on $\bbr^d$. This verifies the optimality assertion for all $d\geq2$.

\section{Conclusion}\label{sec:6}
In this paper, we have studied the KCS model with an attractive nonconvex potential in a phase-spatially extended setting. Since the spatial and velocity diameters may be infinite, classical spatial and velocity support-based arguments are no longer applicable. To overcome this difficulty, we have combined the method of time-varying effective region with a compensated Lyapunov functional. We first established the global existence of Lagrangian weak solutions for centered initial data with finite second moments. Under an exponential mechanical-energy tail, we further obtained an Osgood-type stability estimate in $\mathcal{W}_1$ and hence uniqueness. For a long-time behavior, we have considered initial data with polynomial and exponential mechanical-energy tails. In the polynomial case, we derived the optimal algebraic decay exponent and proved its sharpness by a countably infinite rotating-shell construction. In the exponential case, a two-stage localization argument yielded an optimal exponential decay scale. Our analysis shows that uniform convexity is not essential. Quadratic confinement, a bounded Hessian, and a strict virial inequality are sufficient even when the Hessian changes sign. In particular, our results can also be applied to the hydrodynamic CS model in \cite{ShuTadmor2020,C-C-K-T-2025},  although the global existence of a classical solution in the general high-dimensional hydrodynamic CS model is still open (see \cite{Tadmor2021}). Of course, there are several remaining interesting problems, e.g.,  uniqueness of a weak solution under polynomial tails, and extensions to stochastic and relativistic models. These questions will be left for future work. 
\vspace{.3cm}

\section*{Conflict of interest statement}
The authors declare no conflicts of interest.

\section*{Data availability statement}
The data supporting the findings of this study are available from the corresponding author upon reasonable request.

\section*{Ethical statement}
The authors declare that this manuscript is original, has not been published before, and is not currently being considered for publication elsewhere. The study was conducted by the principles of academic integrity and ethical research practices. All sources and contributions from others have been properly acknowledged and cited. The authors confirm that there is no fabrication, falsification, plagiarism, or inappropriate manipulation of data in the manuscript.
\vspace{1cm}

\appendix
\section{Truncation and justification of the identities}
\label{app-A}
In this appendix, we justify moment identities, energy
dissipation law, and differentiated compensated functional used in
the main text.  Since the solutions under consideration need not be
compactly supported, the relevant identities are first established with
compactly supported test functions and we recover the desired estimates by a
truncation argument as in Lemma 3.1 of \cite{H-W-CS-2026}.  Throughout this appendix, we write
\[
z=(x,v)\in\bbr^{2d},
\qquad
\langle z\rangle^2:=1+|x|^2+|v|^2.\]
Let $\zeta\in \mathcal C_c^\infty(\bbr^{2d})$ satisfy
\[
0\leq\zeta\leq1,
\quad
\zeta\equiv1\ \hbox{on }B_1,
\quad
\zeta\equiv0\ \hbox{outside }B_2,
\]
and define
\[
\zeta_R(z):=\zeta(z/R).
\]
Then, we have
\[
|\nabla\zeta_R(z)|
\leq \frac{C}{R}
\mathbf 1_{{R\leq|z|\leq2R}}.
\]
Let $P=P(x,v)$ be one of
\[
1,\quad x,\quad v,\quad
|x|^2,\quad |v|^2,\quad x\cdot v.
\]
The vector-valued observables $x$ and $v$ are tested componentwise. Testing the weak formulation of the KCS model with
\[
\psi_R(t,z)=\eta(t)P(z)\zeta_R(z),
\qquad
\eta\in {\mathcal C}_c^1((0,T)),
\]
produces the expected identity for $P$ together with cutoff errors.
Using the finite-time second-moment bound and the linear-growth
estimate for the force, the cutoff errors are bounded by
\begin{equation}\label{A.1}
	C_T
	\int_0^T
	\int_{{R\leq|z|\leq2R}}
	(1+|z|^2)f_t(z)\dd z\dd t.
\end{equation}
Note that the quantity in
\eqref{A.1} tends to zero as $R\to\infty$.
Consequently, all moment identities obtained from the above polynomial
test functions are valid in the sense of distributions in time.  The
corresponding moment functions are locally absolutely continuous, and
their differentiated identities hold for a.e. $t>0$. Therefore, all time derivatives of  moments,
energies, and compensated functionals are understood in the following
sense. 
\begin{itemize}
	\item The corresponding identities are first derived using compactly
	supported truncations of the relevant observables.
	\item The cutoff is then
	removed by means of the propagated uniform integrability of the second
	moments and the linear-growth bounds on the force.
	\item Then, we show that the resulting moment and energy functionals are locally absolutely continuous on $[0,\infty)$.
	\item  Moreover, the differential identities
	therefore hold for almost every time, while the corresponding
	integrated identities hold for every \(0\leq s\leq t\).
	\item Furthermore, the derivatives
	along characteristics are likewise understood as almost-everywhere
	derivatives of locally absolutely continuous characteristic
	quantities.
\end{itemize}   

\vspace{0.5cm}

\section{Proofs of  Gr\"onwall-type lemmas}\label{app-B}
In this appendix, we provide proofs of Gr\"onwall-type lemmas in Section \ref{sec:2.3}.

\subsection{Proof of Lemma \ref{L2.5}}   \label{app-B-1}
Let \(y :[t_0,\infty)\to \bbr_+\) be a locally absolutely continuous function such that 
\[
y'(t)\leq-c(1+t)^{-\alpha} y(t)+C(1+t)^{-\alpha-\lambda}, \quad  \mbox{a.e.}~t > t_0,
\]
where parameters are constants satisfying 
\[  c > 0, \quad  C > 0,  \quad 0 \leq \alpha < 1, \quad \lambda > 0. \]
We claim that there exist positive constants ${\tilde C} = {\tilde C}(c, C, y(t_0) )$ and $t_0 = t_0(c, \alpha, \lambda)$ such that 
\begin{equation} \label{B.1}
y(t)\leq {\tilde C} (1+t)^{-\lambda}, \quad t \geq t_0.
\end{equation}
{\it Derivation of \eqref{B.1}}:~We set 
\[ {\mathfrak y}(t)= {\tilde C}(1+t)^{-\lambda}, \]
where ${\tilde C}$ is a positive constant to be determined later.  Then $\mathfrak{y}$ satisfies
\begin{equation} \label{B.2}
{\mathfrak y}'+c(1+t)^{-\alpha} {\mathfrak y} -C(1+t)^{-\alpha-\lambda}
=(1+t)^{-\alpha-\lambda} \Big (c{\tilde C}-C-{\tilde C} \lambda (1+t)^{\alpha-1} \Big).
\end{equation}
Now, we choose constants ${\tilde C} \gg 1$ and $t_0 \gg 1$ to satisfy 
\begin{equation} \label{B.3}
\frac{c{\tilde C}}{2} > C, \quad {\mathfrak y}(t_0)\geq y(t_0), \quad \mbox{and} \quad   \lambda (1+t)^{\alpha-1}\leq \frac{c}{2}, \qquad \forall~t \geq t_0,
\end{equation}
where we used the relations:
\[  \alpha<1 \quad \mbox{and} \quad  \lim_{t \to \infty} (1+t)^{\alpha-1} = 0. \]
Then, it follows from \eqref{B.2} and \eqref{B.3} that
\[ {\mathfrak y}'+c(1+t)^{-\alpha} {\mathfrak y} -C(1+t)^{-\alpha-\lambda} \geq 0, \]
i.e., ${\mathfrak y}$ is a super-solution. Finally, we use the comparison principle to find 
\[ y(t)\leq \mathfrak{y}(t) = {\tilde C}(1+t)^{-\lambda}, \quad \mbox{for $t\geq t_0$}. \]
This completes the proof.

\subsection{Proof of Lemma \ref{L2.6}}   \label{app-B-2}
Let \(y :[t_0,\infty)\to \bbr_+\) be a locally absolutely continuous function such that 
\begin{equation} \label{B.4}
y'(t) + c_0(1+t)^{-\alpha} y(t) \leq C_0(1+t)^{-\alpha}e^{-\sigma(1+t)^\lambda}, \quad t > t_0,
\end{equation}
where parameters are constants satisfying 
\[  c_0 > 0, \quad  C_0 > 0,  \quad \sigma > 0,  \quad 0 \leq \alpha < 1, \quad \lambda > 0. \]
We claim that 
\begin{equation}\label{B.5}
y(t)\leq {\tilde C}_0 \exp\left[-{\tilde c}_0 (1+t)^{\min\{1-\alpha,\lambda \}}\right], \quad t \geq t_0,
\end{equation}
where ${\tilde c}_0$ and ${\tilde C}_0$ are positive constants to be determined to satisfy the relations:
\begin{equation} \label{B.6}
\frac{c_0 {\tilde C}_0}{2} > C_0 \quad \mbox{and} \quad  0 < {\tilde c}_0 < \sigma.
\end{equation}
\noindent {\it Derivation of \eqref{B.5}}:~We set 
\[ \omega:=\min\{1-\alpha,\lambda \} \]
and we choose a sufficiently small $\delta>0$ such that 
\[ \delta<  \sigma \quad \mbox{and} \quad  \delta \omega< \frac{c_0}{2}. \]
We set
\begin{equation} \label{B.7}
{\mathfrak y}(t)= {\tilde C}_0 e^{-\delta(1+t)^\omega} = {\tilde C}_0 e^{-\delta(1+t)^{ \min\{1-\alpha,\lambda \}}}.
\end{equation}
Then, it is easy to see that 
\begin{equation} \label{B.8}
{\mathfrak y}^{\prime}(t) = - \delta \omega(1 + t)^{\omega-1}  {\tilde C}_0   e^{-\delta(1+t)^\omega} = - \delta \omega(1 + t)^{\omega-1} {\mathfrak y}(t).
\end{equation}
We use $\omega\leq1-\alpha$ and \eqref{B.8} to get 
\begin{equation} \label{B.9}
| {\mathfrak y}'(t)|\leq \delta\omega(1+t)^{-\alpha} {\mathfrak y}(t).
\end{equation}
Next, we use \eqref{B.8} and \eqref{B.9} to find 
\begin{align}
\begin{aligned} \label{B.10}
{\mathfrak y}'(t)+ c_0(1+t)^{-\alpha} {\mathfrak y}(t) &\geq (c_0 -\delta\omega) (1+t)^{-\alpha} {\mathfrak y}(t) \\
&  \geq \frac{c_0}{2}(1+t)^{-\alpha} {\mathfrak y}(t) = \frac{c_0}{2} {\tilde C}_0 (1+t)^{-\alpha}  e^{-\delta(1+t)^\omega}, \quad t \gg 1.
\end{aligned}
\end{align}
On the other hand, since $\omega\leq \lambda~\mbox{and} ~\delta< \sigma,$ we have
\begin{equation} \label{B.11}
e^{-\sigma(1+t)^\lambda}\leq e^{-\delta(1+t)^\omega}, \quad t \gg 1.
\end{equation}
Now, we use \eqref{B.4}, \eqref{B.6}, \eqref{B.7}, \eqref{B.10} and \eqref{B.11} to obtain  
\begin{align}
\begin{aligned} \label{B.12}
& y'(t) + c_0(1+t)^{-\alpha} y(t)  \\
& \hspace{0.2cm} \leq C_0(1+t)^{-\alpha}e^{-\sigma(1+t)^\lambda} \leq  C_0(1+t)^{-\alpha} e^{-\delta(1+t)^\omega} \leq  \frac{c_0}{2} {\tilde C}_0 (1+t)^{-\alpha}  e^{-\delta(1+t)^\omega} \\
& \hspace{0.2cm} \leq {\mathfrak y}'(t)+ c_0(1+t)^{-\alpha} {\mathfrak y}(t), \quad t \gg 1.
\end{aligned}
\end{align}
On the other hand, we choose ${\tilde C}_0$ sufficiently large and set $\delta$  such that 
\begin{equation} \label{B.13}
{\mathfrak y}(t_0) = {\tilde C}_0 e^{-\delta(1+t_0)^{ \min\{1-\alpha,\lambda \}}}  \geq y(t_0), \quad \mbox{and} \quad \delta = {\tilde c}_0.
\end{equation}
Then, it follows from \eqref{B.12} and \eqref{B.13} that 
\[ y(t) \leq {\mathfrak y}(t) = {\tilde C}_0 e^{-\delta(1+t)^{ \min\{1-\alpha,\lambda \}}}  =  {\tilde C}_0 e^{-{\tilde c}_0 (1+t)^{ \min\{1-\alpha,\lambda \}}}, \quad t \gg 1. \]
This completes the proof.
\subsection{Proof of Lemma \ref{L2.7}} \label{app-B-3}
Let \(y :[t_0,\infty)\to \bbr_+\) be a locally absolutely continuous function such that 
\begin{equation}\label{B.14}
y'(t) + \frac{c_1}{{1+t}} y(t) \leq \frac{C_1}{1+t} g(t), \quad t > t_0.
\end{equation}
We multiply \eqref{B.14} by an integrating factor $(1+t)^{c_1}$ to get 
\[
\frac{\dd}{\dd t}\bigl((1+t)^{c_1} y(t)\bigr)
\leq C_1(1+t)^{c_1-1}g(t).
\]
We integrate the above relation from $t_0$ to $t$ to obtain
\begin{align}
\begin{aligned} \label{B.15}
y(t) &\leq  y(t_0) \Big( \frac{1 + t_0}{1 + t} \Big)^{c_1}  + C_1  (1 + t)^{-c_1} \int_{t_0}^{t} (1 + s)^{c_1 - 1} g(s) \dd s \\
& =: y(t_0)\Big( \frac{1 + t_0}{1 + t} \Big)^{c_1}  + C_1  (1 + t)^{-c_1} {\mathcal J}(t), \quad t \geq t_0.
\end{aligned}
\end{align}
Next, we consider two cases to estimate ${\mathcal J}(t)$.  \newline

\noindent $\bullet$~Case 1: Suppose that the source $g$ satisfies 
\[ g(t)\leq C_2(1+t)^{-\lambda}. \]
\noindent $\diamond$~Case 1.1 $(c_1 \neq \lambda)$: If $c_1\neq \lambda$, the integral ${\mathcal J}$ satisfies 
\begin{align}
\begin{aligned} \label{B.16}
{\mathcal J}(t) &=  \int_{t_0}^{t} (1 + s)^{c_1 - 1} g(s) \dd s \leq C_2  \int_{t_0}^{t} (1 + s)^{c_1 - \lambda - 1} \dd s \\
&= \frac{C_2}{c_1 - \lambda} \Big[  (1 + t)^{c_1 - \lambda}  - (1 + t_0)^{c_1 - \lambda}   \Big ].
\end{aligned}
\end{align}
By \eqref{B.15} and \eqref{B.16}, we have
\begin{equation*} \label{B.17}
y(t)  \leq y(t_0) \Big( \frac{1 + t_0}{1 + t} \Big)^{c_1}  +  \frac{C_1 C_2}{c_1 - \lambda} \Big[  (1 + t)^{- \lambda}  -  (1 + t_0)^{-\lambda} \Big(  \frac{1 + t_0}{1 + t} \Big)^{c_1} \Big ].
\end{equation*}
\vspace{0.2cm}
~We use 
\[ (1 + t)^{-c_1} \leq (1 + t)^{- \min\{c_1, \lambda \}}, \quad (1 + t)^{- \lambda} \leq (1 + t)^{- \min\{c_1, \lambda \}} \]
to get 
\begin{equation} \label{B.18}
y(t) \leq \Big[  y(t_0)(1 + t_0)^{c_1} +    \frac{C_1 C_2}{|c_1 - \lambda|} \Big( 1 + (1 + t_0)^{c_1 - \lambda} \Big)  \Big ] (1 + t)^{- \min\{c_1, \lambda \}}.
\end{equation}
\noindent $\diamond$~Case 1.2 $(c_1 = \lambda)$:~It follows from definition of ${\mathcal J}(t)$ and $c_1=\lambda$ that 
\begin{equation} \label{B.19}
{\mathcal J}(t) =  \int_{t_0}^{t} (1 + s)^{c_1 - 1} g(s) \dd s \leq C_2  \int_{t_0}^{t} (1 + s)^{-1} \dd s = C_2 \Big( \ln (1 + t) - \ln (1 + t_0)  \Big).
\end{equation}
By \eqref{B.15} and \eqref{B.19}, we have
\begin{align}
\begin{aligned} \label{B.20}
y(t) &\leq y(t_0) \Big( \frac{1 + t_0}{1 + t} \Big)^{c_1}  + C_1C_2  (1 + t)^{-c_1} \Big( \ln (1 + t) - \ln (1 + t_0)  \Big) \\
&\leq (1 + t)^{-c_1} \ln (2 + t) \Big[  y(t_0)(1 + t_0)^{c_1} + C_1 C_2        \Big], \quad t \geq t_0.
\end{aligned}
\end{align}
By \eqref{B.18} and \eqref{B.20}, we have
\begin{equation} \label{B.21}
y(t)\leq {\tilde C}_1
\begin{cases}
(1+t)^{-\min\{c_1, \lambda \}}, \qquad  &c_1 \neq  \lambda,\\
(1+t)^{-\lambda} \ln(2+t), \qquad  &c_1= \lambda.
\end{cases}
\end{equation}
\noindent $\bullet$~Case 2: Suppose that the source $g$ satisfies 
\[ g(t)\leq C_2e^{-ct^\lambda}. \]
In this case, the integral ${\mathcal J}$ satisfies 
\begin{equation} \label{B.22}
{\mathcal J}(t) =  \int_{t_0}^{t} (1 + s)^{c_1 - 1} g(s) \dd s \leq C_2  \int_{t_0}^{t} (1 + s)^{c_1- 1} e^{-cs^\lambda}  \dd s \leq {\tilde C}_2, \quad t \geq t_0.
\end{equation}
By \eqref{B.15} and \eqref{B.22}, we have
\begin{equation} \label{B.23}
y(t) \leq \Big[ y(t_0)(1 + t_0)^{c_1} +  C_1 \tilde{C}_2  \Big] (1 + t)^{-c_1} \leq C_\eta(1+t)^{-\eta},  \quad t \geq t_0,
\end{equation}
for every $0<\eta< c_1$. Finally, we use \eqref{B.21} and \eqref{B.23} to derive the desired estimate. This completes the proof. 
\vspace{1cm}

\section{Finite-time stability estimate}\label{app-C}
\setcounter{equation}{0}
In this appendix, we provide a proof for Theorem~\ref{T3.2} \textnormal{(ii)}.  For $T \in (0, \infty)$, let $\mu$ and $\nu$ be two Lagrangian weak solutions on $[0,T]$ and let $\pi_0\in\Pi(\mu_0,\nu_0)$ be an arbitrary coupling of initial measures $\mu_0$ and $\nu_0$. We set the bi-characteristics:
\[
Z_\mu(t,z)=(X_\mu(t,z),V_\mu(t,z)),
\quad
Z_\nu(t,\bar z)=(X_\nu(t,\bar z),V_\nu(t,\bar z)),
\]
associated with \eqref{A-4} issued from points $z, \bar z \in \bbr^{2d}$. Then, we define the Lagrangian deviation functional:
\begin{equation*}\label{C.1}
	\Delta(t):=\int_{\bbr^{4d}} \bigl(|X_\mu(t,z)-X_\nu(t,\bar z)|
	+|V_\mu(t,z)-V_\nu(t,\bar z)|\bigr)\dd\pi_0(z,\bar z).
\end{equation*}
The transported coupling belongs to $\Pi(\mu_t,\nu_t)$, and therefore
\begin{equation}\label{C.2}
	\mathcal{W}_1(\mu_t,\nu_t)\leq\Delta(t),
	\quad
	\Delta(0)=\int_{\bbr^{4d}} |z-\bar z|\dd\pi_0(z,\bar z).
\end{equation}
All derivatives below can be justified by replacing $|\xi|$ with $(|\xi|^2+\varepsilon^2)^{1/2}$ and then sending $\varepsilon\downarrow0$. \newline

\noindent First, we record uniform moment bounds. We use \eqref{B-22} (more details see Proposition \ref{P4.1}) to obtain 
\[
h_{\mu_t}(Z_\mu(t,z))\leq h_{\mu_0}(z)+C_T,
\qquad
h_{\nu_t}(Z_\nu(t,\bar z))\leq h_{\nu_0}(\bar z)+C_T.
\]
Since $h_{\lambda_t}\geq |v|^2/2$, Young's inequality implies that for some $\alpha_T>0$,
\begin{equation}\label{C.3}
	\sup_{0\leq t\leq T}
	\left(
	\int_{\bbr^{2d}} e^{\alpha_T|V_\mu(t,z)|}\dd\mu_0(z)
	+\int_{\bbr^{2d}} e^{\alpha_T|V_\nu(t,\bar z)|}\dd\nu_0(\bar z)
	\right)\leq C_{T,M}.
\end{equation}
The energy inequality and quadratic confinement also give
\begin{equation}\label{C.4}
	\sup_{0\leq t\leq T}
	\left[
	\int_{\bbr^{2d}} (|X_\mu|^2+|V_\mu|^2)\dd\mu_0
	+\int_{\bbr^{2d}} (|X_\nu|^2+|V_\nu|^2)\dd\nu_0
	\right]\leq C_{T,M}.
\end{equation}
Let $(z_*,\bar z_*)$ be an independent copy with law $\pi_0$. For the alignment force, we insert the transported coupling and split
\begin{align*}
	&\phi(|X_\mu-X_{\mu,*}|)(V_{\mu,*}-V_\mu)
	-\phi(|X_\nu-X_{\nu,*}|)(V_{\nu,*}-V_\nu)\\
	&\quad=
	\phi(|X_\mu-X_{\mu,*}|)
	\bigl[(V_{\mu,*}-V_\mu)-(V_{\nu,*}-V_\nu)\bigr]\\
	&\qquad+
	\bigl[\phi(|X_\mu-X_{\mu,*}|)-\phi(|X_\nu-X_{\nu,*}|)\bigr]
	(V_{\nu,*}-V_\nu).
\end{align*}
The first line is bounded after integration by $C\Delta(t)$. The Lipschitz continuity of $\phi$ bounds the second by
\[
C\bigl(|X_\mu-X_\nu|+|X_{\mu,*}-X_{\nu,*}|\bigr)
\bigl(|V_\nu|+|V_{\nu,*}|\bigr).
\]
After integration and symmetry,
\begin{align}\label{C.5}
	\begin{aligned}
	I_{\cA}(t)&:=\int_{\bbr^{8d}}\Big|\phi(|X_\mu-X_{\mu,*}|)(V_{\mu,*}-V_\mu)-\phi(|X_\nu-X_{\nu,*}|)(V_{\nu,*}-V_\nu)\Big|\dd\pi_0(z,\bar z)\dd\pi_0(z_*,\bar z_*) \\
	&\leq C_{T,M}\Delta(t)
	+C\int_{\bbr^{2d}} |V_\nu(t,\bar z)|
	|X_\mu(t,z)-X_\nu(t,\bar z)|\dd\pi_0.
	\end{aligned}
\end{align}
 For the potential force, the independent-copy representation and the Hessian bound give directly
\begin{align}
	\begin{aligned}
	I_W(t)&:=\int_{\bbr^{8d}}\Big|\nabla W(X_\mu-X_{\mu,*})-\nabla W(X_\nu-X_{\nu,*})\Big|\dd\pi_0(z,\bar z)\dd\pi_0(z_*,\bar z_*)\\
	&\leq  \|D^2W \|_{\infty}\int_{\bbr^{8d}}
	\bigl(|X_\mu-X_\nu|+|X_{\mu,*}-X_{\nu,*}|\bigr)
	\dd\pi_0\dd\pi_0 
	\leq2 \|D^2W \|_{\infty}\Delta(t).
	\label{C.6}
		\end{aligned}
\end{align}
Combining the characteristic equations with \eqref{C.5}--\eqref{C.6}, we obtain
\begin{equation*}\label{C.7}
	\Delta'(t)
	\leq C_{T,M}\Delta(t)
	+C\int_{\bbr^{2d}} |V_\nu|\,|X_\mu-X_\nu|\dd\pi_0, \quad \mbox{for a.e. $t\in[0,T]$.}
\end{equation*}
Fix $R\geq1$. The part with $|V_\nu|\leq R$ is at most $R\Delta(t)$. On the complement, the Cauchy--Schwarz inequality, \eqref{C.3}, and \eqref{C.4} yield
\[
\int_{\{|V_\nu|>R\}}|V_\nu|\,|X_\mu-X_\nu|\dd\pi_0 \leq C_{T,M}
\left(\int_{\{|V_\nu|>R\}}|V_\nu|^2\dd\nu_0\right)^{1/2}
\leq C_{T,M}e^{-cR}.
\]
Consequently, after changing constants, we derive 
\begin{equation}\label{C.8}
	\Delta'(t)\leq C_{T,M} \Big( (1+R)\Delta(t)+ e^{-R} \Big),
	\quad R\geq1.
\end{equation}
We first choose $R=1$ to derive a rough uniform bound 
\[ \Delta(t)\leq K_{T,M}. \]
We set 
\[ a(t)=\Delta(t)/K_{T,M}. \]
If necessary, we enlarge $K_{T,M}$ so that 
\[ a(t)\leq e^{-1}. \]
Since \eqref{C.8} holds for every $R\geq1$, we choose $R=-\log a(t)$. Then, we have
\[
a'(t)\leq C_{T,M}a(t)(1-\log a(t))
\leq-2C_{T,M}a(t)\log a(t).
\]
Equivalently,
\[
\frac{\dd}{\dd t}\log(-\log a(t))\geq-2C_{T,M}.
\]
For $a(0)>0$, we have
\[
a(t)\leq a(0)^{e^{-2C_{T,M}t}}.
\]
For $a(0)=0$, we apply the same scalar comparison with initial value $a(0)+\delta$ and let $\delta\downarrow0$. Hence in all cases
\begin{equation}\label{C.9}
	\sup_{0\leq t\leq T}\Delta(t)
	\leq C_{T,M}\Delta(0)^{e^{-2C_{T,M}T}}=:G_T(\Delta(0)).
\end{equation}
For $\Delta(0)>1$ the same estimate follows from the rough bound after increasing the constant. Taking the infimum in \eqref{C.9} over all initial couplings and using \eqref{C.2} proves \eqref{C-14}.

\vspace{1cm}
\section{Proof of Proposition \ref{P5.1}} \label{app-D}
In this appendix, we provide a detailed proof of Proposition \ref{P5.1}.\newline 

\noindent (1)~Since $W_{\rm rs}$ is even, $\nabla W_{\rm rs}$ is odd. Moreover, since the communication weight depends only on the relative distances, the finite truncated system is invariant under the reflection:
\[(x_{n,+},v_{n,+})
\longleftrightarrow
(-x_{n,-},-v_{n,-}),
\quad
(x_0,v_0)
\longmapsto
(-x_0,-v_0).\]
We truncate after the $N$-th pair and place the omitted mass at the central particle. More precisely, we set
\[
m_0^N:=m_0+\sum_{n>N}m_n,\quad
m_{n,\pm}^N:=m_{n,\pm}\quad  n \in [N],
\]
and keep the central initial state 
\[ x_0^N(0)=0, \quad v_0^N(0)=0. \]
Then the total mass is one and the weighted center and momentum are zero. The finite system is globally well posed and has zero weighted center and momentum.
Together with the symmetric masses $(m_{n,+}=m_{n,-})$, the reflected trajectory solves the same finite system with the same initial data. By the uniqueness of the finite-dimensional ODE, the truncated solution satisfies 
\[x^N_{n,-}(t)=-x^N_{n,+}(t),\quad v^N_{n,-}(t)=-v^N_{n,+}(t),\quad x^N_0(t)=v^N_0(t)=0, \quad t \geq 0.\] 
The initial kinetic and potential energies are bounded uniformly in $N$ by the assumed weighted second moment and the upper quadratic bound on $W_{\rm rs}$. Its exact energy identity therefore gives an $N$-independent bound $E_*$. 

Let \(\mathcal I_N\) denote the index set of the \(N\)-th symmetric
truncation, and let
\[
\mathscr E_N(t)
:=
\sum_{i\in\mathcal I_N} m_i^N|v_i^N(t)|^2
+
\sum_{i,j\in\mathcal I_N}
m_i^Nm_j^N
W_{\rm rs}(x_i^N(t)-x_j^N(t)).
\]
The omitted mass is placed at the central particle, so that the total
mass remains equal to one. By the assumed weighted second moment and the upper quadratic bound
\[
W_{\rm rs}(z)\leq \bar{c}_W |z|^2,
\]
the initial energies are bounded uniformly in \(N\).  Therefore, we set
\begin{equation*}
	E_*:=\sup_{N\geq1}\mathscr E_N(0)<\infty.
\end{equation*}
The exact energy identity for the finite truncated system gives
\begin{align*}
	\mathscr E_N(t)
	&+
	\kappa\int_0^t
	\sum_{i,j\in\mathcal I_N}
	m_i^Nm_j^N
	\phi_\beta(|x_i^N-x_j^N|)
	|v_i^N-v_j^N|^2\dd s
	=
	\mathscr E_N(0)
	\leq E_*.
\end{align*}
In particular, we have
\[
\mathscr E_N(t)\leq E_*
\qquad
\text{for every }N\geq1\text{ and }t\geq0.
\] Since the total mass is one and the weighted center is zero, we have
\[
\sum_{i,j  \in \mathcal{I}_N}m_i^Nm_j^N|x_i^N-x_j^N|^2
=2\sum_{i \in  \mathcal{I}_N}  m_i^N|x_i^N|^2.
\]
Since
\[
\sum_{i,j \in  \mathcal{I}_N}m_i^Nm_j^NW_{\rm rs}(x_i^N-x_j^N)
\geq \underbar{c}_W \sum_{i,j \in \mathcal{I}_N }m_i^Nm_j^N|x_i^N-x_j^N|^2
=2 \underbar{c}_W \sum_{i\in  \mathcal{I}_N} m_i^N|x_i^N|^2,
\]
for every fixed index $i$,
\begin{equation*}\label{D.1}
	|x_i^N(t)|\leq\sqrt{\frac{E_*}{2 \underbar{c}_W m_i^N}},
	\qquad
	|v_i^N(t)|\leq\sqrt{\frac{E_*}{m_i^N}}.
\end{equation*}
Moreover, $\nabla W_{\rm rs}(0)=0$ and its Hessian is bounded, so
\[
|\nabla W_{\rm rs}(x_i^N-x_j^N)|\leq  \|D^2W_{\rm rs} \|_{\infty}(|x_i^N|+|x_j^N|).
\]
Together with 
\[ \sum_{j\in  \mathcal{I}_N}m_j^N|x_j^N|+\sum_{j\in  \mathcal{I}_N}m_j^N|v_j^N|\leq C(E_*), \]
this gives the equicontinuity of every fixed coordinate on compact time intervals.  Arzel\`a--Ascoli and a diagonal extraction yield coordinatewise local uniform convergence. \newline

\noindent (2)~It remains to pass to the infinite sums. We set
\[ M_K=\sum_{j>K}m_j. \]
For fixed $i$, one has 
\begin{align*}
	\begin{aligned}
		& \sum_{j>K}m_j^N\phi_\beta(|x_i^N-x_j^N|)|v_j^N-v_i^N| \leq\bar{c}_{\phi}\bigl(M_K|v_i^N|+\sqrt{M_K}\,\sqrt{E_*}\bigr),\\
		& \sum_{j>K}m_j^N|\nabla W_{\rm rs}(x_i^N-x_j^N)| \leq  \|D^2W_{\rm rs} \|_{\infty}\bigl(M_K|x_i^N|+C\sqrt{M_K}\,\sqrt{E_*}\bigr).
	\end{aligned}
\end{align*}
Both tails vanish uniformly in $N$ on compact time intervals.  We may therefore pass to the equations coordinatewise.  Fatou's lemma and finite second moments applied to the energy and the nonnegative integrated dissipation yield \eqref{E-16}.  Symmetry and the fixed central particle follow from uniqueness of every finite symmetric truncation and passage to the limit.
\vspace{1cm}
\section{Proof of Proposition \ref{P5.2}} \label{app-E}
\setcounter{equation}{0}
In this appendix, we provide details which are sketched in the proof of Proposition \ref{P5.2}. \newline
	
\noindent $\bullet$ \textbf{Step A} (Bootstrap interval and annular bounds): Recall that
	\[
	E_n(0)=\Omega^2R_n^2,
	\qquad
	L_n(0)=\Omega R_n^2.
	\]
Consider a temporal set ${\mathcal T}_n$ and its supremum $T_n=\sup_{\tau\in \mathcal{T}_n} \tau$:
\begin{align}\label{E.1}
{\mathcal T}_n:=\Bigg\{\tau>0 \quad \Bigg|\quad 	\frac12\Omega^2R_n^2
\leq E_n(t)
\leq2\Omega^2R_n^2,
\qquad
|L_n(t)|
\geq\frac12\Omega R_n^2, \quad \forall ~ t\in[0,\tau)\Bigg\}.
\end{align}
Since the initial inequalities are strict, by continuity of $E_n$ and $L_n$
	we have 
	\[  T_n>0. \]
	On \([0,T_n)\), the upper bound for \(E_n\) gives
	\[
	|y_n(t)|\leq2R_n,
	\qquad
	|z_n(t)|\leq2\Omega R_n.
	\]
	On the other hand,
	\[
	|L_n(t)|
	=|y_n(t)\wedge z_n(t)|
	\leq |y_n(t)|\,|z_n(t)|.
	\]
	Combining this inequality with the lower bound for \(|L_n|\) and the
	preceding upper bounds, we obtain
	\[
	|y_n(t)|
	\geq
	\frac{|L_n(t)|}{|z_n(t)|}
	\geq\frac14R_n \quad 
	\mbox{and} \quad 
	|z_n(t)|
	\geq
	\frac{|L_n(t)|}{|y_n(t)|}
	\geq\frac14\Omega R_n.
	\]
	Consequently, we have
	\begin{equation}\label{E.2}
		\frac14R_n\leq |y_n(t)|\leq2R_n,
		\quad
		\frac14\Omega R_n\leq |z_n(t)|\leq2\Omega R_n,
		\quad 0\leq t<T_n.
	\end{equation}
	
	\medskip
	\noindent
$\bullet$	\textbf{Step B} (Separation from the center and all lower modes):~For every \(k<n\) and \(\sigma\in\{+,-\}\), the global energy
	inequality gives
	\[
	2\underbar{c}_Wm_{k,\sigma}|x_{k,\sigma}(t)|^2
	\leq\mathscr E_\infty(t)
	\leq\overline E.
	\]
	Since \(m_{k,\sigma}=m_k/2\), we have
	\begin{equation*}\label{E.3}
		|x_{k,\sigma}(t)|
		\leq
		\sqrt{\frac{\overline E}{\underbar{c}_Wm_k}}.
	\end{equation*}
	By the choice of the radii gap in Lemma \ref{L5.2},
	\[
	\sqrt{\frac{\overline E}{\underbar{c}_Wm_k}}
	\leq\frac{R_n-8}{64}
	\leq\frac1{64}R_n.
	\]
	Thus, we use \eqref{E.2} to get 
	\[
	|y_n(t)-x_{k,\sigma}(t)|
	\geq
	|y_n(t)|-|x_{k,\sigma}(t)|
	\geq
	\left(\frac14-\frac1{64}\right)R_n
	=\frac{15}{64}R_n.
	\]
	The distances from \(y_n\) to the center and to its symmetric partner
	are similarly bounded below:
	\[
	|y_n(t)-x_0(t)|=|y_n(t)|\geq\frac14R_n, \quad 
	|y_n(t)-x_{n,-}(t)|
	=2|y_n(t)|
	\geq\frac12R_n.
	\]
	Since \(R_1\ge16\), all these distances are
	larger than \(2\). Note that
	\[
	\operatorname{supp}U\subset{\{z\in\mathbb R^2:1<|z|<2\}},
	\]
	the perturbation \(U\) produces no force on \(y_n\) from the center,
	the symmetric partner, or any lower mode.
	
	\medskip
	\noindent
$\bullet$	\textbf{Step C} (Equation for the \(n\)-th mode):~Recall that
\[
y_n(t):=x_{n,+}(t),
\qquad
z_n(t):=v_{n,+}(t).
\]
Since
\[
W_{\rm rs}(\xi)
=
\frac{\Omega^2}{2}|\xi|^2+U(\xi),
\]
we have
\[
\nabla W_{\rm rs}(\xi)
=
\Omega^2\xi+\nabla U(\xi).
\]
Let $\mathcal I$ denote the set of all indices. Therefore, the equation for the particle \((n,+)\) can be written as
\begin{align*}
	\begin{cases}
		\displaystyle 
		\dot y_n=z_n,
		\vspace{6pt}\\
		\displaystyle
		\dot z_n
		=
		\underbrace{\kappa\sum_{j\in\mathcal I}
			m_j\phi_\beta(|y_n-x_j|)(v_j-z_n)}_{\text{Alignment}}
		-
		\underbrace{\Omega^2\sum_{j\in\mathcal I}m_j(y_n-x_j)}_{\text{Harmonic}}
		-
		\underbrace{\sum_{j\in\mathcal I}m_j\nabla U(y_n-x_j)}_{\text{Perturbation}}.
		\end{cases}
\end{align*}
$\diamond$ Case 1: We first simplify the harmonic part. Since the total mass is normalized and the weighted center of mass is zero:
\[
\sum_{j\in\mathcal I}m_j=1, \quad \sum_{j\in\mathcal I}m_jx_j(t)=0,
\]
we obtain
\begin{align*}
	-\Omega^2\sum_{j\in\mathcal I}m_j(y_n-x_j)
	&=
	-\Omega^2
	\left(
	y_n\sum_{j\in\mathcal I}m_j
	-\sum_{j\in\mathcal I}m_jx_j
	\right) =
	-\Omega^2y_n.
\end{align*}
Thus, in the absence of alignment and of the compact perturbation
\(U\), the \(n\)-th mode would solve the harmonic oscillator equation
\[
\ddot y_n+\Omega^2y_n=0.
\]
$\diamond$ Case 2: To describe the remaining terms, introduce the index sets
\[
\mathcal I_n^{\rm low}
:=\{0,(n,-)\}
\cup\bigl\{(k,\sigma):1\leq k<n,\ \sigma\in\{+,-\}\bigr\},
\]
\[
\mathcal I_n^{\rm high}
:=\bigl\{(k,\sigma):k>n,\ \sigma\in\{+,-\}\bigr\}.
\]
The self-index \((n,+)\) is excluded from both sets. Its alignment
contribution vanishes because
\[
v_{n,+}-z_n=0,
\]
while its potential contribution vanishes because
\[
y_n-x_{n,+}=0
\qquad\text{and}\qquad
\nabla U(0)=0.
\]
We define the alignment contribution of the center, the symmetric
partner, and all lower modes by
\begin{equation*}
	A_n^{\rm low}(t)
	:=
	\kappa
	\sum_{j\in\mathcal I_n^{\rm low}}
	m_j\phi_\beta(|y_n-x_j|)(v_j-z_n),
\end{equation*}
and the alignment contribution of all higher modes by
\begin{equation*}
	A_n^{\rm high}(t)
	:=
	\kappa
	\sum_{j\in\mathcal I_n^{\rm high}}
	m_j\phi_\beta(|y_n-x_j|)(v_j-z_n).
\end{equation*}
$\diamond$ Case 3:  For the perturbation \(U\), define initially
\[
F_n^U(t)
:=
-\sum_{j\in\mathcal I}
m_j\nabla U(y_n-x_j).
\]
By Step B, the distances from \(y_n\) to the center, to its symmetric
partner, and to every lower mode are all larger than \(2\). Since
\[
\operatorname{supp}U
\subset
\{\xi\in\mathbb R^2:1<|\xi|<2\},
\]
we have
\[
\nabla U(y_n-x_j)=0,
\qquad
j\in\mathcal I_n^{\rm low}.
\]
The self-interaction also vanishes. Hence only particles belonging to
higher modes can possibly contribute, and therefore
\begin{equation*}
	F_n^U(t)
	=
	-\sum_{j\in\mathcal I_n^{\rm high}}
	m_j\nabla U(y_n-x_j)
	=
	-\sum_{k>n}\sum_{\sigma\in\{+,-\}}
	m_{k,\sigma}
	\nabla U(y_n-x_{k,\sigma}).
\end{equation*}
Combining the preceding identities, the \(n\)-th mode satisfies the
perturbed harmonic oscillator system
\begin{equation}\label{E.4}
	\dot y_n=z_n,
	\qquad
	\dot z_n=-\Omega^2y_n+A_n,
\end{equation}
where
\begin{equation*}\label{E.5}
	A_n
	:=
	A_n^{\rm low}
	+
	A_n^{\rm high}
	+
	F_n^U.
\end{equation*}
Thus \(A_n\) contains all effects that are not part of the exact
harmonic restoring force.
	
	\medskip
	\noindent
$\bullet$	\textbf{Step D} (Estimate of the non-harmonic acceleration):~We split the estimates into three cases. 

\noindent $\diamond$ Case 1:  For the center, the symmetric partner, and the lower modes, the
	preceding separation estimates in Step B imply
	\[
	|y_n-x_j|\geq c_0R_n
	\]
	for some \(c_0>0\) independent of \(n\). Hence, we use 
	\[
	\phi_\beta(r)\leq C_\beta r^{-\beta},
	\qquad r\geq1
	\]
 to obtain
	\begin{align*}
		|A_n^{\mathrm{low}}|
		&\leq
		C R_n^{-\beta}
		\sum_{j\in\mathcal I_n^{\rm low}}m_j|v_j-z_n|
		\leq
		C R_n^{-\beta}
		\left(
		|z_n|+\sum_jm_j|v_j|
		\right).
		\label{E.6}
	\end{align*}
	Since \(\sum_jm_j=1\), the Cauchy--Schwarz inequality and the energy inequality give
	\[
	\sum_jm_j|v_j|
	\leq
	\left(\sum_jm_j|v_j|^2\right)^{1/2}
	\leq\sqrt{\overline E}.
	\]
	Together with \(|z_n|\leq2\Omega R_n\), this yields
	\begin{equation}\label{E.7}
		|A_n^{\mathrm{low}}|
		\leq CR_n^{1-\beta}.
	\end{equation}
	\noindent $\diamond$ Case 2: 	For the higher modes we do not use geometric separation. Note that
	\begin{align*}
		|A_n^{\mathrm{high}}|
		\leq
		\kappa\bar{c}_{\phi}
		\sum_{k>n,\sigma}
		m_{k,\sigma}|v_{k,\sigma}-z_n|
		\leq
		C M_n^+|z_n|
		+
		C\sum_{k>n,\sigma}
		m_{k,\sigma}|v_{k,\sigma}|.
		\label{E.8}
	\end{align*}
	By the Cauchy--Schwarz inequality,
	\[
	\sum_{k>n,\sigma}
	m_{k,\sigma}|v_{k,\sigma}|
	\leq
	\sqrt{M_n^+}
	\left(
	\sum_{k>n,\sigma}
	m_{k,\sigma}|v_{k,\sigma}|^2
	\right)^{1/2}
	\leq
	\sqrt{\overline E\,M_n^+}.
	\]
We use 
	\[
	M_n^+\leq R_n^{-2\beta-4}
	\]
	in $\eqref{E-21}_3$ and \(|z_n|\leq2\Omega R_n\) to obtain
\begin{equation} \label{E.9}
		|A_n^{\mathrm{high}}|
		\leq
		C R_n^{-2\beta-3}
		+
		C R_n^{-\beta-2}
		\leq
		C R_n^{1-\beta}.
\end{equation}
\noindent $\diamond$ Case 3: 	Finally, consider the force generated by the compactly supported
perturbation \(U\). Since
\[
\operatorname{supp}U
\subset
\{z\in\mathbb R^2:1<|z|<2\},
\]
we have \(\nabla U(z)=0\) whenever \(|z|\notin(1,2)\).
Therefore only particles belonging to higher
modes can possibly contribute, and
\[
F_n^U=
-\sum_{k>n}\sum_{\sigma\in\{+,-\}}
m_{k,\sigma}
\nabla U(y_n-x_{k,\sigma}).
\]
Consequently, we have
\begin{align} \label{E.10}
	|F_n^U|
	\leq
	\|\nabla U\|_{L^\infty}
	\sum_{k>n}\sum_{\sigma\in\{+,-\}}m_{k,\sigma}
=
	\|\nabla U\|_{L^\infty}M_n^+ \leq
	C R_n^{-2\beta-4}.
\end{align}
Since \(R_n\geq1\) and \(\beta>0\), this stronger estimate also implies
\[
|F_n^U|\leq C R_n^{1-\beta}.
\]
	Combining
	\eqref{E.7},
	\eqref{E.9}, and
	\eqref{E.10}, we conclude that
	\begin{equation}\label{E.11}
		|A_n(t)|
		\leq
		C R_n^{1-\beta},
		\qquad 0\leq t<T_n.
	\end{equation}
	
	\medskip
	\noindent
$\bullet$	\textbf{Step E} (Variation of energy and angular momentum):~From \eqref{E.4},
\[
		\frac{\dd}{\dd t}E_n(t)
		=
		z_n\cdot\dot z_n
		+\Omega^2y_n\cdot\dot y_n
		=
		z_n\cdot(-\Omega^2y_n+A_n)
		+\Omega^2y_n\cdot z_n\ =z_n\cdot A_n.
\]
	Similarly, one has 
	\[
		\frac{\dd}{\dd t}L_n(t)
		=
		\dot y_n\wedge z_n+y_n\wedge\dot z_n 
		=
		z_n\wedge z_n
		+y_n\wedge(-\Omega^2y_n+A_n) 
		=y_n\wedge A_n.
	\]
	Using \eqref{E.2} and
	\eqref{E.11}, we obtain
	\begin{equation}\label{E.12}
		|E_n'(t)|+\Omega|L_n'(t)|
		\leq
		C R_n^{2-\beta},
		\qquad 0\leq t<T_n.
	\end{equation}
We also set
	\[
	\tau_n:=\min\{T_n,~c_*R_n^\beta\}.
	\]
	For every \(0\leq t<\tau_n\), we integrate
	\eqref{E.12} to find 
	\begin{equation}\label{E.13}
		|E_n(t)-E_n(0)|
		+
		\Omega|L_n(t)-L_n(0)|
		\leq
		Cc_*R_n^2.
	\end{equation}
	We choose \(c_*>0\), independently of \(n\), so small that
	\begin{equation*}\label{E.14}
		Cc_*\leq\frac14\Omega^2.
	\end{equation*}
	Since
	\[
	E_n(0)=\Omega^2R_n^2,
	\qquad
	\Omega L_n(0)=\Omega^2R_n^2,
	\]
	it follows from \eqref{E.13} that
	\[
	\frac34\Omega^2R_n^2
	\leq E_n(t)
	\leq\frac54\Omega^2R_n^2, \quad 
	|L_n(t)|
	\geq\frac34\Omega R_n^2,
	\qquad 0\leq t<\tau_n.
	\]
	These inequalities strictly improve the bootstrap assumptions \eqref{E.1}. Hence \(T_n\) cannot be smaller than
	\(c_*R_n^\beta\). Therefore
	$
	T_n\geq c_*R_n^\beta,$
	and by continuity \eqref{E-26} follows.

\end{document}